\documentclass[oneside,english,reqno]{amsart}
\usepackage[T1]{fontenc}
\usepackage[utf8]{inputenc}
\usepackage{geometry}
\usepackage{color}
\definecolor{note_fontcolor}{rgb}{0.80078125, 0.80078125, 0.80078125}
\usepackage{amsthm}
\usepackage{amstext}
\usepackage{amssymb}
\usepackage{esint}

\makeatletter

\numberwithin{equation}{section}
\numberwithin{figure}{section}
\theoremstyle{plain}
\newtheorem{thm}{\protect\theoremname}
  \theoremstyle{plain}
  \newtheorem{lem}[thm]{\protect\lemmaname}
  \theoremstyle{remark}
  \newtheorem*{rem*}{\protect\remarkname}
  \theoremstyle{remark}
  \newtheorem{rem}[thm]{\protect\remarkname}
  \theoremstyle{plain}
  \newtheorem{cor}[thm]{\protect\corollaryname}
  \theoremstyle{definition}
  \newtheorem{defn}[thm]{\protect\definitionname}
  \theoremstyle{definition}
  \newtheorem*{example*}{\protect\examplename}

\newcommand{\with}{\:\middle|\:}

\let\emptyset\varnothing

\DeclareMathOperator*{\esssup}{ess\,sup}

\usepackage{babel}
\providecommand{\corollaryname}{Corollary}
  \providecommand{\definitionname}{Definition}
  \providecommand{\examplename}{Example}
  \providecommand{\lemmaname}{Lemma}
  \providecommand{\remarkname}{Remark}
\providecommand{\theoremname}{Theorem}

\makeatother

\usepackage{babel}
  \providecommand{\corollaryname}{Corollary}
  \providecommand{\definitionname}{Definition}
  \providecommand{\examplename}{Example}
  \providecommand{\lemmaname}{Lemma}
  \providecommand{\remarkname}{Remark}
\providecommand{\theoremname}{Theorem}

\begin{document}

\title{Wavelet Coorbit Spaces viewed as Decomposition Spaces}

\author{Hartmut Führ, Felix Voigtlaender}
\begin{abstract}
In this paper we show that the Fourier transform induces an isomorphism
between the coorbit spaces defined by Feichtinger and Gröchenig 
of the mixed, weighted Lebesgue spaces $L_{v}^{p,q}$ with respect
to the quasi-regular representation of a semi-direct product $\mathbb{R}^{d}\rtimes H$
with suitably chosen dilation group $H$, and certain decomposition
spaces $\mathcal{D}\left(\mathcal{Q},L^{p},\ell_{u}^{q}\right)$ (essentially
as introduced by Feichtinger and Gröbner) 
where the localized ,,parts`` of a function are measured in the $\mathcal{F}L^{p}$-norm.

This equivalence is useful in several ways: It provides access
to a Fourier-analytic understanding of wavelet coorbit spaces, and it allows
to discuss coorbit spaces associated to different dilation groups
in a common framework. As an illustration of these points, we include
a short discussion of dilation invariance properties of coorbit spaces
associated to different types of dilation groups. 
\end{abstract}
\maketitle

\section{Introduction}

\global\long\def\vertiii#1{{\left\vert \kern-0.25ex  \left\vert \kern-0.25ex  \left\vert #1\right\vert \kern-0.25ex  \right\vert \kern-0.25ex  \right\vert }}

There exist several methods in the literature for the construction
of higher-dimensional wavelet systems. A rather general class of constructions
follows the initial inception of the continuous wavelet transform
in \cite{GrossmannMorletPaul} and uses the language of group representations \cite{Murenzi,BernierTaylor,FuehrWaveletFramesAndAdmissibility,Laugesen_etal}:
Picking a suitable matrix group $H\leq{\rm GL}\left(\mathbb{R}^{d}\right)$,
the so-called \textbf{dilation group}, one defines the associated
semidirect product $G=\mathbb{R}^{d}\rtimes H$. This group acts on
${\rm L}^{2}(\mathbb{R}^{d})$ via the (unitary) \textbf{quasi-regular
representation} $\pi$ defined by 
\[
\left(\pi(x,h)f\right)(y)=|{\rm det}(h)|^{-1/2}f(h^{-1}(y-x))\quad,(x,h)\in\mathbb{R}^{d}\times H.
\]
The associated continuous wavelet transform of a signal $f\in{\rm L}^{2}(\mathbb{R}^{d})$
is then obtained by picking a suitable mother wavelet $\psi\in{\rm L}^{2}(\mathbb{R}^{d})$,
and letting 
\begin{equation}
W_{\psi}f:G\to\mathbb{C}~,~(x,h)\mapsto\langle f,\pi(x,h)\psi\rangle~.\label{eqn:wavelet_inv}
\end{equation}
A wavelet $\psi$ is called \textbf{admissible} if the operator $W_{\psi}$
is (a multiple of) an isometry as a map into ${\rm L}^{2}(G,\mu_{G})$,
where $\mu_{G}$ denotes a left Haar measure on $G$. By definition
we thus have for admissible vectors $\psi$ that 
\[
\forall f\in{\rm L}^{2}(\mathbb{R}^{d})~:~\|f\|_{2}^{2}=\frac{1}{C_{\psi}}\cdot\int_{H}\int_{\mathbb{R}^{d}}\left|W_{\psi}f(x,h)\right|^{2}{\rm d}x\,\frac{{\rm d}h}{|{\rm det}(h)|}~,
\]
alternatively expressed in the weak-sense inversion formula 
\[
f=\frac{1}{C_{\psi}}\cdot\int_{H}\int_{\mathbb{R}^{d}}W_{\psi}f(x,h)\cdot\pi(x,h)\psi~{\rm d}x~\frac{{\rm d}h}{\left|\det\left(h\right)\right|}~.
\]

An alternative, with somewhat less structure but higher design flexibility,
is the semi-discrete approach described as follows: Pick a discretely
labelled quadratic partition of unity $(\widehat{\psi}_{i})_{i\in I}$ in frequency
domain, i.e. a family of functions satisfying 
\begin{equation}
\forall_{a.e.}\xi\in\mathbb{R}^{d}~:~\sum_{i\in I}\left|\widehat{\psi_{i}}(\xi)\right|^{2}=1\label{eqn:part_unity}
\end{equation}
and consider the system of all translates of the inverse Fourier transforms
$\psi_{i}=\mathcal{F}^{-1}(\widehat{\psi}_{i})$. This system is a
(continuously labelled) tight frame, i.e. 
\[
\forall f\in{\rm L}^{2}(\mathbb{R}^{d})~:~\|f\|_{2}^{2}=\sum_{i\in I}\int_{\mathbb{R}^{d}}\left|\langle f,L_{x}\psi_{i}\rangle\right|^{2}{\rm d}x,
\]
where $L_{x}$ denotes translation by $x\in\mathbb{R}^{d}$. This
norm equality can also be expressed in the weak-sense inversion formula
\[
f=\sum_{i\in I}f\ast\psi_{i}^{*}\ast\psi_{i}~,
\]
with $\psi_{i}^{*}(x)=\overline{\psi_{i}(-x)}$. For compactly supported
$\widehat{\psi}_{i}$, the translation variable can be discretized
as well, yielding a tight frame, and an associated unconditionally
converging frame expansion for all $f\in{\rm L}^{2}(\mathbb{R}^{d})$.

First and second generation curvelets \cite{StCaDo,CaDo1} are special
examples of this type of generalized wavelets, as well as discrete shearlet
systems (see \cite[Chapter 1]{shearlet_book} for an overview). In
all these constructions, the desired degree of isotropy, directional
selectivity, etc. in the generalized wavelet system is achieved by
suitably prescribing the supports of the functions $\widehat{\psi}_{i}$.

The similarity between the two approaches is best realized by noticing
that the admissible functions in the sense of the group-theoretic
wavelet transforms are characterized by the condition 
\[
    \int_{H}|\widehat{\psi}(h^{T}\xi)|^{2}\,{\rm d}h=C_{\psi}
\]
for almost all $\xi\in\mathbb{R}^{d}$, showing that the wavelet inversion
formula associated to the continuous wavelet transform is also closely
related to a quadratic partition of unity on the Fourier transform side, this
time indexed by the dilation group $H$.

For applications of these transforms, mathematical or otherwise, it
is important to realize that each class of generalized wavelet transforms
comes with a natural scale of related smoothness spaces, which are
defined by norms measuring wavelet coefficient decay. In the group-related
case, these are the so-called {\em coorbit spaces} introduced by
Feichtinger and Gröchenig \cite{FeichtingerCoorbit0,FeichtingerCoorbit1,FeichtingerCoorbit2}.
In the semi-discrete case, it has been realized recently that the
{\em decomposition spaces} and their associated norms, as introduced
by Feichtinger and Gröbner \cite{DecompositionSpaces1,DecompositionSpaces2},
provide a similarly convenient framework for the treatment of approximation-theoretic
properties of anisotropic (mostly shearlet-like) wavelet systems,
see e.g. \cite{BorupNielsenDecomposition,Labate_et_al_Shearlet}.

For a long time, the prime examples of coorbit theory were provided
by the modulation spaces, arising as coorbit spaces associated to
the Schrödinger representation of the Heisenberg group, and the Besov
spaces, which are coorbit spaces associated to the quasi-regular representation
of the $ax+b$ group (and their isotropic counterparts in higher dimensions).
More recently, the introduction of shearlets (at least the group-theoretic
version) triggered the systematic study of the associated coorbit
spaces \cite{Dahlke_etal_sh_coorbit1,Dahlke_etal_sh_coorbit2}; coorbit
spaces over the Blaschke group and their connection to complex analysis
are discussed in \cite{FePa}. The recent papers \cite{FuehrCoorbit1,FuehrCoorbit2}
pointed out that the study of wavelet coorbit spaces could be considerably
extended to cover a multitude of group-theoretically defined wavelet
systems in a unified approach that allows to prove the existence
of easily constructed, nice wavelet systems and atomic decompositions
in a large variety of settings.

However, with the introduction of ever larger classes of function
spaces comes the necessity of developing conceptual tools helping
to understand these spaces and the relationships between them. It
is the chief aim of this paper to provide a bridge between the two
types of generalized wavelet systems, by clarifying how wavelet coorbit
spaces arising from a group action can be understood as decomposition
spaces. There are several motives for this question. The first one
is provided by pre-existing results in the literature pointing in
this direction: In \cite[Section 7.2]{FeichtingerCoorbit0} it was
shown that (homogenous) Besov spaces arise as certain coorbit spaces
of weighted, (mixed) Lebesgue spaces with respect to the quasi-regular
representation of the $ax+b$ group. On the other hand, these spaces
can be defined by localizing the Fourier transform of $f$ using a
dyadic partition of the frequency space $\mathbb{R}^{d}\setminus\left\{ 0\right\} $
and summing the $L^{p}$-norms of the localized ``pieces'' in a
certain weighted $\ell^{q}$-space (cf. \cite[Definition 6.5.1]{GrafakosClassicalAndModern}).

In this paper we will show that this phenomenon is no coincidence,
but merely a manifestation of the general principle that every coorbit
space of a (suitably) weighted mixed Lebesgue space with respect to
the quasi-regular representation of the semidirect product $\mathbb{R}^{d}\rtimes H$
(with a closed subgroup $H\leq\text{GL}\left(\mathbb{R}^{d}\right)$)
arises as (the inverse image under the Fourier transform of) a certain
decomposition space. This means that membership of $f$ in the coorbit
space can be decided by localizing the Fourier transform $\widehat{f}$
with respect to a certain covering (called the covering induced by
$H$) of the \textbf{dual orbit} $\mathcal{O}=H^{T}\xi_{0}$ and summing
the $L^{p}$-norms of the individual pieces in a suitable weighted
$\ell^{q}$-space.

Thus, wavelet coorbit theory becomes a branch of decomposition space
theory. To some extent this was to be expected, because the
structures underlying decomposition spaces -- i.e., certain coverings
of (subsets of) $\mathbb{R}^{d}$ and subordinate partitions of unity
-- are much more flexible than the group structure of the
dilation group associated to coorbit spaces. In some sense, passing from the dilation group and
its associated scale of coorbit spaces to a suitable covering and
its associated scale of decomposition spaces amounts to a loss of
structure, as the group is replaced by a suitably chosen index set
of a discrete covering. This passage is important from a technical
point of view, because by (largely) discarding the dilation group,
we become free to discuss coorbit spaces associated to {\em different}
dilation groups in a common framework. This observation provides a
second reason for studying the connection to decomposition spaces.

Possibly the most fundamental motivation for studying this connection
is that it allows to discuss the approximation-theoretic properties
of a wavelet system in terms of the frequency content of the different
wavelets. To elaborate on this point, let us recall the well-understood
case of wavelet ONB's in dimension one: The typical vanishing moment
and smoothness conditions on the wavelets can be understood as a measure
of frequency concentration. Conceptually speaking, different scales
of the wavelet system correspond to different frequency bands, and
increasing the degrees of smoothness and vanishing moments amounts
to improving the separation between the different frequency bands,
which in turn allows larger classes of homogeneous Besov spaces to
be characterized in terms of the wavelet coefficients with respect to a single wavelet ONB. The papers
\cite{FuehrCoorbit1,FuehrCoorbit2} extend this type of reasoning
to (possibly anisotropic) higher dimensional wavelet systems and their
associated coorbit spaces; here the key concept was provided by the
dual action and in particular the so-called ``blind spot'' of the
wavelet transform.

However, in the study of wavelet systems in higher dimensions, the
description of frequency content poses an increasingly difficult challenge: Different
wavelet systems can be understood as prescribing different ways of
partitioning the frequency space into (possibly oriented) ``frequency bands''; we
argue that their approximation-theoretic properties are describable
in terms of this behaviour. It is important to note that precisely
this intuition was also used in the inception of curvelets by Candés
and Donoho \cite{CaDo1}, and the results of that paper provide further
evidence that the frequency partition determines the approximation-theoretic properties. 
However, what is needed to systematically turn this intuition into provable
theorems is a suitable language describing these partitions, 
and allowing to assess which properties of a partition
are relevant for the approximation-theoretic properties of the corresponding wavelet systems.
Our paper makes a strong case that
this language is provided by the decomposition spaces introduced by
Feichtinger and Gröbner in \cite{DecompositionSpaces1}, and studied
more recently in \cite{BorupNielsenDecomposition,Labate_et_al_Shearlet}.

To illustrate these points, we have included a discussion of dilation
invariance properties of certain coorbit spaces in Section \ref{sec:KonjugationsInvarianz}.
Given a suitable coorbit space ${\rm Co}Y$ associated to a dilation
group $H$, we would like to identify those invertible matrices $g$
such that ${\rm Co}Y$ is invariant under dilation by $g$. It is
clear that the set of these matrices contains $H$; this follows
from the fact that the wavelet transform intertwines (a suitably normalized) dilation by $g\in H$
with left translation by $(0,g)\in G$, and from left invariance of
the Banach function space $Y$ entering the definition of the coorbit
space. It is much less clear whether there are further invertible
matrices $g\not\in H$ which leave ${\rm Co}Y$ invariant. It will
be seen in Section \ref{sec:KonjugationsInvarianz} that this property
depends on $H$: If $H$ is the similitude group in dimension
two and the associated coorbit spaces are the isotropic Besov spaces,
they are in fact invariant under arbitrary dilations. By contrast,
there are shearlet coorbit spaces that are not invariant under dilation
by a ninety degree rotation.

%

While these observations are of some independent interest (for example,
the lack of rotation invariance for shearlet coorbit spaces seems
to be a new observation), we have included this discussion mostly
because of the way it highlights the role of the decomposition space
viewpoint in understanding the different coorbit spaces. It also illustrates
the importance of being able to compare coorbit spaces associated
to different dilation groups: One quickly realizes that dilation by
$g$ is an isomorphism ${\rm Co}Y\to{\rm Co}Y'$, where $Y'$ is a
suitably chosen Banach function space over the group $G'=\mathbb{R}^{d}\rtimes g^{-1}Hg$.
Thus invariance of ${\rm Co}Y$ under dilation by $g$ is equivalent
to an embedding statement ${\rm Co}Y'\hookrightarrow{\rm Co}Y$.

\section{Notation and Preliminaries}

In this paper we will always be working in the following setting:
We assume that $H\leq\text{GL}\left(\mathbb{R}^{d}\right)$ is a closed
subgroup for some $d\in\mathbb{N}$ and we consider the semidirect
product $G=\mathbb{R}^{d}\rtimes H$ with multiplication $\left(x,h\right)\left(y,g\right)=\left(x+hy,hg\right)$.
For the convenience of the reader we recall that a (left) \textbf{Haar
integral} on the locally compact group $G$ is then given by 
\begin{equation}
f\mapsto\int_{G}f\left(x,h\right)\,\text{d}\left(x,h\right):=\int_{H}\int_{\mathbb{R}^{d}}f\left(x,h\right)\,\text{d}x\, \frac{\text{d}h}{\left|\det\left(h\right)\right|},\label{eq:SemiDirektHaarMass}
\end{equation}
where $\text{d}h$ denotes integration against left Haar measure on
$H$. The \textbf{modular function} on $G$ is given by 
\begin{equation}
\Delta_{G}\left(x,h\right)=\Delta_{H}\left(h\right)\cdot\left|\det\left(h\right)\right|^{-1}.\label{eq:SemiDirektModularFunktion}
\end{equation}
We then consider the so-called \textbf{quasi-regular representation}
$\pi$ of $G$ acting unitarily on $L^{2}\left(\mathbb{R}^{d}\right)$
by 
\begin{equation}
\pi\left(x,h\right)f=L_{x}\Delta_{h}f=\left|\det\left(h\right)\right|^{-1/2}\cdot L_{x}D_{h^{-T}}f,\label{eq:QuasiRegulaereDarstellung}
\end{equation}
where we use the operators $L_{x},\Delta_{h},D_{h}$ (and later on
also) $M_{\omega}$ defined by
\begin{align*}
\left(L_{x}f\right)\left(y\right) & =f\left(y-x\right),\\
\left(\Delta_{h}f\right)\left(y\right) & =\left|\det\left(h\right)\right|^{-1/2}\cdot f\left(h^{-1}y\right),\\
\left(D_{h}f\right)\left(y\right) & =f\left(h^{T}y\right),\\
\left(M_{\omega}f\right)\left(y\right) & =e^{2\pi i\left\langle y,\omega\right\rangle }\cdot f\left(y\right)
\end{align*}
for $h\in\text{GL}\left(\mathbb{R}^{d}\right)$, $x,y,\omega\in\mathbb{R}^{d}$
and $f:\mathbb{R}^{d}\rightarrow\mathbb{C}$.

We use the following version of the Fourier transform: 
\[
\left(\mathcal{F}f\right)\left(\xi\right)=\widehat{f}\left(\xi\right)=\int_{\mathbb{R}^{d}}f\left(x\right)\cdot e^{-2\pi i\left\langle x,\xi\right\rangle }\,\text{d}x\qquad\text{for }\xi\in\mathbb{R}^{d}~,
\]
and consequently 
\[
\left(\mathcal{F}^{-1}f\right)\left(x\right)=\int_{\mathbb{R}^{d}}f\left(\xi\right)\cdot e^{2\pi i\left\langle x,\xi\right\rangle }\,\text{d}\xi\qquad\text{for }x\in\mathbb{R}^{d}
\]
for the (inverse) Fourier transform of $f\in L^{1}\left(\mathbb{R}^{d}\right)$.

Using this convention, we note that on the Fourier side the quasi-regular
representation is given by 
\begin{eqnarray}
\mathcal{F}\left(\pi\left(x,h\right)f\right) & = & \left|\det\left(h\right)\right|^{-1/2}\cdot\mathcal{F}\left(L_{x}\left(f\circ h^{-1}\right)\right)\nonumber \\
 & = & \left|\det\left(h\right)\right|^{-1/2}\cdot M_{-x}\left(\mathcal{F}\left(f\circ h^{-1}\right)\right)\nonumber \\
 & = & \left|\det\left(h\right)\right|^{1/2}\cdot M_{-x}D_{h}\widehat{f}.\label{eq:QausiRegulaereDarstellungAufFourierSeite}
\end{eqnarray}

The results in \cite{FuehrGeneralizedCalderonConditions,FuehrWaveletFramesAndAdmissibility}
show that the quasi-regular representation is \textbf{irreducible}
and \textbf{square-integrable} (in short: \textbf{admissible}), if
and only if the following conditions hold: 
\begin{enumerate}
\item There is a $\xi_{0}\in\mathbb{R}^{d}$ such that the \textbf{dual
    orbit} $\mathcal{O}:= H^{T} \xi_{0} = \left\{ h^{T}\xi_{0}\with h\in H\right\} \subset\mathbb{R}^{d}$
is an open set of full measure (i.e. $\lambda\left(\mathcal{O}^{c}\right)=0$,
where $\lambda$ denotes Lebesgue measure on $\mathbb{R}^{d}$) and 
\item the isotropy group $H_{\xi_{0}}:=\left\{ h\in H\with h^{T}\xi_{0}=\xi_{0}\right\} $
of $\xi_{0}$ with respect to the dual action of $H$ is compact.
In this case, the isotropy group $H_{h^{T}\xi_{0}}=h^{-1}H_{\xi_{0}}h$
is a compact subgroup of $H$ for every $h^{T}\xi_{0}\in\mathcal{O}$. 
\end{enumerate}
In the following, we will always assume that these conditions are
met. We will then see below (cf. Theorem \ref{thm:ZulaessigkeitVonbandbeschraenktenFunktionen})
that $\pi$ is indeed an integrable representation, i.e. there exists
$\psi\in L^{2}\left(\mathbb{R}^{d}\right)\setminus\left\{ 0\right\} $
with $W_{\psi}\psi\in L^{1}\left(G\right)$, where the \textbf{Wavelet
transform} $W_{g}f$ of $f$ with respect to $g$ is defined by 
\[
W_{g}f:G\rightarrow\mathbb{C},\left(x,h\right)\mapsto\left\langle f,\pi\left(x,h\right)g\right\rangle _{L^{2}\left(\mathbb{R}^{d}\right)}
\]
for $f,g\in L^{2}\left(\mathbb{R}^{d}\right)$. It should be noted
that this definition of the Wavelet transform coincides with the voice
transform $V_{g}f$ as defined in \cite{FeichtingerCoorbit1}, because
\[
\left(V_{g}f\right)\left(x,h\right)=\left\langle \pi\left(x,h\right)g,f\right\rangle _{{\rm anti}}=\left\langle f,\pi\left(x,h\right)g\right\rangle _{L^{2}\left(\mathbb{R}^{d}\right)},
\]
as Feichtinger uses a scalar-product that is antilinear in the first
component, i.e. 
\[
\left\langle f,g\right\rangle _{{\rm anti}}=\int_{\mathbb{R}^{d}}\overline{f\left(x\right)}\cdot g\left(x\right)\,\text{d}x,
\]
whereas we adopt the convention that the scalar-product $\left\langle \cdot,\cdot\right\rangle _{L^{2}\left(\mathbb{R}^{d}\right)}$
is antilinear in the \emph{second} component.

The paper is organized as follows: In section \ref{sec:ApplicabilityOfCoorbitTheory}
we clarify the exact definitions of the mixed Lebesgue spaces $L_{v}^{p,q}\left(G\right)$
and the requirements on the weight $v:G\rightarrow\left(0,\infty\right)$
for which we will later prove the isomorphism of $\text{Co}\left(L_{v}^{p,q}\right)$
to a suitable decomposition space. Furthermore, we show that the coorbit
theory is indeed applicable in this setting. The only point for
the applicability of coorbit theory that we do not cover in section
\ref{sec:ApplicabilityOfCoorbitTheory} is the existence of analyzing
vectors.

This gap is closed in the ensuing section \ref{sec:AdmissibleVectors}
in which we recall the most important definitions from coorbit theory
and show that any Schwartz function $\psi$ whose Fourier transform
is compactly supported in the dual orbit $\mathcal{O}$ is admissible
as a so-called \textbf{analyzing vector} (even as a ``\textbf{better
vector}''). Furthermore, we show that the ``reservoir'' $\left(\mathcal{H}_{w}^{1}\right)^{\neg}$
from which the elements of the coorbit space are taken can naturally
be identified with a subspace of the space of distributions $\mathcal{D}'\left(\mathcal{O}\right)$
on the dual orbit $\mathcal{O}=H^{T}\xi_{0}$ determined by $H$, as well as with a subspace of $\left( \mathcal{F} \left( \mathcal{D}\left( \mathcal{O} \right) \right) \right)'$,
where we use the notation $\mathcal{D}(U) = C_{c}^{\infty}(U)$ for the space of smooth functions with compact support in the open set $U \subset \mathbb{R}^{d}$.
This notation will also be used in the remainder of the paper.

In section \ref{sec:InducedCovering} we define the concept of an
``\textbf{induced covering}'' $\mathcal{Q}$ of the dual orbit $\mathcal{O}=H^{T}\xi_{0}$
and we give a precise definition of the decomposition spaces $\mathcal{D}\left(\mathcal{Q},L^{p},\ell_{u}^{q}\right)$
for which we will later show that the Fourier transform induces an
isomorphism $\mathcal{F}:\text{Co}\left(L_{v}^{p,q}\right)\rightarrow\mathcal{D}\left(\mathcal{Q},L^{p},\ell_{u}^{q}\right)$.
Furthermore, we recall the essential definitions for the theory of
decomposition spaces (i.e. the concepts of \textbf{admissible coverings},
\textbf{BAPU}s, etc.) and show that the induced covering $\mathcal{Q}$
is a structured admissible covering of $\mathcal{O}$ (cf. Definition
\ref{def:BorupUndNielsenCovering}).

In section \ref{sec:BAPU} we construct a specific partition of unity
subordinate to the induced covering $\mathcal{Q}$. In principle,
one could use any partition of unity $\left(\varphi_{i}\right)_{i\in I}$ subordinate to $\mathcal{Q}$
for which $\left\Vert \mathcal{F}^{-1}\varphi_{i}\right\Vert _{L^{1}\left(\mathbb{R}^{d}\right)}$
is uniformly bounded, but our construction has the advantage that
the localizations $\mathcal{F}^{-1}\left(\smash{\varphi_{i}\cdot\widehat{f}}\right)$
can be explicitly expressed in terms of the Wavelet transform $W_{\psi}f$
of $f$.

This explicit formula will be prominently exploited in section \ref{sec:CoorbitAsDecomposition}
where we prove that the Fourier transform extends to a bounded linear
map $\mathcal{F}:\text{Co}\left(L_{v}^{p,q}\right)\rightarrow\mathcal{D}\left(\mathcal{Q},L^{p},\ell_{u}^{q}\right)$
for a suitable weight $u:\mathcal{O}\rightarrow\left(0,\infty\right)$.

In section \ref{sec:InverseFourierTrafoStetig} we show that the inverse
Fourier transform $\mathcal{F}^{-1}:\mathcal{D}\left(\mathcal{Q},L^{p},\ell_{u}^{q}\right)\rightarrow\text{Co}\left(L_{v}^{p,q}\right)$
is continuous. In this section we also show that instead of the reservoir
$\left(\mathcal{H}_{w}^{1}\right)^{\neg}$ for the elements of the
coorbit space, one can use the more invariant reservoir $\left(\mathcal{F}\left(\mathcal{D}\left(\mathcal{O}\right)\right)\right)'$.

In section \ref{sec:KonjugationsInvarianz} we apply the established
isomorphism between the coorbit space $\text{Co}\left(L_{v}^{p,q}\right)$
and the decomposition space $\mathcal{D}\left(\mathcal{Q},L^{p},\ell_{u}^{q}\right)$
to investigate the invariance of certain specific coorbit spaces under
conjugation of the group. Using the decomposition space view, we show
that all coorbit spaces $\text{Co}\left(L_{v}^{p,q}\right)$ with
respect to the similitude group are invariant under conjugation, whereas
the same is in general not true for the coorbit spaces with respect
to the shearlet group. %

We close the technical preliminaries by noting that while the most
important results of the present paper are Theorems \ref{thm:FourierStetigVonCoorbitNachDecomposition}
and \ref{thm:InverseFourierTransformationStetigkeit}, which establish
the continuity of the (inverse) Fourier transform as a map from the
coorbit space into the associated decomposition space (and vice versa),
the most important parts of the paper in terms of ideas for the proof
are the definition of the specific partition of unity (cf. equation
(\ref{eq:BAPUIdee})) and the calculation of the localization $\mathcal{F}^{-1}\left(\smash{\varphi_{V}\widehat{f}}\right)$
in terms of the wavelet transform $W_{\psi}f$ (cf. Lemma \ref{lem:LokalisierungenAusgedruecktDurchWaveletTransformation}),
as well as Lemma \ref{lem:SpezialWaveletTrafoLiegtInPassendemRaum},
where we show $\left\Vert W_{\psi}\left(f \circ \mathcal{F}^{-1}\right)\right\Vert _{L_{v}^{p,q}}\leq C\cdot\left\Vert f\right\Vert _{\mathcal{D}\left(\mathcal{Q},L^{p},\ell_{u}^{q}\right)}$
for $f\in\mathcal{D}\left(\mathcal{Q},L^{p},\ell_{u}^{q}\right)$.

\section{Applicability of Coorbit-theory for the spaces $L_{v}^{p,q}$}

\label{sec:ApplicabilityOfCoorbitTheory}As the quasi-regular representation
is irreducible and square-integrable by our standing assumptions,
the main requirement for the coorbit-theory as developed by Feichtinger
and Gröchenig in \cite{FeichtingerCoorbit1,FeichtingerCoorbit2} is
fulfilled. In the remainder of this paper, we will routinely abuse
notation by identifying weights $v:H\to(0,\infty)$ with their trivial
extension $G\to(0,\infty)$, $(x,h)\mapsto v(h)$. We will exclusively
consider the coorbit spaces $\text{Co}\left(L_{v}^{p,q}\right)$ where
$v$ only depends on the second factor.

Nevertheless, we will sometimes have occasion to consider the Banach
function space $L_{w}^{p,q}\left(G\right)$, where $w:G\rightarrow\left(0,\infty\right)$
is arbitrary measurable (but we will not consider the coorbit space ${\rm Co}\left(L_{w}^{p,q}\right)$
in this case). This space is defined by 
\[
L_{w}^{p,q}\left(G\right):=\left\{ f:G\rightarrow\mathbb{C}\with f\text{ measurable and }\left\Vert f\right\Vert _{L_{w}^{p,q}}<\infty\right\} ,
\]
with 
\[
\left\Vert f\right\Vert _{L_{w}^{p,q}}:=\left(\int_{H}\left(\left\Vert w\left(\cdot,h\right)\cdot f\left(\cdot,h\right)\right\Vert _{L^{p}\left(\mathbb{R}^{d}\right)}\right)^{q}\,\frac{\text{d}h}{\left|\det\left(h\right)\right|}\right)^{1/q}
\]
for $p\in\left[1,\infty\right]$ and $q\in\left[1,\infty\right)$
and with 
\[
\left\Vert f\right\Vert _{L_{w}^{p,\infty}}:=\esssup_{h\in H}\left[\left\Vert w\left(\cdot,h\right)\cdot f\left(\cdot,h\right)\right\Vert _{L^{p}\left(\mathbb{R}^{d}\right)}\right].
\]
The weight $v:H\rightarrow\left(0,\infty\right)$ need not be submultiplicative
itself, but we will assume that $v$ is \textbf{$v_{0}$-moderate}
for a (measurable, locally bounded) {\em submultiplicative} weight
$v_{0}:H\rightarrow\left(0,\infty\right)$, i.e. we assume 
\[
v\left(ghk\right)\leq v_{0}\left(g\right)v\left(h\right)v_{0}\left(k\right)\qquad\forall g,h,k\in H.
\]
In this case, $L_{v}^{p,q}\left(G\right)$ is invariant under left-
and right translations. More precisely, we have the following:
\begin{lem}
\label{lem:GemischterLebesgueRaumLinksRechtsTranslation}$L_{v}^{p,q}\left(G\right)$
is invariant under left- and right translations and we have the estimates
\[
\left\Vert L_{\left(x,h\right)}\right\Vert _{L_{v}^{p,q}\rightarrow L_{v}^{p,q}}\leq v_{0}\left(1_{H}\right)\cdot v_{0}\left(h\right)\cdot\left|\det\left(h\right)\right|^{\frac{1}{p}-\frac{1}{q}}
\]
and 
\[
\left\Vert R_{\left(x,h\right)}\right\Vert _{L_{v}^{p,q}\rightarrow L_{v}^{p,q}}\leq v_{0}\left(1_{H}\right)\cdot v_{0}\left(h^{-1}\right)\cdot\left|\det\left(h\right)\right|^{\frac{1}{q}}\cdot\left(\Delta_{H}\left(h\right)\right)^{-1/q}
\]
for all $\left(x,h\right)\in G$.\end{lem}
\begin{proof}
Let $\left(x,h\right)\in G$ and $f\in L_{v}^{p,q}\left(G\right)$.
For $\left(y,g\right)\in G$ we have
\begin{align*}
\left(L_{\left(x,h\right)^{-1}}f\right)\left(y,g\right) & =f\left(\left(x,h\right)\left(y,g\right)\right)=f\left(x+hy,hg\right),\\
\left(R_{\left(x,h\right)}f\right)\left(y,g\right) & =f\left(\left(y,g\right)\left(x,h\right)\right)=f\left(y+gx,gh\right).
\end{align*}
We first consider the left translation. For $g\in H$ we get, using
the above formula 
\[
\left\Vert \left(L_{\left(x,h\right)^{-1}}f\right)\left(\cdot,g\right)\right\Vert _{L^{p}\left(\mathbb{R}^{d}\right)}=\left\Vert f\left(x+h\cdot,hg\right)\right\Vert _{L^{p}\left(\mathbb{R}^{d}\right)}=\left|\det\left(h\right)\right|^{-1/p}\cdot\left\Vert f\left(\cdot,hg\right)\right\Vert _{L^{p}\left(\mathbb{R}^{d}\right)},
\]
as can be seen using the change-of-variables formula for $p\in\left[1,\infty\right)$
and for $p=\infty$ using the fact that $\mathbb{R}^{d}\rightarrow\mathbb{R}^{d},z\mapsto x+hz$
and its inverse map both map null-sets to null-sets.

Thus, we arrive at 
\begin{eqnarray*}
 &  & \frac{v\left(g\right)}{\left|\det\left(g\right)\right|^{1/q}}\cdot\left\Vert \left(L_{\left(x,h\right)^{-1}}f\right)\left(\cdot,g\right)\right\Vert _{L^{p}\left(\mathbb{R}^{d}\right)}\\
 & = & \frac{v\left(h^{-1}hg1_{H}\right)}{\left|\det\left(h^{-1}hg\right)\right|^{1/q}}\cdot\left|\det\left(h\right)\right|^{-1/p}\cdot\left\Vert f\left(\cdot,hg\right)\right\Vert _{L^{p}\left(\mathbb{R}^{d}\right)}\\
 & \leq & v_{0}\left(h^{-1}\right)v_{0}\left(1_{H}\right)\left|\det\left(h\right)\right|^{\frac{1}{q}-\frac{1}{p}}\cdot\frac{v\left(hg\right)}{\left|\det\left(hg\right)\right|^{1/q}}\cdot\left\Vert f\left(\cdot,hg\right)\right\Vert _{L^{p}\left(\mathbb{R}^{d}\right)}.
\end{eqnarray*}
Using the (isometric) invariance of $\left\Vert \cdot\right\Vert _{L^{q}\left(H\right)}$
under left-translations, we obtain
\begin{align*}
\left\Vert L_{\left(x,h\right)^{-1}}f\right\Vert _{L_{v}^{p,q}} & =\left\Vert g\mapsto\frac{v\left(g\right)}{\left|\det\left(g\right)\right|^{1/q}}\cdot\left\Vert \left(L_{\left(x,h\right)^{-1}}f\right)\left(\cdot,g\right)\right\Vert _{L^{p}\left(\mathbb{R}^{d}\right)}\right\Vert _{L^{q}\left(H\right)}\\
 & \leq v_{0}\left(h^{-1}\right)v_{0}\left(1_{H}\right)\left|\det\left(h\right)\right|^{\frac{1}{q}-\frac{1}{p}}\cdot\left\Vert g\mapsto\frac{v\left(hg\right)}{\left|\det\left(hg\right)\right|^{1/q}}\cdot\left\Vert f\left(\cdot,hg\right)\right\Vert _{L^{p}\left(\mathbb{R}^{d}\right)}\right\Vert _{L^{q}\left(H\right)}\\
 & =v_{0}\left(h^{-1}\right)v_{0}\left(1_{H}\right)\left|\det\left(h\right)\right|^{\frac{1}{q}-\frac{1}{p}}\cdot\left\Vert f\right\Vert _{L_{v}^{p,q}\left(G\right)}<\infty
\end{align*}
Applying this to $\left(x,h\right)^{-1}=\left(-h^{-1}x,h^{-1}\right)$
instead of $\left(x,h\right)$, we finally see 
\[
\left\Vert L_{\left(x,h\right)}f\right\Vert _{L_{v}^{p,q}}=\left\Vert L_{\left(-h^{-1}x,h^{-1}\right)^{-1}}f\right\Vert _{L_{v}^{p,q}}\leq v_{0}\left(h\right)\cdot v_{0}\left(1_{H}\right)\cdot\left|\det\left(h\right)\right|^{\frac{1}{p}-\frac{1}{q}}\cdot\left\Vert f\right\Vert _{L_{v}^{p,q}}.
\]
We now turn to the right translations. Using the translation-invariance
of $\left\Vert \cdot\right\Vert _{L^{p}\left(\mathbb{R}^{d}\right)}$,
we derive 
\[
\left\Vert \left(R_{\left(x,h\right)}f\right)\left(\cdot,g\right)\right\Vert _{L^{p}\left(\mathbb{R}^{d}\right)}=\left\Vert f\left(\cdot+gx,gh\right)\right\Vert _{L^{p}\left(\mathbb{R}^{d}\right)}=\left\Vert f\left(\cdot,gh\right)\right\Vert _{L^{p}\left(\mathbb{R}^{d}\right)}.
\]
This implies 
\begin{eqnarray*}
 &  & \frac{v\left(g\right)}{\left|\det\left(g\right)\right|^{1/q}}\cdot\left\Vert \left(R_{\left(x,h\right)}f\right)\left(\cdot,g\right)\right\Vert _{L^{p}\left(\mathbb{R}^{d}\right)}\\
 & = & \frac{v\left(1_{H}ghh^{-1}\right)}{\left|\det\left(ghh^{-1}\right)\right|^{1/q}}\cdot\left\Vert f\left(\cdot,gh\right)\right\Vert _{L^{p}\left(\mathbb{R}^{d}\right)}\\
 & \leq & v_{0}\left(1_{H}\right)\cdot v_{0}\left(h^{-1}\right)\cdot\left|\det\left(h\right)\right|^{\frac{1}{q}}\cdot\frac{v\left(gh\right)}{\left|\det\left(gh\right)\right|^{1/q}}\cdot\left\Vert f\left(\cdot,gh\right)\right\Vert _{L^{p}\left(\mathbb{R}^{d}\right)}.
\end{eqnarray*}
Now, for $q\in\left[1,\infty\right)$ and (measurable) $\psi:H\rightarrow\mathbb{C}$,
formula (2.26) of \cite{FollandAHA} implies
\[
\left\Vert \psi\left(\cdot h\right)\right\Vert _{L^{q}\left(H\right)}=\left(\Delta_{H}\left(h\right)\right)^{-1/q}\cdot\left\Vert \psi\right\Vert _{L^{q}\left(H\right)}.
\]
The same is true for $q=\infty$, as right-translations map (left)
null-sets to (left) null-sets.

Therefore, we arrive at
\begin{align*}
\quad\left\Vert R_{\left(x,h\right)}f\right\Vert _{L_{v}^{p,q}} & =\left\Vert g\mapsto\frac{v\left(g\right)}{\left|\det\left(g\right)\right|^{1/q}}\cdot\left\Vert \left(R_{\left(x,h\right)}f\right)\left(\cdot,g\right)\right\Vert _{L^{p}\left(\mathbb{R}^{d}\right)}\right\Vert _{L^{q}\left(H\right)}\\
 & \leq v_{0}\left(1_{H}\right)v_{0}\left(h^{-1}\right)\cdot\left|\det\left(h\right)\right|^{\frac{1}{q}}\cdot\left\Vert g\mapsto\frac{v\left(gh\right)}{\left|\det\left(gh\right)\right|^{1/q}}\left\Vert f\left(\cdot,gh\right)\right\Vert _{L^{p}\left(\mathbb{R}^{d}\right)}\right\Vert _{L^{q}\left(H\right)}\\
 & =v_{0}\left(1_{H}\right)v_{0}\left(h^{-1}\right)\cdot\left|\det\left(h\right)\right|^{\frac{1}{q}}\cdot\left(\Delta_{H}\left(h\right)\right)^{-1/q}\cdot\left\Vert f\right\Vert _{L_{v}^{p,q}}\qquad\qquad\qquad\quad\;\qedhere
\end{align*}

\end{proof}
Using the $v_{0}$-moderateness of $v$ and boundedness of $v_{0}$
on compact sets (from below and above) one can easily establish the
same properties for $v$. These properties (as stated in the next
lemma) will be frequently used in the rest of the paper.
\begin{lem}
\label{lem:PseudoGewichtVerhaeltSichGut}Let $K\subset H$ be compact.
There is a constant $C=C\left(K,v_{0}\right)>0$ such that 
\[
\frac{v\left(g\right)}{v\left(h\right)}\leq C\qquad\text{holds for all }g,h\in K.
\]
Furthermore, there are $\alpha,\beta\in\left(0,\infty\right)$ (only
dependent on $v_{0}$ and $K$) with 
\[
\alpha\leq v\left(h\right)\leq\beta\qquad\text{for all }h\in K.
\]

\end{lem}
Additionally, we note some easy closure-properties of submultiplicative
functions that will be used below:
\begin{lem}
\label{lem:SubmultiplikativAbschlussEigenschaften}Let $w_{1},w_{2}:G\rightarrow\left(0,\infty\right)$
be submultiplicative. Then the same holds for $w_{1}\cdot w_{2}$
and $\max\left\{ w_{1},w_{2}\right\} $ as well as for $w_{1}^{\vee}:G\rightarrow\left(0,\infty\right),x\mapsto w_{1}\left(x^{-1}\right)$. 
\end{lem}
With these preparations, we can now show that the space $L_{v}^{p,q}\left(G\right)$
satisfies all requirements of coorbit theory (with the exception of
the requirement $\mathcal{A}_{w}\neq\left\{ 0\right\} \neq\mathcal{B}_{w}$
which we will establish in Theorem \ref{thm:ZulaessigkeitVonbandbeschraenktenFunktionen}
below).
\begin{lem}
\label{lem:CoorbitVoraussetzungen}For $u:H\rightarrow\left(0,\infty\right)$
we set 
\[
u^{+}:H\rightarrow\left(0,\infty\right),h\mapsto\max\left\{ u\left(h\right),u\left(h^{-1}\right)\right\} .
\]
Let 
\[
w:H\rightarrow\left(0,\infty\right),h\mapsto v_{0}\left(1_{H}\right)\cdot v_{0}^{+}\!\left(h\right)\cdot\left|\det\left(\cdot\right)\right|^{+}\!\!\left(h\right)\cdot\Delta_{H}^{+}\left(h\right).
\]
Then $L_{v}^{p,q}\left(G\right)$ is a solid Banach function space,
$w$ is a (locally bounded, measurable) submultiplicative weight that
satisfies 
\begin{equation}
\max\left\{ \vertiii{L_{\left(x,h\right)}},\vertiii{\smash{L_{\left(x,h\right)^{-1}}}},\vertiii{\smash{R_{\left(x,h\right)}}},\vertiii{\smash{R_{\left(x,h\right)^{-1}}}}\cdot\Delta_{G}\left(\smash{\left(x,h\right)^{-1}}\right)\right\} \leq w\left(h\right)\label{eq:FelixGewichtDominiertFeichtingerGewicht}
\end{equation}
for all $\left(x,h\right)\in G$, where we have written $\vertiii T:=\left\Vert T\right\Vert _{L_{v}^{p,q}\rightarrow L_{v}^{p,q}}$.

Interpreting $w$ as a submultiplicative weight on $G$, we have 
\begin{equation}
\left\Vert f\ast g\right\Vert _{L_{v}^{p,q}}\leq\left\Vert f\right\Vert _{L_{v}^{p,q}}\cdot\left\Vert g^{\vee}\right\Vert _{L_{w}^{1}}\label{eq:FaltungsRelationGemischterLebesgueRaum}
\end{equation}
for all $f\in L_{v}^{p,q}\left(G\right)$ and $g\in\left(L_{w}^{1}\left(G\right)\right)^{\vee}$,
where we used the notation $g^{\vee}:G\rightarrow\mathbb{C},x\mapsto g\left(x^{-1}\right)$.
Here, the integral defining the convolution converges (absolutely)
for almost every $\left(x,h\right)\in G$.\end{lem}
\begin{rem*}
In summary, this shows that $w$ is a suitable \textbf{control weight}
for $L_{v}^{p,q}\left(G\right)$ in the sense of \cite{FeichtingerCoorbit1}
(cf. \cite[equations (3.1) and (4.10)]{FeichtingerCoorbit1} and note
that any weight dominating a control weight is again an admissible
control weight by \cite[Theorem 4.2(iii)]{FeichtingerCoorbit1}).\end{rem*}
\begin{proof}
We first note that $v_{0}\left(1_{H}\right)=v_{0}\left(1_{H}1_{H}\right)\leq v_{0}\left(1_{H}\right)\cdot v_{0}\left(1_{H}\right)$
implies that the constant map $h\mapsto v_{0}\left(1_{H}\right)$
is submultiplicative. Now Lemma \ref{lem:SubmultiplikativAbschlussEigenschaften}
easily shows (with the multiplicativity of $\left|\det\left(\cdot\right)\right|$
and $\Delta_{H}$) that $w$ is submultiplicative. As $\left|\det\left(\cdot\right)\right|$
and $\Delta_{H}$ are continuous and $v_{0}$ is locally bounded and
measurable, the same is true of $w$.

We first prove inequality (\ref{eq:FelixGewichtDominiertFeichtingerGewicht}).
To this end, we notice that for $\alpha\in\left(0,\infty\right)$
and $r\in\left[-1,1\right]$ the inequality $\max\left\{ \alpha,\alpha^{-1}\right\} \geq1$
yields the estimate 
\begin{equation}
\alpha^{r} 
 \leq \max\left\{ \alpha,\alpha^{-1}\right\} .\label{eq:PotenzAbschaetzung}
\end{equation}
In the following we will apply this for $\alpha=\left|\det\left(h\right)\right|$
or $\alpha=\Delta_{H}\left(h\right)$.

Lemma \ref{lem:GemischterLebesgueRaumLinksRechtsTranslation} together
with $-1\leq-\frac{1}{q}\leq\frac{1}{p}-\frac{1}{q}\leq\frac{1}{p}\leq1$
shows 
\begin{eqnarray*}
\left\Vert L_{\left(x,h\right)}\right\Vert _{L_{v}^{p,q}\rightarrow L_{v}^{p,q}} & \leq & v_{0}\left(h\right)\cdot v_{0}\left(1_{H}\right)\cdot\left|\det\left(h\right)\right|^{\frac{1}{p}-\frac{1}{q}}\\
 & \overset{\text{Eq. }\eqref{eq:PotenzAbschaetzung}}{\leq} & v_{0}\left(1_{H}\right)\cdot\max\left\{ v_{0}\left(h\right),v_{0}\left(h^{-1}\right)\right\} \cdot\max\left\{ \left|\det\left(h\right)\right|,\left|\det\left(h^{-1}\right)\right|\right\} \\
 & \leq & w\left(h\right),
\end{eqnarray*}
where we used $\max\left\{ \Delta_{H}\left(h\right),\Delta_{H}\left(h^{-1}\right)\right\} \geq1$.

In the same way, equation (\ref{eq:PotenzAbschaetzung}) and Lemma
\ref{lem:GemischterLebesgueRaumLinksRechtsTranslation} imply $\left\Vert R_{\left(x,h\right)}\right\Vert _{L_{v}^{p,q}\rightarrow L_{v}^{p,q}}\leq w\left(x,h\right)$.
By symmetry of $w$, we also get $\left\Vert \smash{L_{\left(x,h\right)^{-1}}}\right\Vert _{L_{v}^{p,q}\rightarrow L_{v}^{p,q}}\leq w\left(\smash{\left(x,h\right)^{-1}}\right)=w\left(x,h\right)$.
Finally, by Lemma \ref{lem:GemischterLebesgueRaumLinksRechtsTranslation}
and because of $1-\frac{1}{q}\in\left[-1,1\right]$, we arrive at
\begin{eqnarray*}
    &  & \left\Vert \smash{R_{\left(x,h\right)^{-1}}}\right\Vert _{L_{v}^{p,q}\rightarrow L_{v}^{p,q}}\cdot\Delta_{G}\left(\smash{\left(x,h\right)^{-1}}\right)\\
 & \leq & v_{0}\left(1_{H}\right)\cdot v_{0}\left(h\right)\cdot\left|\det\left(h\right)\right|^{-\frac{1}{q}}\cdot\left(\Delta_{H}\left(h\right)\right)^{1/q}\cdot\left(\Delta_{G}\left(x,h\right)\right)^{-1}\\
 & \overset{\text{Eq. }\eqref{eq:SemiDirektModularFunktion}}{=} & v_{0}\left(1_{H}\right)\cdot v_{0}\left(h\right)\cdot\left|\det\left(h\right)\right|^{-\frac{1}{q}}\cdot\left(\Delta_{H}\left(h\right)\right)^{1/q}\cdot\left(\Delta_{H}\left(h\right)\right)^{-1}\cdot\left|\det\left(h\right)\right|\\
 & = & v_{0}\left(1_{H}\right)\cdot v_{0}\left(h\right)\cdot\left|\det\left(h\right)\right|^{1-\frac{1}{q}}\cdot\left(\Delta_{H}\left(h\right)\right)^{\frac{1}{q}-1}\\
 & \leq & w\left(h\right).
\end{eqnarray*}

The Banach function space properties of $L_{v}^{p,q}$ are routinely
checked. Finally, we establish the convolution relation (\ref{eq:FaltungsRelationGemischterLebesgueRaum}).
Here we first observe the identity 
\begin{eqnarray*}
F\left(x,h\right) & := & \int_{G}\left|f\left(y,k\right)\cdot g\left(\left(y,k\right)^{-1}\left(x,h\right)\right)\right|\,\text{d}\left(y,k\right)\\
 & = & \int_{G}\left|f\left(\left(x,h\right)\left(y,k\right)\right)\cdot g^{\vee}\left(y,k\right)\right|\,\text{d}\left(y,k\right)\in\left[0,\infty\right]
\end{eqnarray*}
which is valid by left invariance. Now Minkowski's inequality for
integrals (cf. \cite[Theorem 6.19]{FollandRA}) with 
\[
\text{d}\nu:=\frac{v\left(h\right)}{\left|\det\left(h\right)\right|^{1/q}}\cdot\text{d}h
\]
yields (together with the solidity of $L^{q}\left(\nu\right)$) the
estimate 
\begin{eqnarray}
\left\Vert F\right\Vert _{L_{v}^{p,q}} & = & \left\Vert h\mapsto\left\Vert x\mapsto\int_{G}\left|f\left(\left(x,h\right)\left(y,k\right)\right)\cdot g^{\vee}\left(y,k\right)\right|\,\text{d}\left(y,k\right)\right\Vert _{L^{p}\left(\mathbb{R}^{d}\right)}\right\Vert _{L^{q}\left(\nu\right)}\nonumber \\
 & \leq & \left\Vert h\mapsto\int_{G}\left\Vert x\mapsto\left|f\left(\left(x,h\right)\left(y,k\right)\right)\cdot g^{\vee}\left(y,k\right)\right|\right\Vert _{L^{p}\left(\mathbb{R}^{d}\right)}\,\text{d}\left(y,k\right)\right\Vert _{L^{q}\left(\nu\right)}\nonumber \\
 & \leq & \int_{G}\left\Vert h\mapsto\left\Vert x\mapsto\left|f\left(\left(x,h\right)\left(y,k\right)\right)\cdot g^{\vee}\left(y,k\right)\right|\right\Vert _{L^{p}\left(\mathbb{R}^{d}\right)}\right\Vert _{L^{q}\left(\nu\right)}\,\text{d}\left(y,k\right)\nonumber \\
 & = & \int_{G}\left|g^{\vee}\left(y,k\right)\right|\cdot\left\Vert R_{\left(y,k\right)}f\right\Vert _{L_{v}^{p,q}}\,\text{d}\left(y,k\right)\nonumber \\
 & \overset{\text{Eq. }\eqref{eq:FelixGewichtDominiertFeichtingerGewicht}}{\leq} & \int_{G}\left|g^{\vee}\left(y,k\right)\right|\cdot w\left(k\right)\cdot\left\Vert f\right\Vert _{L_{v}^{p,q}}\,\text{d}\left(y,k\right)\nonumber \\
 & = & \left\Vert f\right\Vert _{L_{v}^{p,q}\left(G\right)}\cdot\left\Vert g^{\vee}\right\Vert _{L_{w}^{1}\left(G\right)}<\infty.\label{eq:GemischteLebesgueRaeumeFaltung}
\end{eqnarray}
In particular, we conclude $F\left(x,h\right)<\infty$ for almost
every (depending on $h$) $x\in\mathbb{R}^{d}$ for $\nu$-almost
every $h\in H$. Sine we have $\frac{v\left(h\right)}{\left|\det\left(h\right)\right|^{1/q}}>0$
for all $h\in H$ and because $F$ is measurable (which is implied
by Fubini's theorem), we see $F\left(x,h\right)<\infty$ for almost
every $\left(x,h\right)\in G$. Thus, the convolution-defining integral
\[
\left(f\ast g\right)\left(x,h\right)=\int_{G}f\left(y,k\right)\cdot g\left(\left(y,k\right)^{-1}\left(x,h\right)\right)\,\text{d}\left(y,k\right)
\]
converges absolutely for almost every $\left(x,h\right)\in G$ with
$\left|\left(f\ast g\right)\left(x,h\right)\right|\leq F\left(x,h\right)$.
By solidity of $L_{v}^{p,q}$, this implies $f\ast g\in L_{v}^{p,q}\left(G\right)$
and 
\[
\left\Vert f\ast g\right\Vert _{L_{v}^{p,q}}\leq\left\Vert F\right\Vert _{L_{v}^{p,q}}\leq\left\Vert f\right\Vert _{L_{v}^{p,q}}\cdot\left\Vert g^{\vee}\right\Vert _{L_{w}^{1}\left(G\right)}<\infty.\qedhere
\]

\end{proof}

\section{Admissibility of $\mathcal{F}^{-1}\left(\mathcal{D}\left(\mathcal{O}\right)\right)$
as analyzing vectors and identification of $\left(\mathcal{H}_{w}^{1}\right)^{\neg}$
with a subspace of $\mathcal{D}'\left(\mathcal{O}\right)$}

\label{sec:AdmissibleVectors}In this section we show that any Schwartz
function $\psi\in\mathcal{S}\left(\mathbb{R}^{d}\right)$ whose Fourier
transform $\widehat{\psi}$ is compactly supported in the dual orbit
$\mathcal{O}$ is admissible as an analyzing vector. This will also
allow us to identify the ``reservoir'' $\left(\mathcal{H}_{w}^{1}\right)^{\neg}$
that is used in the definition of coorbit spaces with (a subspace
of) the space of distributions $\mathcal{D}'\left(\mathcal{O}\right)$
on the dual orbit $\mathcal{O}$ as well as with a subspace of $\left(\mathcal{F}\left(\mathcal{D}\left(\mathcal{O}\right)\right)\right)'$.

Before we go into the details of the proof, we recall some important
definitions related to coorbit theory. First of all we recall the
definition of the set of \textbf{analyzing vectors} 
\[
\mathcal{A}_{w}:=\left\{ \psi\in L^{2}\left(\mathbb{R}^{d}\right)\with W_{\psi}\psi\in L_{w}^{1}\left(G\right)\right\} 
\]
and of the set of ``\textbf{better vectors}'' 
\[
\mathcal{B}_{w}:=\left\{ \psi\in L^{2}\left(\mathbb{R}^{d}\right)\with W_{\psi}\psi\in W^{R}\left(L^{\infty},L_{w}^{1}\right)\right\} 
\]
from \cite[pages 317 and 321]{FeichtingerCoorbit1}.

Here, we use the notion of the so-called \textbf{Wiener amalgam space}
$W^{R}\left(L^{\infty},Y\right)$ for a solid Banach function space
$Y$ (cf. \cite[pages 312 and 315]{FeichtingerCoorbit1}). For the
definition of this space, let $U\subset G$ be an open, precompact
unit-neighborhood. For $f:G\rightarrow\mathbb{C}$ we define the \textbf{(right
sided) control function} $K_{U}f$ of $f$ with respect to $U$ by
\begin{equation}
K_{U}f:G\rightarrow\left[0,\infty\right],x\mapsto\left\Vert \chi_{Ux}f\right\Vert _{L^{\infty}\left(G\right)}.\label{eq:RechtsseitigeWienerKontrollFunktion}
\end{equation}
The \textbf{(right sided) Wiener amalgam space} with local component
$L^{\infty}$ and global component $Y$ is then defined by 
\[
W^{R}\left(L^{\infty},Y\right):=\left\{ f\in L_{\text{loc}}^{\infty}\left(G\right)\with K_{U}f\in Y\right\} 
\]
with norm $\left\Vert f\right\Vert _{W^{R}\left(L^{\infty},Y\right)}:=\left\Vert K_{U}f\right\Vert _{Y}$.
Here one should note that (as long as $G$ is first countable) $K_{U}f$
is a lower semicontinuous (and hence measurable) function. Then $W^{R}\left(L^{\infty},Y\right)$
is a Banach space that is independent of the actual choice of $U$
and is continuously embedded in $Y$. These properties are shown in
\cite[Lemma 2.1 and Lemma 2.2 together with Theorem 2.3]{RauhutWienerAmalgam}
for the left sided Wiener amalgam space
\[
W\left(L^{\infty},Y\right)=\left\{ f\in L_{\text{loc}}^{\infty}\left(G\right)\with K_{U}'f\in Y\right\} ,
\]
where the (left sided) control function $K_{U}'f$ of $f$ with respect
to $U$ is defined by
\begin{equation}
K_{U}'f:G\rightarrow\left[0,\infty\right],x\mapsto\left\Vert \chi_{xU}f\right\Vert _{L^{\infty}\left(G\right)}.\label{eq:LinksseitigeWienerKontrollFunktion}
\end{equation}
Note that we have
\[
    \left(K_{U}f\right)(x) = \left\Vert \chi_{Ux}f\right\Vert _{L^{\infty}\left(G\right)}=\left\Vert \chi_{x^{-1}U^{-1}}\cdot f^{\vee}\right\Vert _{L^{\infty}\left(G\right)}=\left(K_{U^{-1}}'\left(f^{\vee}\right)\right)\left(x^{-1}\right)=\left(K_{U^{-1}}'\left(f^{\vee}\right)\right)^{\vee}\left(x\right)
\]
and hence
\[
    \left\Vert f\right\Vert _{W^{R}\left(L^{\infty},Y\right)}=\left\Vert \left(K_{U^{-1}}'\left(f^{\vee}\right)\right)^{\vee}\right\Vert _{Y}=\left\Vert K_{U^{-1}}'\left(f^{\vee}\right)\right\Vert _{Y^{\vee}}=\left\Vert f^{\vee}\right\Vert _{W\left(L^{\infty},Y^{\vee}\right)}=\left\Vert f\right\Vert _{\left(W\left(L^{\infty},Y^{\vee}\right)\right)^{\vee}}.
\]
Thus, one can easily derive the analogous properties for the right
sided amalgam spaces.

We mention that in \cite{FeichtingerCoorbit1}, Feichtinger uses continuous
functions $k\in C_{c}\left(G\right)$ as a cut-off for localization
instead of the simple characteristic function $\chi_{U}$ that we
use. As we use $L^{\infty}\left(G\right)$ as the local component,
this makes no difference.

Below, we will show that any Schwartz function $\psi$ with Fourier
transform $\widehat{\psi}\in\mathcal{D}\left(\mathcal{O}\right)$
already satisfies $\psi\in\mathcal{B}_{w}\subset\mathcal{A}_{w}$
for every (locally bounded, submultiplicative) weight $w:G\rightarrow\left(0,\infty\right)$
that only depends on the second component, i.e. which satisfies $w\left(x,h\right)=w\left(h\right)$
for $\left(x,h\right)\in G$. This shows in particular that $\mathcal{B}_{w}$
is nontrivial, which closes the gap for the applicability of coorbit
theory that was left open in section \ref{sec:ApplicabilityOfCoorbitTheory}
(cf. Lemma \ref{lem:CoorbitVoraussetzungen}).

Moreover, we show that the map 
\[
\mathcal{F}^{-1}:\mathcal{D}\left(\mathcal{O}\right)\rightarrow\mathcal{H}_{w}^{1},g\mapsto\mathcal{F}^{-1}g
\]
is well-defined and continuous, where $\mathcal{H}_{w}^{1}$ is defined
by 
\[
\mathcal{H}_{w}^{1}:=\left\{ f\in L^{2}\left(\mathbb{R}^{d}\right)\with W_{\psi}f\in L_{w}^{1}\left(G\right)\right\} 
\]
for some fixed analyzing vector $\psi\in\mathcal{A}_{w}\setminus\left\{ 0\right\} $
with norm $\left\Vert f\right\Vert _{\mathcal{H}_{w}^{1}}:=\left\Vert W_{\psi}f\right\Vert _{L_{w}^{1}\left(G\right)}$,
cf. \cite[page 317]{FeichtingerCoorbit1}. This will allow us to show
that for $f\in\left(\mathcal{H}_{w}^{1}\right)^{\neg}$ (where $\left(\mathcal{H}_{w}^{1}\right)^{\neg}$
denotes the space of bounded, \emph{anti}linear functionals on $\mathcal{H}_{w}^{1}$)
the map 
\begin{equation}
\mathcal{F}f:\mathcal{D}\left(\mathcal{O}\right)\rightarrow\mathbb{C},g\mapsto f\left(\mathcal{F}^{-1}\overline{g}\right)\label{eq:FouriertrafoAufFeichtingerReservoir}
\end{equation}
is well-defined, linear and continuous, i.e. an element of the space
of distributions $\mathcal{D}'\left(\mathcal{O}\right)$. This definition
of $\mathcal{F}f$ may seem akward at first, but it is natural; see
Remark \ref{rem:SpezialFourierTrafoIsFortsetzungVonNormaler} below.

Recall that the coorbit space $\text{Co}\left(L_{v}^{p,q}\right)$
is defined by 
\[
\text{Co}\left(L_{v}^{p,q}\right)=\left\{ f\in\left(\mathcal{H}_{w}^{1}\right)^{\neg}\with W_{\psi}f\in L_{v}^{p,q}\left(G\right)\right\} 
\]
for some fixed $\psi\in\mathcal{A}_{w}\setminus\left\{ 0\right\} $
and a control weight $w:G\rightarrow\left(0,\infty\right)$ for $L_{v}^{p,q}\left(G\right)$.
Thus, the map 
\[
\mathcal{F}:\text{Co}\left(L_{v}^{p,q}\right)\rightarrow\mathcal{D}'\left(\mathcal{O}\right),f\mapsto\mathcal{F}f\text{ with }\mathcal{F}f\text{ as defined in equation }\eqref{eq:FouriertrafoAufFeichtingerReservoir}
\]
is well-defined. In section \ref{sec:CoorbitAsDecomposition} we will show that it is indeed
well-defined and bounded as a map into the decomposition space $\mathcal{D}\left(\mathcal{Q},L^{p},\ell_{u}^{q}\right)$,
where $\mathcal{Q}$ is a suitable covering of $\mathcal{O}$ induced
by $H$ and where $u:\mathcal{O}\rightarrow\left(0,\infty\right)$
is a suitably chosen weight.
\begin{rem}
\label{rem:SpezialFourierTrafoIsFortsetzungVonNormaler}The definition
of $\mathcal{F}f$ in equation (\ref{eq:FouriertrafoAufFeichtingerReservoir})
(denoted as $\mathcal{F}_{\left(\mathcal{H}_{w}^{1}\right)^{\neg}}f$
in this remark to distinguish it from the ``ordinary'' Fourier transform)
is natural, because we have the embedding 
\[
\mathcal{H}_{w}^{1}=\left\{ f\in L^{2}\left(\mathbb{R}^{d}\right)\with W_{\psi}f\in L_{w}^{1}\left(G\right)\right\} \;\overset{f\mapsto f}{\hookrightarrow}\; L^{2}\left(\mathbb{R}^{d}\right)\;\overset{f\mapsto\left\langle \cdot,f\right\rangle _{\text{anti}}}{\hookrightarrow}\;\left(\mathcal{H}_{w}^{1}\right)^{\neg},
\]
where $\left\langle \cdot,\cdot\right\rangle _{\text{anti}}$ is linear
in the second variable and antilinear in the first (cf. \cite[equation (4.1)]{FeichtingerCoorbit1}).

Thus, for $f\in L^{2}\left(\mathbb{R}^{d}\right)\hookrightarrow\left(\mathcal{H}_{w}^{1}\right)^{\neg}$
and $\psi\in\mathcal{D}\left(\mathcal{O}\right)$ we have 
\begin{eqnarray*}
\left(\mathcal{F}_{\left(\mathcal{H}_{w}^{1}\right)^{\neg}}f\right)\left(\psi\right) & \overset{\text{Eq. }\eqref{eq:FouriertrafoAufFeichtingerReservoir}}{=} & f\left(\mathcal{F}^{-1}\overline{\psi}\right)\\
 & = & \left\langle \mathcal{F}^{-1}\overline{\psi},f\right\rangle _{\text{anti}}\\
 & = & \left\langle f,\mathcal{F}^{-1}\overline{\psi}\right\rangle _{L^{2}\left(\mathbb{R}^{d}\right)}\\
 & \overset{\text{Plancherel}}{=} & \left\langle \mathcal{F}f,\overline{\psi}\right\rangle _{L^{2}\left(\mathbb{R}^{d}\right)}\\
 & = & \left\langle \mathcal{F}f,\psi\right\rangle _{\mathcal{S}',\mathcal{S}},
\end{eqnarray*}
where $\left\langle \cdot,\cdot\right\rangle _{\mathcal{S}',\mathcal{S}}$
denotes the \emph{bilinear(!)} pairing between $\mathcal{S}'\left(\mathbb{R}^{d}\right)$
and $\mathcal{S}\left(\mathbb{R}^{d}\right)$.

This shows that $\mathcal{F}_{\left(\mathcal{H}_{w}^{1}\right)^{\neg}}$
and the ``ordinary'' Fourier transform agree on $L^{2}\left(\mathbb{R}^{d}\right)\subset\left(\mathcal{H}_{w}^{1}\right)^{\neg}$.
\end{rem}
In order to show that every Schwartz function $\psi$ whose Fourier
transform is compactly supported in $\mathcal{O}$ is admissible as
an analyzing vector, we will need the (well known) fact that the \textbf{orbit
map} 
\begin{equation}
p_{\xi_{0}}:H\rightarrow\mathcal{O},h\mapsto h^{T}\xi_{0}\label{eq:DefinitionOrbitProjektion}
\end{equation}
is a proper map, i.e. for $K\subset\mathcal{O}$ compact the preimage
$p_{\xi_{0}}^{-1}\left(K\right)\subset H$ is also compact. We will
see that this is a consequence of our admissibility assumptions on
$H$, more precisely of the compactness of the isotropy group $H_{\xi_{0}}\leq H$
and of the fact that the orbit $\mathcal{O}$ is an open subset of
$\mathbb{R}^{d}$. This is the first point (apart from the applicability
of coorbit theory), where we actually use these assumptions.
\begin{lem}
\label{lem:OrbitProjektionIstProper}For a compact set $K\subset\mathcal{O}$
the inverse image $p_{\xi_{0}}^{-1}\left(K\right)\subset H$ is also
compact.\end{lem}
\begin{proof}
By the closed subgroup theorem, $H$ is a Lie group. As a second countable,
locally compact space it admits an exhaustion by precompact open sets,
i.e. $H=\bigcup_{n\in\mathbb{N}}U_{n}$, where the $U_{n}\subset H$
are open precompact sets satisfying $U_{n}\subset U_{n+1}$ for all
$n\in\mathbb{N}$. This implies that $C\subset H$ is relatively compact
iff $C\subset U_{n}$ holds for some $n\in\mathbb{N}$.

By general properties of the orbit maps of smooth Lie group actions
(cf. \cite[Propostion 7.26]{SmoothManifolds}), $p_{\xi_{0}}$ has
constant rank. As $p_{\xi_{0}}:H\rightarrow\mathcal{O}$ is surjective,
the global rank theorem (cf. \cite[Theorem 4.14]{SmoothManifolds})
shows that $p_{\xi_{0}}$ is a smooth submersion and hence an open
map.

This implies that the sets $V_{n}:=p_{\xi_{0}}(U_{n})\subset\mathcal{O}$
form an increasing cover of $\mathcal{O}$ by open sets. Let $K\subset\mathcal{O}$
be compact. By the same reasoning as before, we get $K\subset V_{n}$ for
some $n\in\mathbb{N}$. But this implies $p_{\xi_{0}}^{-1}(K)\subset H_{\xi_{0}}U_{n}\subset H_{\xi_{0}}\overline{U_{n}}$,
where $H_{\xi_{0}}\leq H$ is the compact(!) stabilizer of $\xi_{0}$.
As $\overline{U_{n}}\subset H$ is compact, this shows that $p_{\xi_{0}}^{-1}\left(K\right)$
is compact as a closed subset of the compact set $H_{\xi_{0}}\overline{U_{n}}$.
\end{proof}
Before we can prove the main result of this section, we need some
additional results on the continuity of the maps $H\rightarrow\mathcal{S}\left(\mathbb{R}^{d}\right),h\mapsto D_{h}\psi$
and $H\rightarrow\mathcal{S}\left(\mathbb{R}^{d}\right),h\mapsto\left(W_{\psi}\left(\mathcal{F}^{-1}f\right)\right)\left(\cdot,h\right)$.
These results will then be used to show the continuity of the map
\[
\mathcal{D}\left(\mathcal{O}\right)\rightarrow W^{R}\left(L^{\infty},L_{w}^{p,q}\left(G\right)\right),g\mapsto W_{\psi}\left(\mathcal{F}^{-1}g\right).
\]

\begin{lem}
\label{lem:GlatteVerkettung}Let $\emptyset\neq U\subset\mathbb{R}^{\ell}$
be open and let $\gamma:U\times\mathbb{R}^{d}\rightarrow\mathbb{R}^{d}$
be smooth with the additional property that for all compact sets $L\subset U$
and $K\subset\mathbb{R}^{d}$ the set 
\[
\bigcup_{p\in L}\left(\gamma\left(p,\cdot\right)\right)^{-1}\left(K\right)\subset\mathbb{R}^{d}
\]
is bounded. Furthermore, let $\varphi\in\mathcal{D}\left(\mathbb{R}^{d}\right)$
be arbitrary.

Then the map 
\[
\Phi:U\rightarrow\mathcal{D}\left(\mathbb{R}^{d}\right),p\mapsto\varphi\left(\gamma\left(p,\cdot\right)\right)
\]
is well-defined and continuous. In particular, $\Phi:U\rightarrow\mathcal{S}\left(\mathbb{R}^{d}\right)$
is continuous.\end{lem}
\begin{rem*}
The stated requirements for $\gamma$ are (under the identification
$\mathbb{R}^{d\times d}\cong\mathbb{R}^{d^{2}}$) fulfilled for the
choice 
\[
\gamma:\text{GL}\left(\smash{\mathbb{R}^{d}}\right)\times\mathbb{R}^{d}\rightarrow\mathbb{R}^{d},\left(h,\xi\right)\mapsto h^{T}\xi.
\]
\end{rem*}
\begin{proof}[Proof of the remark]
Let $L\subset\text{GL}\left(\mathbb{R}^{d}\right)$ and $K\subset\mathbb{R}^{d}$
be compact. Then $L^{-T}\subset\text{GL}\left(\mathbb{R}^{d}\right)$
is also compact, which implies $\left\Vert h^{-T}\right\Vert \leq C_{1}$
for some $C_{1}>0$ and all $h\in L$. Furthermore, there is some
$C_{2}>0$ such that $\left|\xi\right|\leq C_{2}$ holds for all $\xi\in K$.

For $y\in\bigcup_{h\in L}\left(\gamma\left(h,\cdot\right)\right)^{-1}\left(K\right)$
we then have $\xi:=h^{T}y=\gamma\left(h,y\right)\in K$ for some $h\in L$.
This implies 
\[
\left|y\right|=\left|h^{-T}\xi\right|\leq\left\Vert h^{-T}\right\Vert \cdot\left|\xi\right|\leq C_{1}C_{2}.\qedhere
\]
\end{proof}

\begin{proof}[Proof of Lemma \ref{lem:GlatteVerkettung}]
As $\gamma$ (and hence also $\gamma\left(p,\cdot\right)$ for every
$p\in U$) is smooth, we see that $\varphi\left(\gamma\left(p,\cdot\right)\right)\in C^{\infty}\left(\mathbb{R}^{d}\right)$
is also smooth. Now let $K:=\text{supp}\left(\varphi\right)$. As
$\gamma\left(p,\cdot\right)$ is continuous, $\left(\gamma\left(p,\cdot\right)\right)^{-1}\left(K\right)\subset\mathbb{R}^{d}$
is closed. It is thus easy to see that 
\begin{equation}
\text{supp}\left(\Phi\left(p\right)\right)=\text{supp}\left(\varphi\left(\gamma\left(p,\cdot\right)\right)\right)\subset\left(\gamma\left(p,\cdot\right)\right)^{-1}\left(K\right)\label{eq:GlatteVerkettungTraeger}
\end{equation}
holds for every $p\in U$. Now the assumption (with $L=\left\{ p\right\} $)
yields that the right-hand side is a bounded subset of $\mathbb{R}^{d}$.
Thus, $\varphi\left(\gamma\left(p,\cdot\right)\right)\in\mathcal{D}\left(\mathbb{R}^{d}\right)$
is compactly supported, so that $\Phi$ is well-defined.

To prove the continuity of $\Phi$, let $\left(p_{n}\right)_{n\in\mathbb{N}}\in U^{\mathbb{N}}$
with $p_{n}\xrightarrow[n\rightarrow\infty]{}p_{0}$ for some $p_{0}\in U$.
Then $L:=\left\{ p_{n}\with n\in\mathbb{N}\right\} \cup\left\{ p_{0}\right\} $
is a compact subset of $U$. The assumption (and Heine-Borel) thus
yield the compactness of 
\[
M:=\overline{\bigcup_{p\in L}\left(\gamma\left(p,\cdot\right)\right)^{-1}\left(K\right)}\subset\mathbb{R}^{d}.
\]
By equation (\ref{eq:GlatteVerkettungTraeger}) we see 
\begin{equation}
\text{supp}\left(\Phi\left(p_{0}\right)\right)\subset M\qquad\text{ and }\qquad\text{supp}\left(\Phi\left(p_{n}\right)\right)\subset M\qquad\text{ for all }n\in\mathbb{N}.\label{eq:GlatteVerkettungGleichmaessigerTraeger}
\end{equation}
Now for every multi-index $\beta\in\mathbb{N}_{0}^{d}$ the map 
\[
\Psi_{\beta}:U\times\mathbb{R}^{d}\rightarrow\mathbb{C},\left(p,x\right)\mapsto\left(\partial^{\beta}\left(\Phi\left(p\right)\right)\right)\left(x\right)=\frac{\partial^{\left|\beta\right|}\left(\left(\varphi\circ\gamma\right)\left(q,y\right)\right)}{\partial y_{1}^{\beta_{1}}\cdots\partial y_{d}^{\beta_{d}}}\bigg|_{\left(q,y\right)=\left(p,x\right)}
\]
is smooth, hence continuous. In particular, $\Psi_{\beta}$ is uniformly
continuous on the compact set $L\times M\subset U\times\mathbb{R}^{d}$.
This yields, for arbitrary $\varepsilon>0$, some $\delta>0$ such
that $\left|\Psi_{\beta}\left(p,x\right)-\Psi_{\beta}\left(q,y\right)\right|<\varepsilon$
holds for all $\left(p,x\right),\left(q,y\right)\in L\times M$ with
$\left|\left(p,x\right)-\left(q,y\right)\right|<\delta$. Let $n_{0}\in\mathbb{N}$
with $\left|p_{n}-p_{0}\right|<\delta$ for $n\geq n_{0}$. For $n\geq n_{0}$
and $x\in\mathbb{R}^{d}$ there are two cases: 
\begin{enumerate}
\item $x\notin M$. By equation \eqref{eq:GlatteVerkettungGleichmaessigerTraeger}
this means $x\notin\text{supp}\left(\Phi\left(p_{n}\right)\right)$
and $x\notin\text{supp}\left(\Phi\left(p_{0}\right)\right)$. Hence,
we conclude $\Phi\left(p_{0}\right)|_{V}\equiv0\equiv\Phi\left(p_{n}\right)|_{V}$
for the neighborhood $V:=M^{c}$ of $x$. In particular 
\[
\left|\left(\partial^{\beta}\left(\Phi\left(p_{n}\right)-\Phi\left(p_{0}\right)\right)\right)\left(x\right)\right|=0<\varepsilon.
\]

\item $x\in M$. Then $\left(p_{n},x\right),\left(p_{0},x\right)\in L\times M$
with $\left|\left(p_{n},x\right)-\left(p_{0},x\right)\right|=\left|p_{n}-p_{0}\right|<\delta$.
By choice of $\delta$ this implies 
\[
\left|\left(\partial^{\beta}\left(\Phi\left(p_{n}\right)-\Phi\left(p_{0}\right)\right)\right)\left(x\right)\right|=\left|\Psi_{\beta}\left(p_{n},x\right)-\Psi_{\beta}\left(p,x\right)\right|<\varepsilon.
\]

\end{enumerate}

Now \cite[Theorem 6.5]{RudinFA} (and the associated remark) show
$\Phi\left(p_{n}\right)\xrightarrow[n\rightarrow\infty]{\mathcal{D}\left(\mathbb{R}^{d}\right)}\Phi\left(p_{0}\right)$
(recall that the supports of $\Phi\left(p_{n}\right)$ are ``uniformly
compact'' by equation (\ref{eq:GlatteVerkettungGleichmaessigerTraeger})).

Since the inclusion $\mathcal{D}\left(\mathbb{R}^{d}\right)\hookrightarrow\mathcal{S}\left(\mathbb{R}^{d}\right)$
is continuous by \cite[Theorem 7.10]{RudinFA}, we are done.\qedhere

\end{proof}
Using this lemma, we can now show that $h\mapsto\left(W_{\psi}f\right)\left(\cdot,h\right)$
is continuous with compact support as a map of $H$ into the space
of Schwartz functions $\mathcal{S}\left(\mathbb{R}^{d}\right)$ as
long as we have $\widehat{\psi},\widehat{f}\in\mathcal{D}\left(\mathcal{O}\right)$: 
\begin{lem}
\label{lem:FunktionenwertigeStetigkeitUndTraegerVonWaveletTrafo}For
$f,\psi\in L^{2}\left(\mathbb{R}^{d}\right)$ we have the identity
\begin{equation}
\left(W_{\psi}f\right)\left(x,h\right)=\left|\det\left(h\right)\right|^{1/2}\cdot\left(\mathcal{F}^{-1}\left(\widehat{f}\cdot D_{h}\overline{\widehat{\psi}}\right)\right)\left(x\right)\qquad\text{ for all }\left(x,h\right)\in G.\label{eq:WavletTransformationDarstellung}
\end{equation}
Now let $f,\psi\in\mathcal{S}\left(\mathbb{R}^{d}\right)$ with $\widehat{f},\widehat{\psi}\in\mathcal{D}\left(\mathcal{O}\right)$.
Then 
\[
\Gamma:H\rightarrow\mathcal{S}\left(\mathbb{R}^{d}\right),h\mapsto\mathcal{F}^{-1}\left(\widehat{f}\cdot D_{h}\overline{\widehat{\psi}}\right)
\]
is well-defined and continuous with compact support 
\begin{equation}
\text{supp}\left(\Gamma\right)\subset\left(p_{\xi_{0}}^{-1}\left(\mbox{supp}\left(\smash{\widehat{f}}\right)\right)\right)^{-1}\cdot H_{\xi_{0}}\cdot p_{\xi_{0}}^{-1}\left(\mbox{supp}\left(\smash{\widehat{\psi}}\right)\right),\label{eq:WaveletTrafoTraeger}
\end{equation}
where we used $p_{\xi_{0}}:H\rightarrow\mathcal{O},h\mapsto h^{T}\xi_{0}$.\end{lem}
\begin{proof}
Equation \eqref{eq:WavletTransformationDarstellung} is an easy consequence
of the Plancherel theorem, equation \eqref{eq:QausiRegulaereDarstellungAufFourierSeite}
and the definitions.

We now show that $\Gamma$ is well-defined and continuous under the
assumptions $\widehat{f},\widehat{\psi}\in\mathcal{D}\left(\mathcal{O}\right)$.
To this end we note that the multiplication map 
\[
\mu_{\widehat{f}}:\mathcal{S}\left(\mathbb{R}^{d}\right)\rightarrow\mathcal{S}\left(\mathbb{R}^{d}\right),g\mapsto\widehat{f}\cdot g
\]
is (well-defined and) continuous by \cite[Theorem 7.4(b)]{RudinFA},
and the same holds for the inverse Fourier transform. Finally, Lemma
\ref{lem:GlatteVerkettung} and the corresponding remark show that
the map 
\[
\Phi:H\rightarrow\mathcal{S}\left(\mathbb{R}^{d}\right),h\mapsto D_{h}\overline{\widehat{\psi}}
\]
is well-defined and continuous. Here we used the assumption $\widehat{\psi}\in\mathcal{D}\left(\mathbb{R}^{d}\right)$.
In summary, this shows that $\Gamma=\mathcal{F}^{-1}\circ\mu_{\widehat{f}}\circ\Phi$
is well-defined and continuous.

Now let $h\in H$ with $0\neq\Gamma\left(h\right)=\mathcal{F}^{-1}\left(\widehat{f}\cdot D_{h}\overline{\widehat{\psi}}\right)$.
This yields $\widehat{f}\cdot D_{h}\overline{\widehat{\psi}}\neq0$
and thus there is some 
\[
\xi\in\text{supp}\left(\smash{\widehat{f}}\right)\cap\text{supp}\left(D_{h}\overline{\widehat{\psi}}\right)\subset\text{supp}\left(\smash{\widehat{f}}\right)\cap h^{-T}\left(\text{supp}\left(\smash{\widehat{\psi}}\right)\right).
\]
The inclusions $\text{supp}\left(\smash{\widehat{f}}\right)\subset\mathcal{O}=H^{T}\xi_{0}$
and $\text{supp}\left(\smash{\widehat{\psi}}\right)\subset\mathcal{O}=H^{T}\xi_{0}$
yield $g_{1},g_{2}\in H$ satisfying $g_{1}^{T}\xi_{0}=\xi=h^{-T}g_{2}^{T}\xi_{0}$.

This implies $\left(g_{1}hg_{2}^{-1}\right)^{T}\xi_{0}=\xi_{0}$,
i.e. $g_{1}hg_{2}^{-1}\in H_{\xi_{0}}$ and thus $h\in g_{1}^{-1}H_{\xi_{0}}g_{2}$.
The inclusions $g_{1}\in p_{\xi_{0}}^{-1}\left(\text{supp}\left(\smash{\widehat{f}}\right)\right)$
and $g_{2}\in p_{\xi_{0}}^{-1}\left(\text{supp}\left(\smash{\widehat{\psi}}\right)\right)$
establish equation \eqref{eq:WaveletTrafoTraeger}. Now Lemma \ref{lem:OrbitProjektionIstProper}
shows that $\Gamma$ indeed has compact support.
\end{proof}
Using the lemmata \ref{lem:OrbitProjektionIstProper}, \ref{lem:GlatteVerkettung}
and \ref{lem:FunktionenwertigeStetigkeitUndTraegerVonWaveletTrafo},
we now show the announced admissibility of every $\psi\in\mathcal{S}\left(\mathbb{R}^{d}\right)$
whose Fourier transform is compactly supported in the dual orbit $\mathcal{O}$.
The following result extends \cite[Lemma 2.7]{FuehrCoorbit1}, by
including a continuity statement that will be useful for the following. 
\begin{thm}
\label{thm:ZulaessigkeitVonbandbeschraenktenFunktionen}Let $w_{0}:H\rightarrow\left(0,\infty\right)$
be measurable and locally bounded and let $N\in\mathbb{N}_{0}$. Define
\[
w:G\rightarrow\left(0,\infty\right),\left(x,h\right)\mapsto\left(1+\left|x\right|\right)^{N}\cdot w_{0}\left(h\right).
\]
Fix $\psi\in\mathcal{S}\left(\mathbb{R}^{d}\right)$ with $\widehat{\psi}\in\mathcal{D}\left(\mathcal{O}\right)$.
Then the map 
\[
\varrho:\mathcal{D}\left(\mathcal{O}\right)\rightarrow W^{R}\left(L^{\infty},L_{w}^{p,q}\left(G\right)\right),g\mapsto W_{\psi}\left(\mathcal{F}^{-1}g\right)
\]
is well-defined and continuous.\end{thm}
\begin{rem*}
This implies in particular that the map 
\[
\mathcal{D}\left(\mathcal{O}\right)\rightarrow W^{R}\left(L^{\infty},L_{w}^{1,1}\left(G\right)\right)\hookrightarrow L_{w}^{1,1}\left(G\right)=L_{w}^{1}\left(G\right),g\mapsto W_{\psi}\left(\mathcal{F}^{-1}g\right)
\]
is well-defined and continuous. Furthermore it shows $W_{\psi}\psi=W_{\psi}\left(\mathcal{F}^{-1}\widehat{\psi}\right)\in W^{R}\left(L^{\infty},L_{w}^{1}\left(G\right)\right)$
which means $\psi\in\mathcal{B}_{w}\subset\mathcal{A}_{w}\subset\mathcal{H}_{w}^{1}$
(with the notation of \cite[Page 321]{FeichtingerCoorbit1}). As $\psi\in\mathcal{F}^{-1}\left(\mathcal{D}\left(\mathcal{O}\right)\right)$
was arbitrary, we get $\mathcal{F}^{-1}\left(\mathcal{D}\left(\mathcal{O}\right)\right)\subset\mathcal{B}_{w}$.

Finally, the above theorem implies that the map 
\[
\mathcal{F}^{-1}:\mathcal{D}\left(\mathcal{O}\right)\rightarrow\mathcal{H}_{w}^{1},g\mapsto\mathcal{F}^{-1}g
\]
is well-defined and continuous.\end{rem*}
\begin{proof}
For $\kappa\in\mathbb{N}_{0}$ and $g\in\mathcal{S}\left(\mathbb{R}^{d}\right)$,
let 
\[
\left|g\right|_{\kappa}:=\max_{\substack{\alpha\in\mathbb{N}_{0}^{d}\\
\left|\alpha\right|\leq\kappa
}
}\sup_{x\in\mathbb{R}^{d}}\left(1+\left|x\right|\right)^{\kappa}\left|\left(\partial^{\alpha}g\right)\left(x\right)\right|.
\]
Then the topology on $\mathcal{S}\left(\mathbb{R}^{d}\right)$ is
induced by the family of norms $\left(\left|\cdot\right|_{\kappa}\right)_{\kappa\in\mathbb{N}_{0}}$.
Choose $N_{0}\in\mathbb{N}$ satisfying $N_{0}>\frac{d}{p}+N$. By
continuity of the (inverse) Fourier transform $\mathcal{F}^{-1}:\mathcal{S}\left(\mathbb{R}^{d}\right)\rightarrow\mathcal{S}\left(\mathbb{R}^{d}\right)$,
there is some $N_{1}\in\mathbb{N}$ and a constant $C_{1}>0$ such
that 
\begin{equation}
\left|\mathcal{F}^{-1}g\right|_{N_{0}}\leq C_{1}\cdot\left|g\right|_{N_{1}}\label{eq:FourierTrafoBeschraenktAufSchwartz}
\end{equation}
holds for all $g\in\mathcal{S}\left(\mathbb{R}^{d}\right)$. Here
we used that the norms $\left|\cdot\right|_{\ell}$ are ordered, i.e.
we have $\left|\cdot\right|_{\ell}\leq\left|\cdot\right|_{m}$ for
$\ell\leq m$.

Let $K\subset\mathcal{O}$ be an arbitrary compact subset. For $K_{2}:=\text{supp}\left(\smash{\widehat{\psi}}\right)$
we define 
\[
L:=\left(p_{\xi_{0}}^{-1}\left(K\right)\right)^{-1}\cdot H_{\xi_{0}}\cdot p_{\xi_{0}}^{-1}\left(K_{2}\right)\subset H.
\]
By Lemma \ref{lem:OrbitProjektionIstProper}, $L$ is a compact subset
of $H$. For 
\[
g\in\mathcal{D}_{K}\left(\mathcal{O}\right):=\left\{ f\in\mathcal{D}\left(\mathcal{O}\right)\with\text{supp}\left(f\right)\subset K\right\} ,
\]
Lemma \ref{lem:FunktionenwertigeStetigkeitUndTraegerVonWaveletTrafo}
shows that 
\[
\Gamma_{g}:H\rightarrow\mathcal{S}\left(\mathbb{R}^{d}\right),h\mapsto\mathcal{F}^{-1}\left(g\cdot D_{h}\overline{\widehat{\psi}}\right)=\mathcal{F}^{-1}\left(\widehat{\mathcal{F}^{-1}g}\cdot D_{h}\overline{\widehat{\psi}}\right)
\]
is well-defined and continuous with compact support $\text{supp}\left(\Gamma_{g}\right)\subset L$.

By Lemma \ref{lem:GlatteVerkettung} the map 
\[
\Phi:H\rightarrow\mathcal{S}\left(\mathbb{R}^{d}\right),h\mapsto\overline{D_{h}\widehat{\psi}}=D_{h}\overline{\widehat{\psi}}
\]
is continuous, so that the continuous function 
\[
H\rightarrow\mathbb{R}_{+},h\mapsto\max_{\substack{\alpha\in\mathbb{N}_{0}^{d}\\
\left|\alpha\right|\leq N_{1}
}
}\sum_{\beta\leq\alpha}\binom{\alpha}{\beta}\cdot\left|D_{h}\overline{\widehat{\psi}}\right|_{N_{1}}
\]
attains its maximum $C_{2}\geq0$ on the compact set $L\subset H$.
Now the Leibniz rule shows, for $h\in L$, $x\in\mathbb{R}^{d}$ and
$\alpha\in\mathbb{N}_{0}^{d}$ with $\left|\alpha\right|\leq N_{1}$:
\begin{eqnarray*}
 &  & \left(1+\left|x\right|\right)^{N_{1}}\cdot\left|\left(\partial^{\alpha}\left(g\cdot D_{h}\overline{\widehat{\psi}}\right)\right)\left(x\right)\right|\\
 & \leq & \sum_{\beta\leq\alpha}{\alpha \choose \beta}\cdot\left|\left(\partial^{\beta}g\right)\left(x\right)\right|\cdot\left(1+\left|x\right|\right)^{N_{1}}\cdot\left|\left(\partial^{\alpha-\beta}D_{h}\overline{\widehat{\psi}}\right)\left(x\right)\right|\\
 & \leq & \left|g\right|_{N_{1}}\cdot\sum_{\beta\leq\alpha}{\alpha \choose \beta}\cdot\left|D_{h}\overline{\widehat{\psi}}\right|_{N_{1}}\\
 & \leq & C_{2}\cdot\left|g\right|_{N_{1}},
\end{eqnarray*}
which implies the estimate $\left|g\cdot D_{h}\overline{\widehat{\psi}}\right|_{N_{1}}\leq C_{2}\cdot\left|g\right|_{N_{1}}$.
Thus, for $h\in L$ we derive 
\[
\left|\Gamma_{g}\left(h\right)\right|_{N_{0}}=\left|\mathcal{F}^{-1}\left(g\cdot D_{h}\overline{\widehat{\psi}}\right)\right|_{N_{0}}\overset{\text{Eq. }\eqref{eq:FourierTrafoBeschraenktAufSchwartz}}{\leq}C_{1}\cdot\left|g\cdot D_{h}\overline{\widehat{\psi}}\right|_{N_{1}}\leq C_{1}C_{2}\cdot\left|g\right|_{N_{1}}.
\]
For $g\in\mathcal{D}_{K}$ and $h\in H\setminus L\subset H\setminus\text{supp}\left(\Gamma_{g}\right)$
we have $\left|\Gamma_{g}\left(h\right)\right|_{N_{0}}=0$. Together,
this shows our first intermediate estimate 
\begin{equation}
\left|\Gamma_{g}\left(h\right)\right|_{N_{0}}\leq C_{1}C_{2}\cdot\left|g\right|_{N_{1}}\cdot\chi_{L}\left(h\right)\qquad\text{ for }h\in H\text{ and }g\in\mathcal{D}_{K}.\label{eq:AdmissibilityFundamental}
\end{equation}

Let $V\subset H$ be an arbitrary open, precompact unit neighborhood.
In the following, we will use $U:=B_{1}\left(0\right)\times V\subset G$
for the control function $K_{U}$ of the Wiener amalgam space. Let
$C_{3}:=\max_{k\in\overline{V}}\left\Vert k^{-1}\right\Vert >0$.
For $\left(y,v\right)\in U$ and $\left(x,h\right)\in G$ we then
have 
\[
\left|x\right|=\left|v^{-1}vx\right|\leq\left\Vert v^{-1}\right\Vert \cdot\left|vx\right|\leq C_{3}\cdot\left|vx\right|
\]
and thus 
\[
\left|y+vx\right|\geq\left|vx\right|-\left|y\right|\geq\frac{\left|x\right|}{C_{3}}-1.
\]
This implies 
\[
1+\left|x\right|=1+C_{3}\frac{\left|x\right|}{C_{3}}\leq1+C_{3}\left(1+\left|y+vx\right|\right)\leq\left(1+C_{3}\right)\cdot\left(1+\left|y+vx\right|\right)
\]
and hence 
\begin{eqnarray*}
 &  & \left|\left(W_{\psi}\left(\mathcal{F}^{-1}g\right)\right)\left(\left(y,v\right)\left(x,h\right)\right)\right|\\
 & = & \left|\left(W_{\psi}\left(\mathcal{F}^{-1}g\right)\right)\left(\left(y+vx,vh\right)\right)\right|\\
 & \overset{\text{Eq. }\eqref{eq:WavletTransformationDarstellung}}{=} & \left|\det\left(vh\right)\right|^{1/2}\cdot\left|\left(\mathcal{F}^{-1}\left(g\cdot D_{vh}\overline{\widehat{\psi}}\right)\right)\left(y+vx\right)\right|\\
 & \leq & \left|\det\left(vh\right)\right|^{1/2}\cdot\left(1+\left|y+vx\right|\right)^{-N_{0}}\cdot\left|\mathcal{F}^{-1}\left(g\cdot D_{vh}\overline{\widehat{\psi}}\right)\right|_{N_{0}}\\
 & \leq & \left(1+C_{3}\right)^{N_{0}}\cdot\left|\det\left(vh\right)\right|^{1/2}\cdot\left(1+\left|x\right|\right)^{-N_{0}}\cdot\left|\Gamma_{g}\left(vh\right)\right|_{N_{0}}\\
 & \overset{\text{Eq. }\eqref{eq:AdmissibilityFundamental}}{\leq} & C_{1}C_{2}\left(1+C_{3}\right)^{N_{0}}\cdot\left|g\right|_{N_{1}}\cdot\chi_{L}\left(vh\right)\cdot\left|\det\left(vh\right)\right|^{1/2}\cdot\left(1+\left|x\right|\right)^{-N_{0}}.
\end{eqnarray*}
Note that $\chi_{L}\left(vh\right)\neq0$ implies $vh\in L$ and thus
$h\in v^{-1}L\subset\overline{V}^{-1}L$. With $C_{4}:=\max_{k\in L}\left|\det\left(k\right)\right|^{1/2}$
and $C_{5}:=C_{1}C_{2}\left(1+C_{3}\right)^{N_{0}}C_{4}$, we thus
derive 
\begin{eqnarray*}
\left(K_{U}\left(W_{\psi}\left(\mathcal{F}^{-1}g\right)\right)\right)\left(x,h\right) & = & \left\Vert \chi_{U\left(x,h\right)}\cdot W_{\psi}\left(\mathcal{F}^{-1}g\right)\right\Vert _{L^{\infty}\left(G\right)}\\
 & \leq & \sup_{\left(y,v\right)\in U}\left|\left(W_{\psi}\left(\mathcal{F}^{-1}g\right)\right)\left(\left(y,v\right)\left(x,h\right)\right)\right|\\
 & \leq & C_{5}\cdot\left|g\right|_{N_{1}}\cdot\chi_{\overline{V}^{-1}L}\left(h\right)\cdot\left(1+\left|x\right|\right)^{-N_{0}}.
\end{eqnarray*}
Because of $N_{0}>\frac{d}{p}+N$, the constant $C_{6}:=\left\Vert \left(1+\left|x\right|\right)^{N-N_{0}}\right\Vert _{L^{p}\left(\mathbb{R}^{d}\right)}$
is finite. This shows (cf. the definition of $w$ in the statement
of the theorem) 
\begin{align*}
\left\Vert \varrho\left(g\right)\right\Vert _{W^{R}\left(L^{\infty},L_{w}^{p,q}\right)} & =\left\Vert K_{U}\left(W_{\psi}\left(\mathcal{F}^{-1}g\right)\right)\right\Vert _{L_{w}^{p,q}\left(G\right)}\\
 & \leq C_{5}\left|g\right|_{N_{1}}\cdot\left\Vert \left|\det\left(h^{-1}\right)\right|^{1/q}w_{0}\left(h\right)\cdot\chi_{\overline{V}^{-1}L}\left(h\right)\cdot\left\Vert \left(1+\left|x\right|\right)^{N-N_{0}}\right\Vert _{L^{p}\left(\mathbb{R}^{d}\right)}\right\Vert _{L^{q}\left(H\right)}\\
 & \leq C_{5}C_{6}\left|g\right|_{N_{1}}\cdot\left\Vert \chi_{\overline{V}^{-1}L}\right\Vert _{L^{q}\left(H\right)}\cdot\sup_{h\in\overline{V}^{-1}L}\left[w_{0}\left(h\right)\cdot\left|\det\left(h^{-1}\right)\right|^{1/q}\right]\\
 & =:C_{7}\left|g\right|_{N_{1}}<\infty
\end{align*}
for all $g\in\mathcal{D}_{K}\left(\mathcal{O}\right)$, where the
constant $C_{7}$ does not depend upon $g$. Here, we used compactness
of $\overline{V}^{-1}L$ and local boundedness of $w_{0}$.

As the norm $\left|\cdot\right|_{N_{1}}$ is easily seen to be continuous
on $\mathcal{D}_{K}\left(\mathcal{O}\right)$, where the topology
on $\mathcal{D}_{K}$ is given by uniform convergence of all derivatives
(cf. \cite[Section 6.2]{RudinFA}), we see that the map $\varrho|_{\mathcal{D}_{K}\left(\mathcal{O}\right)}:\mathcal{D}_{K}\left(\mathcal{O}\right)\rightarrow W^{R}\left(L^{\infty},L_{w}^{p,q}\right)$
is well-defined and continuous. Now \cite[Theorem 6.6]{RudinFA} yields
continuity of $\varrho$.
\end{proof}
Using the above theorem, we now show that the \textbf{reservoir}
$\left(\mathcal{H}_{w}^{1}\right)^{\neg}$ can be identified with
a subspace of the space of all distributions $\mathcal{D}'\left(\mathcal{O}\right)$
on the dual orbit $\mathcal{O}$. This is a more convenient reservoir
than $\left(\mathcal{H}_{w}^{1}\right)^{\neg}$ for two reasons: 
\begin{enumerate}
\item As long as the group $H$ is fixed, the space $\mathcal{D}'\left(\mathcal{O}\right)$
is independent of the parameters $p,q,v$ of the space $L_{v}^{p,q}\left(G\right)$.

The same is not true for $\left(\mathcal{H}_{w}^{1}\right)^{\neg}$,
as different choices of $v$ lead to different control weights $w$
(cf. Lemma \ref{lem:CoorbitVoraussetzungen}) and thus to different
spaces $\mathcal{H}_{w}^{1}$. Note though, that this is not a serious
issue, as \cite[Theorem 4.2]{FeichtingerCoorbit1} shows that the
resulting coorbit space is (with certain restrictions) independent
of the choice of $w$.

\item Even if two different groups $H,H'$ are considered, the spaces $\mathcal{D}'\left(\mathcal{O}\right)$
and $\mathcal{D}'\left(\mathcal{O}'\right)$ (where $\mathcal{O}'$
is the open dual orbit of $H'$) can be compared with each other.

If for example the dual orbits $\mathcal{O},\mathcal{O}'$ of $H$
and $H'$ coincide, it is possible to make sense of the statement
that the coorbit space $\text{Co}\left(Y,H\right)$ embeds into $\text{Co}\left(Y',H'\right)$
if each $f\in\text{Co}\left(Y,H\right)\subset\left(\mathcal{H}_{w}^{1}\right)^{\neg}\subset\mathcal{D}'\left(\mathcal{O}\right)=\mathcal{D}'\left(\mathcal{O}'\right)$
is also an element of $\text{Co}\left(Y',H'\right)\subset\mathcal{D}'\left(\mathcal{O}'\right)$
(and if the map thus defined is bounded). Here we have already used
the identification of $\left(\mathcal{H}_{w}^{1}\right)^{\neg}$ with
a subspace of $\mathcal{D}'\left(\mathcal{O}\right)$.

One can even do this if the orbits do not coincide, but are merely
ordered (i.e. $\mathcal{O}\subset\mathcal{O}'$ or vice versa).

\end{enumerate}
\begin{cor}
\label{cor:FourierTrafoAufFeichtingerReservoir}Let $w:H\rightarrow\left(0,\infty\right)$
be locally bounded. Then the map 
\[
\mathcal{F}:\left(\mathcal{H}_{w}^{1}\right)^{\neg}\rightarrow\mathcal{D}'\left(\mathcal{O}\right),f\mapsto\mathcal{F}f
\]
with
\[\mathcal{F}f : \mathcal{D}(\mathcal{O}) \rightarrow \mathbb{C}, g \mapsto f\left( \mathcal{F}^{-1} \overline{g} \right)\]
as defined in equation (\ref{eq:FouriertrafoAufFeichtingerReservoir})
is well-defined, injective and continuous with respect to the weak-$\ast$-topology
on $\left(\mathcal{H}_{w}^{1}\right)^{\neg}$.\end{cor}
\begin{proof}
Let $f\in\left(\mathcal{H}_{w}^{1}\right)^{\neg}$ and $g\in\mathcal{D}\left(\mathcal{O}\right)$.
Then we also have $\overline{g}\in\mathcal{D}\left(\mathcal{O}\right)$
and the conjugation map 
\[
c:\mathcal{D}\left(\mathcal{O}\right)\rightarrow\mathcal{D}\left(\mathcal{O}\right),g\mapsto\overline{g}
\]
is easily seen to be antilinear and continuous. Theorem \ref{thm:ZulaessigkeitVonbandbeschraenktenFunktionen}
(and the ensuing remark) show $\mathcal{F}^{-1}\overline{g}\in\mathcal{H}_{w}^{1}$
as well as the continuity of $\mathcal{F}^{-1}:\mathcal{D}\left(\mathcal{O}\right)\rightarrow\mathcal{H}_{w}^{1}$.

The expression 
\[
\left(\mathcal{F}f\right)\left(g\right)\overset{\text{Eq. }\eqref{eq:FouriertrafoAufFeichtingerReservoir}}{=}f\left(\mathcal{F}^{-1}\overline{g}\right)\in\mathbb{C}
\]
is well-defined because of $\mathcal{F}^{-1}\overline{g}\in\mathcal{H}_{w}^{1}$.
The \emph{anti}linearity of $f$ and $c$ show that the map 
\[
\mathcal{F}f=f\circ\mathcal{F}^{-1}\circ c:\mathcal{D}\left(\mathcal{O}\right)\rightarrow\mathbb{C}
\]
is linear and continuous as a composition of continuous maps, i.e.
$\mathcal{F}f\in\mathcal{D}'\left(\mathcal{O}\right)$. This shows
that $\mathcal{F}$ is well-defined.

In order to show continuity of $\mathcal{F}$, let $\iota_{g}:\mathcal{D}'\left(\mathcal{O}\right)\rightarrow\mathbb{C},\varphi\mapsto\varphi\left(g\right)$
be the evaluation map (for some $g\in\mathcal{D}\left(\mathcal{O}\right)$).
Then we have 
\[
\left(\iota_{g}\circ\mathcal{F}\right)\left(f\right)=\iota_{g}\left(\mathcal{F}f\right)=\left(\mathcal{F}f\right)\left(g\right)=f\left(\mathcal{F}^{-1}\overline{g}\right)=\iota_{\mathcal{F}^{-1}\overline{g}}\left(f\right)\qquad\text{ for all }f\in\left(\mathcal{H}_{w}^{1}\right)^{\neg},
\]
where $\iota_{\mathcal{F}^{-1}\overline{g}}$ denotes the evaluation
map on $\left(\mathcal{H}_{w}^{1}\right)^{\neg}$. This map is continuous
on $\left(\mathcal{H}_{w}^{1}\right)^{\neg}$ by the definition of
the weak-$\ast$-topology. Hence we see that $\iota_{g}\circ\mathcal{F}=\iota_{\mathcal{F}^{-1}\overline{g}}$
is continuous. As the topology on $\mathcal{D}'\left(\mathcal{O}\right)$
is induced by the family of evaluation maps, this shows the continuity
of $\mathcal{F}:\left(\mathcal{H}_{w}^{1}\right)^{\neg}\rightarrow\mathcal{D}'\left(\mathcal{O}\right)$
with respect to the weak-$\ast$-topology on $\left(\mathcal{H}_{w}^{1}\right)^{\neg}$.

In order to show the injectivity of $\mathcal{F}$, let $\psi\in\mathcal{S}\left(\mathbb{R}^{d}\right)\setminus\left\{ 0\right\} $
with $\widehat{\psi}\in\mathcal{D}\left(\mathcal{O}\right)$ be arbitrary
and let $f\in\left(\mathcal{H}_{w}^{1}\right)^{\neg}$ with $\mathcal{F}f=0$.
Note that $\pi\left(x,h\right)\psi\in\mathcal{S}\left(\mathbb{R}^{d}\right)$
is a Schwartz function whose Fourier transform has compact support
\[
\text{supp}\left(\mathcal{F}\left(\pi\left(x,h\right)\psi\right)\right)\overset{\text{Eq. }\eqref{eq:QausiRegulaereDarstellungAufFourierSeite}}{=}\text{supp}\left(M_{-x}D_{h}\widehat{\psi}\right)\subset h^{-T}\left(\text{supp}\left(\smash{\widehat{\psi}}\right)\right)\subset\mathcal{O}.
\]
This shows $\overline{\mathcal{F}\left(\pi\left(x,h\right)\psi\right)}\in\mathcal{D}\left(\mathcal{O}\right)$
and thus 
\begin{eqnarray*}
\left(W_{\psi}f\right)\left(x,h\right) & = & \left\langle \pi\left(x,h\right)\psi,f\right\rangle _{\text{anti}}=f\left(\pi\left(x,h\right)\psi\right)\\
 & = & f\left(\mathcal{F}^{-1}\overline{\overline{\mathcal{F}\left(\pi\left(x,h\right)\psi\right)}}\right)\\
 & \overset{\text{Eq. }\eqref{eq:FouriertrafoAufFeichtingerReservoir}}{=} & \left(\mathcal{F}f\right)\left(\overline{\mathcal{F}\left(\pi\left(x,h\right)\psi\right)}\right)\\
 & \overset{\mathcal{F}f=0}{=} & 0,
\end{eqnarray*}
i.e. $W_{\psi}f\equiv0$. But \cite[Theorem 4.1]{FeichtingerCoorbit1}
shows that $W_{\psi}:\left(\mathcal{H}_{w}^{1}\right)^{\neg}\rightarrow L_{1/w}^{\infty}\left(G\right)$
is injective (note that we have $\psi\in\mathcal{B}_{w}\setminus\left\{ 0\right\} \subset\mathcal{A}_{w}\setminus\left\{ 0\right\} $
by Theorem \ref{thm:ZulaessigkeitVonbandbeschraenktenFunktionen}),
which implies $f=0$.
\end{proof}
Instead of applying the Fourier transform to $f\in\left(\mathcal{H}_{w}^{1}\right)^{\neg}$
in order to yield $\mathcal{F}f\in\mathcal{D}'\left(\mathcal{O}\right)$,
we can also ``pass on'' the application of the Fourier transform
to the space on which $f$ is defined. This is described in the next
corollary. We will see in section \ref{sec:InverseFourierTrafoStetig}
that the reservoir $\left(\mathcal{F}\left(\mathcal{D}\left(\mathcal{O}\right)\right)\right)'$
that is used in this corollary is a very natural alternative ``reservoir''
for the definition of coorbit spaces.
\begin{cor}
\label{cor:FourierTrafoAufRaumAbgewaelzt}Let $w:H\rightarrow\left(0,\infty\right)$
be locally bounded. Then the map
\[
\Theta:\left(\mathcal{H}_{w}^{1}\right)^{\neg}\rightarrow\left(\mathcal{F}\left(\mathcal{D}\left(\mathcal{O}\right)\right)\right)',f\mapsto\left(\varphi\mapsto f\left(\overline{\varphi}\right)\right)
\]
is a well-defined, injective linear map that is continuous with
respect to the weak-$\ast$-topology on $\left(\mathcal{H}_{w}^{1}\right)^{\neg}$.

Here, the space $\mathcal{F}\left(\mathcal{D}\left(\mathcal{O}\right)\right)$
is endowed with the unique topology that makes the Fourier transform
$\mathcal{F}:\mathcal{D}\left(\mathcal{O}\right)\rightarrow\mathcal{F}\left(\mathcal{D}\left(\mathcal{O}\right)\right)\leq\mathcal{S}\left(\mathbb{R}^{d}\right)$
a homeomorphism and the dual space $\left(\mathcal{F}\left(\mathcal{D}\left(\mathcal{O}\right)\right)\right)'$
is equipped with the weak-$\ast$-topology.

With the definition of the Fourier transform on $\left(\mathcal{H}_{w}^{1}\right)^{\neg}$
of corollary \ref{cor:FourierTrafoAufFeichtingerReservoir}, we have
\begin{equation}
\mathcal{F}f=\left(\Theta f\right)\circ\mathcal{F}\qquad\text{ for all }f\in\left(\mathcal{H}_{w}^{1}\right)^{\neg}.\label{eq:FourierTrafoVertauschtMitTheta}
\end{equation}
\end{cor}
\begin{proof}
First note that we have $\overline{\mathcal{F}\varphi}=\mathcal{F}^{-1}\overline{\varphi}$
for $\varphi\in\mathcal{D}\left(\mathcal{O}\right)$. For $\psi=\mathcal{F}\varphi\in\mathcal{F}\left(\mathcal{D}\left(\mathcal{O}\right)\right)$,
this shows $\overline{\psi}=\mathcal{F}^{-1}\overline{\varphi}\in\mathcal{H}_{w}^{1}$
by Theorem \ref{thm:ZulaessigkeitVonbandbeschraenktenFunktionen}.
Here, we used that $\overline{\varphi}\in\mathcal{D}\left(\mathcal{O}\right)$
holds as well, i.e. that $\mathcal{D}\left(\mathcal{O}\right)$ is
invariant under conjugation. In summary, this entails that 
\[
\Theta f:\mathcal{F}\left(\mathcal{D}\left(\mathcal{O}\right)\right)\rightarrow\mathbb{C}
\]
is a well-defined linear map.

For $\varphi\in\mathcal{D}\left(\mathcal{O}\right)$ we have
\[
\left(\left(\Theta f\right)\circ\mathcal{F}\right)\left(\varphi\right)=\left(\Theta f\right)\left(\widehat{\varphi}\right)=f\left(\overline{\widehat{\varphi}}\right)=f\left(\mathcal{F}^{-1}\overline{\varphi}\right)\overset{\text{Eq. }\eqref{eq:FouriertrafoAufFeichtingerReservoir}}{=}\left(\mathcal{F}f\right)\left(\varphi\right),
\]
which proves equation \eqref{eq:FourierTrafoVertauschtMitTheta}.
Furthermore, corollary \ref{cor:FourierTrafoAufFeichtingerReservoir}
implies that the right hand side is a continuous linear function of
$\varphi\in\mathcal{D}\left(\mathcal{O}\right)$. The definition of
the topology on $\mathcal{F}\left(\mathcal{D}\left(\mathcal{O}\right)\right)$
thus implies $\Theta f\in\left(\mathcal{F}\left(\mathcal{D}\left(\mathcal{O}\right)\right)\right)'$.

In order to show the injectivity of $\Theta$, assume $\Theta f=0$
for some $f\in\left(\mathcal{H}_{w}^{1}\right)^{\neg}$. Equation
\eqref{eq:FourierTrafoVertauschtMitTheta} then yields $\mathcal{F}f=0$,
which implies $f\equiv0$ by Corollary \ref{cor:FourierTrafoAufFeichtingerReservoir}.

Finally, let $\psi=\mathcal{F}\varphi\in\mathcal{F}\left(\mathcal{D}\left(\mathcal{O}\right)\right)$
be arbitrary. As in the proof of corollary \ref{cor:FourierTrafoAufFeichtingerReservoir},
we see that the evaluation map $\iota_{\psi}:\left(\mathcal{F}\left(\mathcal{D}\left(\mathcal{O}\right)\right)\right)'\rightarrow\mathbb{C},f\mapsto f\left(\psi\right)$
satisfies
\[
\iota_{\psi}\left(\Theta f\right)=\left(\Theta f\right)\left(\mathcal{F}\varphi\right)\overset{\text{Eq. }\eqref{eq:FourierTrafoVertauschtMitTheta}}{=}\left(\mathcal{F}f\right)\left(\varphi\right)=f\left(\mathcal{F}^{-1}\overline{\varphi}\right)=\iota_{\mathcal{F}^{-1}\overline{\varphi}}\left(f\right),
\]
where the right hand side is a continuous function of $f\in\left(\mathcal{H}_{w}^{1}\right)^{\neg}$
with respect to the weak-$\ast$-topology. This proves the claimed continuity
of $\Theta$.
\end{proof}

\section{Construction of an induced covering and definition of the corresponding
decomposition space}

\label{sec:InducedCovering}In this section we will show how to obtain
the induced covering $\mathcal{Q}$ mentioned in the introduction
and we will prove that our construction indeed yields an admissible
covering (cf. Definition \ref{def:AdmissibleCoveringUndCluster} below).
The idea for the construction of $\mathcal{Q}$ is the following:
Choose a (necessarily countable) well-spread family $\left(h_{i}\right)_{i\in I}$
in $H$. For precompact $Q\subset\mathcal{O}$ with $\overline{Q}\subset\mathcal{O}$
and $\mathcal{O}\mathcal{=}\bigcup_{i\in I}h_{i}^{-T}Q$ we then define
$\mathcal{Q}:=\left(Q_{i}\right)_{i\in I}:=\left(h_{i}^{-T}Q\right)_{i\in I}$.
It is worth noting that this induced covering is of a very simple
form in which every set $Q_{i}$ is a linear image of a fixed set
$Q$. We will also see that the covering is well behaved in the sense
that there is a constant $C>0$ such that $\left\Vert h_{i}^{-1}h_{j}\right\Vert \leq C$
holds for all $i,j\in I$ with $Q_{i}\cap Q_{j}\neq\emptyset$.

Finally, we will state the exact definition of the space $\mathcal{D}\left(\mathcal{Q},L^{p},\ell_{u}^{q}\right)$
as considered in this paper. Our definition is slightly different
than the one in \cite[Definition 3]{BorupNielsenDecomposition}.

Before we show that our construction of the induced covering indeed
yields an admissible covering, we recall the following fundamental
definitions from \cite[Definition 2.1 and Definition 2.3]{DecompositionSpaces1}:
\begin{defn}
\label{def:AdmissibleCoveringUndCluster}(cf. \cite[Definition 2.1 and Definition 2.3]{DecompositionSpaces1})

Let $X\neq\emptyset$ be a set and assume that $\mathcal{Q}=\left(Q_{i}\right)_{i\in I}$
is a family of subsets of $X$. For a subset $J\subset I$ we then
define the \textbf{(index)-cluster} \textbf{of $J$} as 
\[
J^{\ast}:=\left\{ i\in I\with\exists j\in J:\: Q_{i}\cap Q_{j}\neq\emptyset\right\} .
\]
Inductively, we define $J^{0\ast}:=J$ and $J^{\left(n+1\right)\ast}:=\left(J^{n\ast}\right)^{\ast}$
for $n\in\mathbb{N}_{0}$. For convenience, we also set $i^{n\ast}:=\left\{ i\right\} ^{n\ast}$
and $i^{\ast}:=\left\{ i\right\} ^{\ast}$ for $i\in I$. Furthermore,
for any subset $J\subset I$ we define $Q_{J}:=\bigcup_{i\in J}Q_{i}$.
With this notation we introduce the convenient shortcuts $Q_{i}^{k\ast}:=Q_{i^{k\ast}}$
and $Q_{i}^{\ast}:=Q_{i^{\ast}}$ for $i\in I$ and $k\in\mathbb{N}_{0}$.

We say that $\mathcal{Q}$ is an \textbf{admissible covering} of $X$,
if the following holds
\begin{enumerate}
\item $X=\bigcup_{i\in I}Q_{i}$ (i.e. $\mathcal{Q}$ is a covering of $X$)
and 
\item There exists $n_{0}\in\mathbb{N}$ with $\left|i^{\ast}\right|\leq n_{0}$
for all $i\in I$. 
\end{enumerate}
\end{defn}
In \cite[Definiton 7]{BorupNielsenDecomposition}, Borup and Nielsen
specialized this notion to the concept of a so-called \textbf{structured
admissible covering}. They only considered coverings of the whole
euclidean space $\mathbb{R}^{d}$. In the next definition we generalize
this to coverings of arbitrary open subsets $\emptyset\neq U\subset\mathbb{R}^{d}$.
\begin{defn}
\label{def:BorupUndNielsenCovering}(based upon \cite[Definition 7]{BorupNielsenDecomposition})

Let $\emptyset\neq U\subset\mathbb{R}^{d}$ be open and let $I\neq\emptyset$
be a countable index-set. Furthermore assume that $\left(T_{i}\right)_{i\in I}$
and $\left(b_{i}\right)_{i\in I}$ are families of invertible linear
transformations $T_{i}\in\text{GL}\left(\mathbb{R}^{d}\right)$ and
of translations $b_{i}\in\mathbb{R}^{d}$, respectively.

Let $P,Q\subset\mathbb{R}^{d}$ be precompact open subsets with $\overline{P}\subset Q$.
We then say that the family $\mathcal{Q}:=\left(Q_{i}\right)_{i\in I}:=\left(T_{i}Q+b_{i}\right)_{i\in I}$
is a \textbf{structured admissible covering (of $U$)}, if 
\begin{enumerate}
\item $\mathcal{Q}$ and $\mathcal{P}:=\left(T_{i}P+b_{i}\right)_{i\in I}$
are admissible coverings%
\footnote{This implies in particular that we have $Q_{i}\subset U$ for all
$i\in I$.%
} of $U$ and 
\item there is a constant $C>0$ such that $\left\Vert T_{i}^{-1}T_{j}\right\Vert \leq C$
holds for all $i,j\in I$ satisfying $Q_{i}\cap Q_{j}\neq\emptyset$. 
\end{enumerate}
\end{defn}
Borup and Nielsen then showed (cf. \cite[Proposition 1]{BorupNielsenDecomposition})
that every structured admissible covering admits a so-called \textbf{bounded
admissible partition of unity} \textbf{(BAPU)} which can then be used
in order to define the decomposition spaces $\mathcal{D}\left(\mathcal{Q},L^{p},\ell_{u}^{q}\right)$.
More precisely, the conditions for a BAPU as used in this paper are
as follows: 
\begin{defn}
\label{def:BAPUDefinition}(cf. \cite[Definition 2.2]{DecompositionSpaces1}
and \cite[Definition 2]{BorupNielsenDecomposition})

Let $\emptyset\neq U\subset\mathbb{R}^{d}$ be open and let $\mathcal{Q}=\left(Q_{i}\right)_{i\in I}$
be an admissible covering of $U$. A family $\left(\varphi_{i}\right)_{i\in I}$
of functions is called a \textbf{bounded admissible partition of unity}
(BAPU) subordinate to $\mathcal{Q}$, if 
\begin{enumerate}
\item $\varphi_{i}\in\mathcal{D}\left(U\right)$ for all $i\in I$, 
\item $\sum_{i=1}^{\infty}\varphi_{i}\left(x\right)=1$ for all $x\in U$
(i.e. $\left(\varphi_{i}\right)_{i\in I}$ is a partition of unity
on $U$), 
\item $\varphi_{i}\left(x\right)=0$ for all $x\in U\setminus Q_{i}$ for
all $i\in I$, 
\item $\sup_{i\in I}\left\Vert \mathcal{F}^{-1}\varphi_{i}\right\Vert _{L^{1}\left(\mathbb{R}^{d}\right)}<\infty$. 
\end{enumerate}
\end{defn}
Note that Borup and Nielsen even require $\sup_{i\in I}\left|\det\left(T_{i}\right)\right|^{\frac{1}{p}-1}\left\Vert \mathcal{F}^{-1}\varphi_{i}\right\Vert _{L^{p}\left(\mathbb{R}^{d}\right)}<\infty$
for all $p\in\left(0,1\right]$ where each $Q_{i}$ is given by $Q_{i}=T_{i}Q+b_{i}$.
This stronger condition is necessary to ensure well-definedness of
the decomposition space $\mathcal{D}\left(\mathcal{Q},L^{p},\ell_{u}^{q}\right)$
in the Quasi-Banach regime $p\in\left(0,1\right)$. In this paper
we will only consider the range $p\in\left[1,\infty\right]$.

Mainly as a simplification of notation, we introduce the term of a
\textbf{decomposition covering}.
\begin{defn}
Let $\emptyset\neq U\subset\mathbb{R}^{d}$ be an open set. A family
$\mathcal{Q}=\left(Q_{i}\right)_{i\in I}$ of subsets of $U$ is called
a \textbf{decomposition covering}, if 
\begin{enumerate}
\item $\mathcal{Q}$ is an admissible covering of $U$, 
\item $Q_{i}\neq\emptyset$ for all $i\in I$ and 
\item there exists a BAPU $\left(\varphi_{i}\right)_{i\in I}$ subordinate
to $\mathcal{Q}$. 
\end{enumerate}
\end{defn}
An easy adaptation of the proof of \cite[Proposition 1]{BorupNielsenDecomposition}
(we allow any open set $\emptyset\neq U\subset\mathbb{R}^{d}$, whereas
Borup and Nielsen only consider coverings of the whole euclidean space
$\mathbb{R}^{d}$) then yields the following:
\begin{thm}
\label{thm:StructuredCoveringsSindDekompositionsCovering}Let $\emptyset\neq U\subset\mathbb{R}^{d}$
be an open set. Then every structured admissible covering $\mathcal{Q}=\left(Q_{i}\right)_{i\in I}$
of $U$ is a decomposition covering of $U$.
\end{thm}
We next recall density and discreteness properties for subsets of
topological groups.
\begin{defn}
\label{defn:dense_wellspread} Let $G$ denote an arbitrary locally
compact group, and $(g_{j})_{j\in J}\subset G$ a family of group
elements. Let $U\subset G$ denote a neighborhood of the identity.
The family $(g_{j})_{j\in J}$ is called \textbf{$U$-dense} if we
have $G=\bigcup_{j\in J}g_{j}U$ and \textbf{$U$-discrete} if for
all distinct $j,j'\in J$ the equality $g_{j}U\cap g_{j'}U=\emptyset$
holds. Furthermore, $\left(g_{j}\right)_{j\in J}$ is called \textbf{separated}
if it is $U$-discrete with respect to some neighborhood $U$ of the
identity. The family is called \textbf{well-spread}, if it is separated
and $U$-dense, for a relatively compact neighborhood $U$ of the
identity. Finally, we call $(g_{j})_{j\in J}$ \textbf{relatively
separated} if it is the union of finitely many separated sets.
\end{defn}
Using the fact that $p_{\xi_{0}}$ is a proper map (cf. Lemma \ref{lem:OrbitProjektionIstProper}),
we can prove the following fundamental result which shows that elements
of $H$ that are projected ``close to each other'' by $p_{\xi_{0}}$
are already close in $H$. This property will ensure that the induced
covering is admissible.
\begin{lem}
\label{lem:AdmissibilityVorbereitung}Let $\emptyset\neq K_{1},K_{2}\subset\mathcal{O}$
be compact. Then there is a compact set $L\subset H$ such that $g\in hL$
holds for all $g,h\in H$ satisfying $h^{-T}K_{1}\cap g^{-T}K_{2}\neq\emptyset$.

Now choose a relatively compact unit-neighborhood $V\subset H$ and
let $(h_{i})_{i\in I}\subset H$ denote a $V$-separated family. Given
$h\in H$, let 
\[
I_{h}:=\left\{ i\in I\with h^{-T}K_{1}\cap h_{i}^{-T}K_{2}\neq\emptyset\right\} .
\]
Then there exists a constant $C=C\left(K_{1},K_{2},V\right)$ such
that 
\[
|I_{h}|\le C\qquad\text{ holds for all }h\in H.
\]
\end{lem}
\begin{proof}
Choose 
\[
L=L\left(K_{1},K_{2}\right):=\left(p_{\xi_{0}}^{-1}\left(K_{1}\right)\right)^{-1}\cdot H_{\xi_{0}}\cdot p_{\xi_{0}}^{-1}\left(K_{2}\right).
\]
Lemma \ref{lem:OrbitProjektionIstProper} shows that $L\subset H$
is compact.

Let $g,h\in G$ with $h^{-T}K_{1}\cap g^{-T}K_{2}\neq\emptyset$.
This yields suitable $k_{1}=h_{1}^{T}\xi_{0}\in K_{1}\subset\mathcal{O}$
and $k_{2}=h_{2}^{T}\xi_{0}\in K_{2}\subset\mathcal{O}$ satisfying
$h^{-T}h_{1}^{T}\xi_{0}=g^{-T}h_{2}^{T}\xi_{0}$. We conclude 
\[
\left(h_{1}h^{-1}gh_{2}^{-1}\right)^{T}\xi_{0}=\xi_{0}.
\]
In other words, we have $h_{1}h^{-1}gh_{2}^{-1}\in H_{\xi_{0}}$ and
thus 
\[
g\in hh_{1}^{-1}H_{\xi_{0}}h_{2}\subset h\cdot\left(p_{\xi_{0}}^{-1}\left(K_{1}\right)\right)^{-1}\cdot H_{\xi_{0}}\cdot p_{\xi_{0}}^{-1}\left(K_{2}\right)=h\cdot L.
\]

Now let $\left(h_{i}\right)_{i\in I}$ be $V$-separated. For $i\in I_{h}$
we have $h^{-T}K_{1}\cap h_{i}^{-T}K_{2}\neq\emptyset$ and thus $h_{i}\in hL$
as shown above. But this implies $h_{i}V^{\circ}\subset h_{i}V\subset hL\overline{V}$.
Thus, $\left(h_{i}V^{\circ}\right)_{i\in I_{h}}$ is a pairwise disjoint
collection of subsets of $hL\overline{V}$. We thus get, for every
finite subset $J\subset I_{h}$, 
\[
\left|J\right|=\frac{\sum_{i\in J}\mu_{H}\left(h_{i}V^{\circ}\right)}{\mu_{H}\left(V^{\circ}\right)}=\mu_{H}\left(\biguplus_{i\in J}h_{i}V^{\circ}\right)\bigg/\mu_{H}\left(V^{\circ}\right)\leq\frac{\mu_{H}\left(hL\overline{V}\right)}{\mu_{H}\left(V^{\circ}\right)}=\frac{\mu_{H}\left(L\overline{V}\right)}{\mu_{H}\left(V^{\circ}\right)},
\]
where we used the left-invariance of $\mu_{H}$ in the first and last
step. Noting that the right hand side is independent of $h$ completes
the proof.
\end{proof}
As a last preparation, we need the following well-known existence
result for countable, ``well-spread'' families in $H$. It follows
by choosing (using Zorn's Lemma) a $V$-discrete subset of $H$ that
is maximal with respect to inclusion.
\begin{lem}
\label{lem:WellSpreadExistenz}Let $U,V\subset H$ be neighborhoods
of the identity such that $VV^{-1}\subset U$ holds. Then there exists
a family $\left(h_{i}\right)_{i\in I}$ of elements of $H$ such that
we have $G=\bigcup_{i\in I}h_{i}U$ and so that the family of sets
$\left(h_{i}V\right)_{i\in I}$ is pairwise disjoint. Every such family
is necessarily countably infinite.

In particular, there exists a countably infinite well-spread family
$\left(h_{i}\right)_{i\in I}$ in $H$.
\end{lem}
We now give the construction of the induced covering. It is worth
noting that there is still some freedom in choosing that covering.
As we will see below, the coorbit space and the decomposition space
will be isomorphic under the Fourier transformation for \emph{every}
such choice.
\begin{thm}
\label{thm:InduzierteUeberdeckungKonstruktion}Let $\left(h_{i}\right)_{i\in I}$
be well-spread in $H$. For every precompact set $Q\subset\mathbb{R}^{d}$
that satisfies $\overline{Q}\subset\mathcal{O}$ and $\mathcal{O}=\bigcup_{i\in I}h_{i}^{-T}Q$
we say that the covering $\mathcal{Q}=\left(h_{i}^{-T}Q\right)_{i\in I}$
of $\mathcal{O}$ is a \textbf{covering of $\mathcal{O}$ induced
by $H$}. Such a set $Q$ exists for every choice of the family $\left(h_{i}\right)_{i\in I}$.
In particular, one can choose $Q=U^{-T}\xi_{0}$, where $U\subset H$
is a precompact set satisfying $H=\bigcup_{i\in I}h_{i}U$.

Every such covering is admissible. Furthermore there is a constant
$C=C\left(\left(h_{i}\right)_{i\in I},Q\right)>0$ such that $\left\Vert h_{i}^{T}h_{j}^{-T}\right\Vert \leq C$
holds for all $i\in I$ and $j\in i^{\ast}$, where the cluster is
formed with respect to $\mathcal{Q}$.

If $Q\subset\mathbb{R}^{d}$ is open and precompact with $\overline{Q}\subset\mathcal{O}$
and if there is some open set $P\subset\mathbb{R}^{d}$ satisfying $\overline{P}\subset Q$
and $\mathcal{O}=\bigcup_{i\in I}h_{i}^{-T}P$ then $\mathcal{Q}$
is a structured admissible covering of $\mathcal{O}$ and in particular
a decomposition covering.\end{thm}
\begin{proof}
We first show the existence of $Q\subset\mathbb{R}^{d}$ with the
stated properties. By assumption on $\left(h_{i}\right)_{i\in I}$,
there is some precompact set $U\subset H$ such that $H=\bigcup_{i\in I}h_{i}U$.
This means 
\[
\mathcal{O}=H^{T}\xi_{0}=H^{-T}\xi_{0}=\bigcup_{i\in I}h_{i}^{-T}U^{-T}\xi_{0}.
\]
This implies that the choice $Q=U^{-T}\xi_{0}\subset\mathcal{O}$
guarantees the stated properties. Here we note that $\overline{U}^{-T}\xi_{0}\subset\mathcal{O}$
is compact (as a continuous image of the compact set $\overline{U}\subset H$).

Now let $\mathcal{Q}=\left(h_{i}^{-T}Q\right)_{i\in I}$ be a covering
of $\mathcal{O}$ induced by $H$. Let $V\subset H$ be a precompact
unit neighborhood such that $\left(h_{i}\right)_{i\in I}$ is $V$-discrete.
Set $K:=K_{1}:=K_{2}:=\overline{Q}\subset\mathcal{O}$ and choose
the compact set $L=L\left(K_{1},K_{2}\right)=L\left(Q\right)\subset H$
and the constant $C=C\left(K_{1},K_{2},V\right)>0$ as in Lemma \ref{lem:AdmissibilityVorbereitung}.
For $i\in I$ and $j\in i^{\ast}$ we then have $\emptyset\neq Q_{i}\cap Q_{j}\subset h_{i}^{-T}\overline{Q}\cap h_{j}^{-T}\overline{Q}$
so that Lemma \ref{lem:AdmissibilityVorbereitung} implies $h_{j}\in h_{i}L$
and thus $\left\Vert h_{i}^{T}h_{j}^{-T}\right\Vert =\left\Vert h_{j}^{-1}h_{i}\right\Vert \leq\max_{g\in L^{-1}}\left\Vert g\right\Vert $.
Furthermore, we have, in the notation of Lemma \ref{lem:AdmissibilityVorbereitung}:
\[
i^{\ast}=\left\{ j\in I\with h_{j}^{-T}Q\cap h_{i}^{-T}Q\neq\emptyset\right\} =I_{h_{i}}
\]
and thus $\left|i^{\ast}\right|=\left|I_{h_{i}}\right|\leq C$. This
shows that $\mathcal{Q}$ is admissible.

Finally, assume that $Q\subset\mathbb{R}^{d}$ is open and precompact
with $\overline{Q}\subset\mathcal{O}$ and that there is an open
set $P\subset\mathbb{R}^{d}$ satisfying $\overline{P}\subset Q$ and
$\mathcal{O}=\bigcup_{i\in I}h_{i}^{-T}P$. This yields $\mathcal{O}=\bigcup_{i\in I}h_{i}^{-T}P\subset\bigcup_{i\in I}h_{i}^{-T}Q \subset \mathcal{O}$.
Hence the above implies that $\mathcal{Q}$ and $\mathcal{P}=\left(h_{i}^{-T}P\right)_{i\in I}$
are both admissible coverings of $\mathcal{O}$. The estimate $\left\Vert h_{i}^{T}h_{j}^{-T}\right\Vert =\left\Vert \left(h_{i}^{-T}\right)^{-1}h_{j}^{-T}\right\Vert \leq C$
for all $i\in I$ and $j\in I$ with $Q_{i}\cap Q_{j}\neq\emptyset$
shown above then completes the proof of the fact that $\mathcal{Q}$
is a structured admissible covering. Theorem \ref{thm:StructuredCoveringsSindDekompositionsCovering}
then shows that $\mathcal{Q}$ is a decomposition covering.
\end{proof}
We next introduce our notion of decomposition spaces. Note that we
somewhat extend the definition of Borup and Nielsen in \cite{BorupNielsenDecomposition};
they only consider coverings of all of $\mathbb{R}^{d}$, whereas
the decomposition spaces that we consider arise from a covering of
the open dual orbit, which is always a proper subset of $\mathbb{R}^{d}$.
\begin{defn}
\label{def:DecompositionSpace}(cf. \cite[Definitions 2.4 and 3.1]{DecompositionSpaces1}
and \cite[Definition 3]{BorupNielsenDecomposition})

Let $\emptyset\neq U\subset\mathbb{R}^{d}$ be an open subset. Furthermore
assume that $\mathcal{Q}=\left(Q_{i}\right)_{i\in I}$ is a decomposition
covering of $U$ with BAPU $\left(\varphi_{i}\right)_{i\in I}$. Let
$u:I\rightarrow\left(0,\infty\right)$ be a \textbf{$\mathcal{Q}$-moderate
weight}, that is there exists some $C>0$ satisfying $u\left(i\right)\leq C\cdot u\left(j\right)$
for all $i\in I$ and all $j\in i^{\ast}$.

Let $p,q\in\left[1,\infty\right]$. For $f\in\mathcal{D}'\left(U\right)$
we define 
\[
\left\Vert f\right\Vert _{\mathcal{D}\left(\mathcal{Q},L^{p},\ell_{u}^{q}\right)}:=\left\Vert \left(u\left(i\right)\cdot\left\Vert \mathcal{F}^{-1}\left(\varphi_{i}f\right)\right\Vert _{L^{p}\left(\mathbb{R}^{d}\right)}\right)_{i\in I}\right\Vert _{\ell^{q}\left(I\right)}\in\left[0,\infty\right],
\]
where we use the convention that for a family $(c_{i})_{i\in I}$
with $c_{i}\in\left[0,\infty\right]$, the $\ell^{q}$-norm is infinite
as soon as one of the $c_{i}$ equals infinity.

Finally, we define the \textbf{decomposition space with respect to
the covering $\mathcal{Q}$ and the weight $u$ with integrability
exponents $p,q$} as 
\[
\mathcal{D}\left(\mathcal{Q},L^{p},\ell_{u}^{q}\right):=\left\{ f\in\mathcal{D}'\left(U\right)\with\left\Vert f\right\Vert _{\mathcal{D}\left(\mathcal{Q},L^{p},\ell_{u}^{q}\right)}<\infty\right\} .
\]
\end{defn}
\begin{rem*}
A few remarks pertaining this definition are in order:
\begin{enumerate}
\item We first note that $\varphi_{i}f\in\mathcal{D}'\left(U\right)$ is
a distribution with compact support, because of $\varphi_{i}\in\mathcal{D}\left(U\right)$.
By \cite[Example 7.12(a)]{RudinFA}, $\varphi_{i}f$ is thus a tempered
distribution so that $\mathcal{F}^{-1}\left(\varphi_{i}f\right)\in\mathcal{S}'\left(\mathbb{R}^{d}\right)$
makes sense. Furthermore, \cite[Theorem 7.23]{RudinFA} shows that
$\mathcal{F}\left(\varphi_{i}f\right)$ (and hence also $\mathcal{F}^{-1}\left(\varphi_{i}f\right)$)
is a smooth function of polynomial growth. In particular, $\mathcal{F}^{-1}\left(\varphi_{i}f\right)$
is pointwise defined by 
\begin{equation}
\left(\mathcal{F}^{-1}\left(\varphi_{i}f\right)\right)\left(x\right)=\left(\varphi_{i}f\right)\left(e_{x}\right)=f\left(\varphi_{i}e_{x}\right)\label{eq:FourierPunktweise}
\end{equation}
with $e_{x}:\mathbb{R}^{d}\rightarrow\mathbb{C},\xi\mapsto e^{2\pi i\left\langle x,\xi\right\rangle }$.
Thus the expression $\left\Vert \mathcal{F}^{-1}\left(\varphi_{i}f\right)\right\Vert _{L^{p}\left(\mathbb{R}^{d}\right)}\in\left[0,\infty\right]$
makes sense.
\item In the following we will use the notations $u_{i}:=u\left(i\right)$
and  $\left\Vert \left(x_{i}\right)_{i\in I}\right\Vert _{\ell_{u}^{q}}:=\left\Vert \left(u_{i}\cdot x_{i}\right)_{i\in I}\right\Vert _{\ell^{q}}$.
As $\ell^{q}$ is permutation-invariant and because $u$ is $\mathcal{Q}$-moderate,
\cite[Lemma 3.2]{DecompositionSpaces1} shows that $\ell_{u}^{q}$
is a \textbf{$\mathcal{Q}$-regular BK-space} (see \cite[Definition 2.5]{DecompositionSpaces1}).
The proof of \cite[Theorem 2.3(B)]{DecompositionSpaces1} then shows
that $\mathcal{D}\left(\mathcal{Q},L^{p},\ell_{u}^{q}\right)$ is
independent of the particular choice of the BAPU $\left(\varphi_{i}\right)_{i\in I}$
with equivalent norms for each choice.

For the convenience of the reader (and because our definition differs slightly from the one in \cite{DecompositionSpaces1}), we give a sketch of
the argument:
\begin{enumerate}
\item If $\left(\psi_{i}\right)_{i\in I}$ is another BAPU for $\mathcal{Q}$,
we have $\psi_{i}=\psi_{i}\varphi_{i}^{\ast}$ for all $i \in I$. Indeed, for $j\in I\setminus i^{\ast}$
we have $\varphi_{j}\equiv0$ on $Q_{i}$, so that $\sum_{i\in I}\varphi_{i}\equiv1$
on $\mathcal{O}\supset Q_{i}$ forces $\varphi_{i}^{\ast}\equiv1$
on $Q_{i}$. Now $\psi_{i}\equiv0$ on $\mathcal{O} \setminus Q_{i}$ implies $\psi_{i}=\psi_{i}\varphi_{i}^{\ast}$.
\item Using Young's inequality, this implies
\begin{align*}
\left\Vert \mathcal{F}^{-1}\left(\psi_{i}f\right)\right\Vert _{L^{p}\left(\mathbb{R}^{d}\right)} & =\left\Vert \mathcal{F}^{-1}\left(\psi_{i}\varphi_{i}^{\ast}f\right)\right\Vert _{L^{p}\left(\mathbb{R}^{d}\right)}\\
 & =\left\Vert \left(\mathcal{F}^{-1}\psi_{i}\right)\ast\mathcal{F}^{-1}\left(\varphi_{i}^{\ast}f\right)\right\Vert _{L^{p}\left(\mathbb{R}^{d}\right)}\\
 & \leq\left\Vert \mathcal{F}^{-1}\psi_{i}\right\Vert _{L^{1}\left(\mathbb{R}^{d}\right)}\cdot\left\Vert \mathcal{F}^{-1}\left(\varphi_{i}^{\ast}f\right)\right\Vert _{L^{p}\left(\mathbb{R}^{d}\right)}\\
 & \leq C\cdot\sum_{j\in i^{\ast}}\left\Vert \mathcal{F}^{-1}\left(\varphi_{j}f\right)\right\Vert _{L^{p}\left(\mathbb{R}^{d}\right)},
\end{align*}
where we used the property $\sup_{i\in I}\left\Vert \mathcal{F}^{-1}\psi_{i}\right\Vert _{L^{1}\left(\mathbb{R}^{d}\right)} < \infty$
of a $\mathcal{Q}$-BAPU in the last step.
\item The fact that $\ell_{u}^{q}$ is $\mathcal{Q}$-regular means by definition
that the map
\[
    \Gamma:\ell_{u}^{q}(I)\rightarrow\ell_{u}^{q}(I),\left(x_{i}\right)_{i\in I}\mapsto\bigg(\smash{\sum_{j\in i^{\ast}}x_{j}}\bigg)_{i\in I}
\]
is well-defined and bounded. Using the $\mathcal{Q}$-moderateness of $u$
and the fact that $\left|i^{\ast}\right|\leq C$ is uniformly bounded,
this is also easy to see directly.
\item Putting everything together, we arrive at
\begin{align*}
\left\Vert \left(\left\Vert \mathcal{F}^{-1}\left(\psi_{i}f\right)\right\Vert _{L^{p}\left(\mathbb{R}^{d}\right)}\right)_{i}\right\Vert _{\ell_{u}^{q}} & \leq C\cdot\left\Vert \Gamma\left(\left(\left\Vert \mathcal{F}^{-1}\left(\varphi_{i}f\right)\right\Vert _{L^{p}\left(\mathbb{R}^{d}\right)}\right)_{i}\right)\right\Vert _{\ell_{u}^{q}}\\
 & \leq C\cdot\left\Vert \Gamma\right\Vert \cdot\left\Vert \left(\left\Vert \mathcal{F}^{-1}\left(\varphi_{i}f\right)\right\Vert _{L^{p}\left(\mathbb{R}^{d}\right)}\right)\right\Vert _{\ell_{u}^{q}}.
\end{align*}
By symmetry, we also get the reverse inequality.
\end{enumerate}
\item %
It is worth noting that the exact construction used in \cite{DecompositionSpaces1}
would be to take $B=\mathcal{F}L^{p}\left(\mathbb{R}^{d}\right)$
and $A=\mathcal{F}L^{1}\left(\mathbb{R}^{d}\right)$ and then to define
the decomposition space as 
\[
\mathcal{D}\left(\mathcal{Q},L^{p},\ell_{u}^{q}\right)=\left\{ f\in\left(\mathcal{F}L^{1}\left(\mathbb{R}^{d}\right)\cap C_{c}\left(\mathbb{R}^{d}\right)\right)'\with\left\Vert f\right\Vert _{\mathcal{D}\left(\mathcal{Q},L^{p},\ell_{u}^{q}\right)}<\infty\right\} .
\]
Thus we are not exactly in this setting, as we use $\mathcal{D}'\left(U\right)$
as our reservoir instead of $\left(\mathcal{F}L^{1}\left(\mathbb{R}^{d}\right)\cap C_{c}\left(\mathbb{R}^{d}\right)\right)'$.
As seen above, the arguments used in \cite{DecompositionSpaces1}
nevertheless carry over to the present situation.
\item Analogous to the proof of \cite[Theorem 2.2(A)]{DecompositionSpaces1}
we can also show that $\mathcal{D}\left(\mathcal{Q},L^{p},\ell_{u}^{q}\right)$
is a Banach space that embeds continuously into $\mathcal{D}'\left(U\right)$.
For the convenience of the reader we again give a sketch of the proof:

\begin{enumerate}
\item The identity $\sum_{i\in I}\varphi_{i}\equiv1$ on $U$ shows that
$\left(\varphi_{i}\left(\mathbb{C}^{\ast}\right)\right)_{i\in I}$
covers $U$. Because of $\varphi_{i}\left(\mathbb{C}^{\ast}\right)\subset Q_{i}^{\circ}$,
this shows that $\left(Q_{i}^{\circ}\right)_{i\in I}$ is an open
cover of $U$, where $Q_{i}^{\circ}$ denotes the topological interior
of $Q_{i}$.
\item Let $K\subset U$ be compact and note that we have $K\subset\bigcup_{i\in I_{K}}Q_{i}^{\circ}$
for some finite set $I_{K}\subset I$. This easily entails $\varphi_{I_{K}}^{\ast}\equiv1$
on $K$.
\item For $f\in\mathcal{D}\left(\mathcal{Q},L^{p},\ell_{u}^{q}\right)$
and $\gamma\in\mathcal{D}_{K}\left(U\right)$ (i.e. $\gamma\in\mathcal{D}\left(U\right)$
with ${\rm supp}\left(\gamma\right)\subset K$), we thus get
\begin{align}
    \left|f\left(\gamma\right)\right| & =\left|f\bigg(\smash{\sum_{i\in I_{K}^{\ast}}\varphi_{i}\gamma}\bigg)\right|\leq\sum_{i\in I_{K}^{\ast}}\left|\left(\varphi_{i}f\right)\left(\gamma\right)\right|\nonumber \\
 & =\sum_{i\in I_{K}^{\ast}}\left|\left\langle \mathcal{F}^{-1}\left(\varphi_{i}f\right),\widehat{\gamma}\right\rangle _{\mathcal{S}',\mathcal{S}}\right|\nonumber \\
 & \leq\left\Vert \widehat{\gamma}\right\Vert _{L^{p'}\left(\mathbb{R}^{d}\right)}\cdot\sum_{i\in I_{K}^{\ast}}\left[\frac{1}{u_{i}}\cdot u_{i}\left\Vert \mathcal{F}^{-1}\left(\varphi_{i}f\right)\right\Vert _{L^{p}\left(\mathbb{R}^{d}\right)}\right]\nonumber \\
 & \leq\left\Vert \widehat{\gamma}\right\Vert _{L^{p'}\left(\mathbb{R}^{d}\right)}\cdot\bigg(\smash{\sum_{i\in I_{K}^{\ast}}\frac{1}{u_{i}}}\bigg)\cdot\left\Vert f\right\Vert _{\mathcal{D}\left(\mathcal{Q},L^{p},\ell_{u}^{q}\right)},\label{eq:DecompositionBettetInDistributionenEin}
\end{align}
where $p'\in\left[1,\infty\right]$ is conjugate to $p$.
\item The estimate \eqref{eq:DecompositionBettetInDistributionenEin} easily
yields $f_{n}\left(\gamma\right)\xrightarrow[n\rightarrow\infty]{}f\left(\gamma\right)$
for $f_{n}\xrightarrow[n\rightarrow\infty]{\mathcal{D}\left(\mathcal{Q},L^{p},\ell_{u}^{q}\right)}f$
and all $\gamma\in\mathcal{D}\left(U\right)$, i.e. $f_{n}\xrightarrow[n\rightarrow\infty]{}f$
in the weak-$\ast$-topology on $\mathcal{D}'\left(U\right)$.
\item If $\left(f_{n}\right)_{n\in\mathbb{N}}$ is Cauchy in $\mathcal{D}\left(\mathcal{Q},L^{p},\ell_{u}^{q}\right)$,
equation \eqref{eq:DecompositionBettetInDistributionenEin} shows
that $\left(f_{n}\left(\gamma\right)\right)_{n\in\mathbb{N}}$ is
Cauchy and hence convergent to some $f\left(\gamma\right)\in\mathbb{C}$
for every $\gamma\in\mathcal{D}\left(U\right)$. Now \cite[Theorem 6.17]{RudinFA}
implies that this yields a distribution $f\in\mathcal{D}'\left(U\right)$.
Using equation \eqref{eq:FourierPunktweise}, we derive (with $e_{x} : \mathbb{R}^{d} \rightarrow \mathbb{C}, \xi \mapsto e^{2\pi i \left\langle x, \xi \right\rangle}$)
\[
\qquad\qquad  \left(\mathcal{F}^{-1}\left(\varphi_{i}f\right)\right)\left(x\right)=f\left(\varphi_{i}e_{x}\right)=\lim_{n\rightarrow\infty}f_{n}\left(\varphi_{i}e_{x}\right)=\lim_{n\rightarrow\infty}\left(\mathcal{F}^{-1}\left(\varphi_{i}f_{n}\right)\right)\left(x\right).
\]
It is easy to see that $\left(\mathcal{F}^{-1}\left(\varphi_{i}f_{n}\right)\right)_{n\in\mathbb{N}}$
is Cauchy in $L^{p}\left(\mathbb{R}^{d}\right)$ for all $i\in I$. Together,
this yields 
\[
c_{i}^{\left(n\right)}:=\left\Vert \mathcal{F}^{-1}\left(\varphi_{i}\left(f-f_{n}\right)\right)\right\Vert _{L^{p}\left(\mathbb{R}^{d}\right)}\xrightarrow[n\rightarrow\infty]{}0.
\]

\item The (reversed) triangle inequality yields
\[
\left|c_{i}^{\left(n\right)}-c_{i}^{\left(m\right)}\right|\leq\left\Vert \mathcal{F}^{-1}\left(\varphi_{i}\left(f_{n}-f_{m}\right)\right)\right\Vert _{L^{p}\left(\mathbb{R}^{d}\right)}
\]
and thus
\begin{align*}
\left\Vert \left(c_{i}^{\left(n\right)}-c_{i}^{\left(m\right)}\right)_{i}\right\Vert _{\ell_{u}^{q}} & \leq\left\Vert \left(\left\Vert \mathcal{F}^{-1}\left(\varphi_{i}\left(f_{n}-f_{m}\right)\right)\right\Vert _{L^{p}\left(\mathbb{R}^{d}\right)}\right)_{i\in I}\right\Vert _{\ell_{u}^{q}}\\
 & =\left\Vert f_{n}-f_{m}\right\Vert _{\mathcal{D}\left(\mathcal{Q},L^{p},\ell_{u}^{q}\right)}\xrightarrow[n,m\rightarrow\infty]{}0,
\end{align*}
so that $\left(\left(c_{i}^{\left(n\right)}\right)_{i}\right)_{n}$
is Cauchy in $\ell_{u}^{q}$. As seen above, we have $c_{i}^{\left(n\right)}\xrightarrow[n\rightarrow\infty]{}0$
for all $i\in I$. Together, this implies $\left(c_{i}^{\left(n\right)}\right)_{i}\xrightarrow[n\rightarrow\infty]{\ell_{u}^{q}}0$,
which means nothing but $\left\Vert f_{n}-f\right\Vert _{\mathcal{D}\left(\mathcal{Q},L^{p},\ell_{u}^{q}\right)}\xrightarrow[n\rightarrow\infty]{}0$.
This shows that $\mathcal{D}\left(\mathcal{Q},L^{p},\ell_{u}^{q}\right)$
is complete.
\end{enumerate}
\item The completeness of $\mathcal{D}\left(\mathcal{Q},L^{p},\ell_{u}^{q}\right)$
is in contrast to \cite{BorupNielsenDecomposition}, where the authors
only consider the case $U=\mathbb{R}^{d}$ and then define the decomposition
space as 
\[
\mathcal{D}_{\mathcal{S}'}\left(\mathcal{Q},L^{p},\ell_{u}^{q}\right):=\left\{ f\in\mathcal{S}'\left(\mathbb{R}^{d}\right)\with\left\Vert f\right\Vert _{\mathcal{D}\left(\mathcal{Q},L^{p},\ell_{u}^{q}\right)}<\infty\right\} .
\]
With this definition, the decomposition space is in general \emph{not}
complete, as the following example shows. 
\end{enumerate}
\end{rem*}
\begin{example*}
In the following, we provide a specific example showing that $\mathcal{D}_{\mathcal{S}'}\left(\mathcal{Q},L^{p},\ell_{u}^{q}\right)$
as defined in the previous remark is in general \emph{not} complete.
Let $I:=\mathbb{Z}$, $T_{i}:=\text{id}_{\mathbb{R}}$ and $b_{i}:=i$
for $i\in\mathbb{Z}$. Furthermore, let $Q:=\left(-\frac{3}{4},\frac{3}{4}\right)$
and $P:=\left(-\frac{5}{8},\frac{5}{8}\right)$, as well as $Q_{i}:=T_{i}Q+b_{i}=\left(i-\frac{3}{4},i+\frac{3}{4}\right)$.
It is then easy to see that $\bigcup_{i\in I}\left(T_{i}P+b_{i}\right)=\mathbb{R}$
and that $x\in Q_{i}\cap Q_{j}\neq\emptyset$ implies 
\[
i-\frac{3}{4}<x<j+\frac{3}{4}
\]
and hence $i-j<\frac{6}{4}<2$. Because of $i-j\in\mathbb{Z}$ we
conclude $i-j\leq1$. By symmetry we get $\left|i-j\right|\leq1$
and thus $i^{\ast}\subset\left\{ i-1,i,i+1\right\} $. This shows
that $\mathcal{Q}=\left(Q_{i}\right)_{i\in I}$ is a structured admissible
covering of $\mathbb{R}$. Now consider the weight $u_{i}:=10^{-i}$
for $i\in\mathbb{Z}$ and note that because of the estimate 
\[
\frac{u_{i}}{u_{j}}=10^{j-i}\leq10^{\left|j-i\right|}\leq10
\]
which is valid for all $i\in I$ and $j\in i^{\ast}\subset\left\{ i-1,i,i+1\right\} $,
the weight $u$ is $\mathcal{Q}$-moderate.

Theorem \ref{thm:StructuredCoveringsSindDekompositionsCovering} guarantees
the existence of a BAPU $\left(\varphi_{i}\right)_{i\in I}$ subordinate
to $\mathcal{Q}$. Note that for $i\in I$ we have 
\[
\bigcup_{j\in I\setminus\left\{ i\right\} }Q_{j}\subset\left(-\infty,\left(i-1\right)+\frac{3}{4}\right)\cup\left(\left(i+1\right)-\frac{3}{4},\infty\right)=\left(-\infty,i-\frac{1}{4}\right)\cup\left(i+\frac{1}{4},\infty\right).
\]
This implies, together with $\varphi_{j}\left(x\right)=0$ for $x\in\mathbb{R}\setminus Q_{j}$
and $\sum_{i\in I}\varphi_{i}\left(x\right)=1$ for all $x\in\mathbb{R}$
that we have $\varphi_{i}\left(x\right)=1$ for $x\in\left[i-\frac{1}{4},i+\frac{1}{4}\right]$
for all $i\in I$.

Now choose a nonnegative function $\psi\in\mathcal{D}\left(\left(-\frac{1}{4},\frac{1}{4}\right)\right)\setminus\left\{ 0\right\} $
and define $f_{n}:=\sum_{j=1}^{n}4^{n}\cdot L_{n}\psi$ for $n\in\mathbb{N}$.
Because of 
\[
\text{supp}\left(L_{n}\psi\right)\subset\left(n-\frac{1}{4},n+\frac{1}{4}\right)\subset\left(\bigcup_{j\in I\setminus\left\{ n\right\} }Q_{j}\right)^{c}
\]
it is then easy to see that 
\[
\varphi_{i}\cdot L_{n}\psi=\begin{cases}
0, & i\neq n,\\
L_{i}\psi, & i=n
\end{cases}
\]
holds for $i,n\in\mathbb{Z}$. For $n\geq m\geq m_{0}$ we thus get
\begin{eqnarray*}
\left\Vert f_{n}-f_{m}\right\Vert _{\mathcal{D}\left(\mathcal{Q},L^{1},\ell_{u}^{1}\right)} & = & \sum_{i\in\mathbb{Z}}10^{-i}\left\Vert \mathcal{F}^{-1}\left(\varphi_{i}\cdot\sum_{j=m+1}^{n}4^{j}\cdot L_{j}\psi\right)\right\Vert _{L^{1}\left(\mathbb{R}^{d}\right)}\\
 & = & \sum_{i=m+1}^{n}10^{-i}4^{i}\cdot\left\Vert \mathcal{F}^{-1}\left(L_{i}\psi\right)\right\Vert _{L^{1}\left(\mathbb{R}^{d}\right)}\\
 & \leq & \left\Vert \mathcal{F}^{-1}\psi\right\Vert _{L^{1}\left(\mathbb{R}^{d}\right)}\cdot\sum_{i=m_{0}+1}^{\infty}\left(\frac{4}{10}\right)^{i}\xrightarrow[m_{0}\rightarrow\infty]{}0,
\end{eqnarray*}
so that $\left(f_{n}\right)_{n\in\mathbb{N}}$ is a Cauchy sequence
in $\mathcal{D}_{\mathcal{S}'}\left(\mathcal{Q},L^{1},\ell_{u}^{1}\right)$.

If there was some $f\in\mathcal{D}_{\mathcal{S}'}\left(\mathcal{Q},L^{1},\ell_{u}^{1}\right)\subset\mathcal{S}'\left(\mathbb{R}^{d}\right)$
satisfying $f_{n}\xrightarrow[n\rightarrow\infty]{\mathcal{D}_{\mathcal{S}'}\left(\mathcal{Q},L^{1},\ell_{u}^{1}\right)}f$,
the continuous embeddings 
\[
\mathcal{D}_{\mathcal{S}'}\left(\mathcal{Q},L^{1},\ell_{u}^{1}\right)\hookrightarrow\mathcal{D}\left(\mathcal{Q},L^{1},\ell_{u}^{1}\right)\hookrightarrow\mathcal{D}'\left(\mathbb{R}^{d}\right)
\]
would imply 
\[
\left\langle f,g\right\rangle _{\mathcal{D}',\mathcal{D}}=\lim_{n\rightarrow\infty}\left\langle f_{n},g\right\rangle _{\mathcal{D}',\mathcal{D}}\qquad\text{for all }g\in\mathcal{D}\left(\mathbb{R}^{d}\right).
\]
Because of the definition of the topology on $\mathcal{S}\left(\mathbb{R}^{d}\right)$,
by \cite[Proposition 5.15]{FollandRA} and because of $f\in\mathcal{S}'\left(\mathbb{R}^{d}\right)$
there are suitable $N\in\mathbb{N}$ and $C>0$ such that 
\[
\left|\left\langle f,g\right\rangle _{\mathcal{S}',\mathcal{S}}\right|\leq C\cdot\sup_{\substack{\alpha\in\mathbb{N}_{0}^{d}\\
\left|\alpha\right|\leq N
}
}\sup_{x\in\mathbb{R}^{d}}\left(1+\left|x\right|\right)^{N}\cdot\left|\left(\partial^{\alpha}g\right)\left(x\right)\right|
\]
holds for all $g\in\mathcal{S}\left(\mathbb{R}^{d}\right)$.

For $n\in\mathbb{N}$ and $g:=T_{n}\psi$ we have $\text{supp}\left(g\right)\subset\left(n-\frac{1}{4},n+\frac{1}{4}\right)\subset\left[0,n+1\right]$
and thus 
\[
\sup_{\substack{\alpha\in\mathbb{N}_{0}^{d}\\
\left|\alpha\right|\leq N
}
}\sup_{x\in\mathbb{R}^{d}}\left(1+\left|x\right|\right)^{N}\cdot\left|\left(\partial^{\alpha}g\right)\left(x\right)\right|\leq\left(n+2\right)^{N}\cdot\sup_{\substack{\alpha\in\mathbb{N}_{0}^{d}\\
\left|\alpha\right|\leq N
}
}\sup_{x\in\mathbb{R}^{d}}\left|\left(\partial^{\alpha}\psi\right)\left(x\right)\right|=\left(n+2\right)^{N}\cdot C_{\psi,N}
\]
for some constant $C_{\psi,N}\in\left(0,\infty\right)$. But because
of $\text{supp}\left(L_{n}\psi\right)\cap\text{supp}\left(L_{i}\psi\right)=\emptyset$
for $i,n\in\mathbb{Z}$ with $i\neq n$ we have, for $m\geq n$, the
identity 
\[
\left\langle f_{m},g\right\rangle _{\mathcal{S}',\mathcal{S}}=\sum_{j=1}^{m}4^{j}\left\langle L_{j}\psi,L_{n}\psi\right\rangle _{\mathcal{S}',\mathcal{S}}=4^{n}\cdot\left\langle \psi,\psi\right\rangle _{\mathcal{S}',\mathcal{S}}
\]
and thus 
\begin{align*}
4^{n}\cdot\left\Vert \psi\right\Vert _{L^{2}\left(\mathbb{R}^{d}\right)}^{2} & =  \lim_{m\rightarrow\infty}\left|\left\langle f_{m},g\right\rangle _{\mathcal{S}',\mathcal{S}}\right|=\left|\left\langle f,g\right\rangle _{\mathcal{S}',\mathcal{S}}\right|\\
 & \leq  C\cdot\sup_{\substack{\alpha\in\mathbb{N}_{0}^{d}\\
\left|\alpha\right|\leq N
}
}\sup_{x\in\mathbb{R}^{d}}\left(1+\left|x\right|\right)^{N}\cdot\left|\left(\partial^{\alpha}g\right)\left(x\right)\right|\\
 & \leq  CC_{\psi,N}\cdot\left(n+2\right)^{N}
\end{align*}
for all $n\in\mathbb{N}$, a contradiction.

This shows that there is no $f\in\mathcal{D}_{\mathcal{S}'}\left(\mathcal{Q},L^{1},\ell_{u}^{1}\right)$
with $\left\Vert f_{n}-f\right\Vert _{\mathcal{D}\left(\mathcal{Q},L^{1},\ell_{u}^{1}\right)}\xrightarrow[n\rightarrow\infty]{}0$,
so that $\mathcal{D}_{\mathcal{S}'}\left(\mathcal{Q},L^{1},\ell_{u}^{1}\right)$
is \emph{not} complete.
\end{example*}
In the next lemma, we indicate the way in which we will choose the
$\mathcal{Q}$-moderate weight $u:I\rightarrow\left(0,\infty\right)$.
\begin{lem}
\label{lem:GewichtsDiskretisierung}Let $\emptyset\neq U\subset\mathbb{R}^{d}$
be open and let $\mathcal{Q}=\left(Q_{i}\right)_{i\in I}$ be a cover
of $U$. Finally, let $u:U\rightarrow\left(0,\infty\right)$ be \textbf{$\mathcal{Q}$-moderate}
in the sense that there is some constant $C>0$ such that 
\[
\frac{u\left(x\right)}{u\left(y\right)}\leq C\qquad\text{ holds for all }i\in I\text{ and all }x,y\in Q_{i}.
\]
We then say that $u':I\rightarrow\left(0,\infty\right)$ is \textbf{a
discretization of $u$} if for every $i\in I$ there is some $x_{i}\in Q_{i}$
so that $u'\left(i\right)=u\left(x_{i}\right)$ holds. In that case, $u'$
is also $\mathcal{Q}$-moderate%
\footnote{Note that $u'$ is a weight on the discrete index set $I$, so that
$\mathcal{Q}$-moderateness of $u'$ means that there is a constant
$C>0$ such that we have $u_{i}\leq C\cdot u_{j}$ for all $i\in I$
and $j\in i^{\ast}$ (cf. Definition \ref{def:DecompositionSpace}).%
}.

Furthermore, any two discretizations $u',u''$ of $u$ are equivalent
in the sense that the estimate $C^{-1}\cdot u_{i}''\leq u_{i}'\leq C\cdot u_{i}''$
holds for all $i\in I$. In particular we have $\ell_{u'}^{q}=\ell_{u''}^{q}$
with equivalent norms and thus also $\mathcal{D}\left(\mathcal{Q},L^{p},\ell_{u'}^{q}\right)=\mathcal{D}\left(\mathcal{Q},L^{p},\ell_{u''}^{q}\right)$
(if $\mathcal{Q}$ is a decomposition covering of $U$).\end{lem}
\begin{rem*}
Employing the independence of $\ell_{u'}^{q}$ of the chosen discretization,
we write $\mathcal{D}\left(\mathcal{Q},L^{p},\ell_{u}^{q}\right)$
for $\mathcal{D}\left(\mathcal{Q},L^{p},\ell_{u'}^{q}\right)$ and
$\ell_{u}^{q}\left(I\right)$ for $\ell_{u'}^{q}\left(I\right)$ for
every discretization $u'$ of $u$. The chosen discretization $u'$ of $u$
will often be denoted by $u$ again.

It is important to note that the notation $\ell_{u}^{q}\left(I\right)$
is a slight abuse of notation, as the discretization $u'$ of $u$
(heavily) depends upon the chosen covering $\mathcal{Q}$. Hence,
different coverings with the same index set can lead to very different
spaces $\ell_{u}^{q}\left(I\right)$.

It is thus important to remember that for a weight $u:U\rightarrow\left(0,\infty\right)$
the notation $\ell_{u}^{q}\left(I\right)$ should (and will) only
be used as long it is clearly understood which covering is used to
form the discretization.\end{rem*}
\begin{proof}[Proof of Lemma \ref{lem:GewichtsDiskretisierung}]
Let $i\in I$ and $j\in i^{\ast}$. Thus there is some $x\in Q_{i}\cap Q_{j}\neq\emptyset$.
The $\mathcal{Q}$-moderateness of $u$ yields 
\[
\frac{u_{i}'}{u_{j}'}=\frac{u\left(x_{i}\right)}{u\left(x_{j}\right)}\;\overset{x_{i},x\in Q_{i}}{\leq}\;\frac{C\cdot u\left(x\right)}{u\left(x_{j}\right)}\;\overset{x,x_{j}\in Q_{j}}{\leq}\; C^{2}
\]
which shows that $u'$ is $\mathcal{Q}$-moderate.

If $u',u''$ are both discretizations of $u$, let $i\in I$, choose
$x_{i},x_{i}'\in Q_{i}$ satisfying $u_{i}'=u\left(x_{i}\right)$
and $u_{i}''=u\left(x_{i}'\right)$ and derive
\[
u_{i}'=u\left(x_{i}\right)\;\overset{x_{i},x_{i}'\in Q_{i}}{\leq}\; C\cdot u\left(x_{i}'\right)=C\cdot u_{i}''.
\]
The reverse estimate follows by symmetry.
\end{proof}
Finally, we transplant a weight $v:H\rightarrow\left(0,\infty\right)$
onto the dual orbit $\mathcal{O}$ by choice of a cross-section. The
resulting function on $\mathcal{O}$ will be called a \textbf{transplant
of $v$ from $H$ onto $\mathcal{O}$}. The main observation of the
following lemma is that for any two such transplants $u_{1},u_{2}$
of a moderate weight $v$, the quotient $u_{1}/u_{2}$ is bounded
from above and away from zero, i.e. the two transplants are equivalent.
\begin{lem}
\label{lem:GewichtsTransplantation}Let $v:H\rightarrow\left(0,\infty\right)$
be $v_{0}$-moderate for some locally bounded, submultiplicative weight
$v_{0}:H\rightarrow\left(0,\infty\right)$.

For each $\xi\in\mathcal{O}$ choose some $h_{\xi}\in H$ satisfying
$h_{\xi}^{T}\xi_{0}=\xi$ and define 
\[
u:\mathcal{O}\rightarrow\left(0,\infty\right),\xi\mapsto v\left(h_{\xi}\right).
\]
Then $u$ is a $\mathcal{Q}$-moderate function for every covering
$\mathcal{Q}$ of $\mathcal{O}$ induced by $H$.

Furthermore, any two choices $h_{\xi},h_{\xi}'\in H$ satisfying $h_{\xi}^{T}\xi_{0}=\xi=\left(\smash{h_{\xi}'}\right)^{T}\xi_{0}$
yield equivalent weights $u,u'$ in the sense that $C^{-1}\cdot u'\leq u\leq C\cdot u'$
holds for some constant $C\in\left(0,\infty\right)$. The same is
true for any two choices of $\xi_{0}$.\end{lem}
\begin{rem}
\label{rem:SpezielleDiskretisierung}The induced covering is of the
form $\mathcal{Q}=\left(Q_{i}\right)_{i\in I}=\left(h_{i}^{-T}Q\right)_{i\in I}$
for some well-spread family $\left(h_{i}\right)_{i\in I}$ and a suitable
set $Q\subset\mathcal{O}$. As long as we have $\xi_{0}\in Q$, we
can choose $h_{h_{i}^{-T}\xi_{0}}=h_{i}^{-1}$ and discretize $u$
by $u_{i}'=u\left(h_{i}^{-T}\xi_{0}\right)=v\left(h_{h_{i}^{-T}\xi_{0}}\right)=v\left(h_{i}^{-1}\right)$,
where we used $h_{i}^{-T}\xi_{0}\in Q_{i}$.

It is worth noting that $\xi_{0}\in\mathcal{O}$ can be selected arbitraryly,
so that it is always possible to choose $\xi_{0}\in Q$.\end{rem}
\begin{proof}
Let $\mathcal{Q}=\left(h_{i}^{-T}Q\right)_{i\in I}$ be a covering
of $\mathcal{O}$ induced by $H$. This means that $\left(h_{i}\right)_{i\in I}$
is well-spread in $H$, that $Q\subset\mathbb{R}^{d}$ is precompact
with $\overline{Q}\subset\mathcal{O}$ and that $\mathcal{Q}$ is
a covering of $\mathcal{O}$. Let $i\in I$ and $x,y\in Q_{i}=h_{i}^{-T}Q$
be arbitrary. Then we have 
\[
\left(h_{x}h_{i}\right)^{T}\xi_{0}=h_{i}^{T}h_{x}^{T}\xi_{0}=h_{i}^{T}x\in Q\subset\overline{Q}\qquad\text{ and similarly }\qquad\left(h_{y}h_{i}\right)^{T}\xi_{0}\in\overline{Q}.
\]
This means $h_{x}h_{i}\in p_{\xi_{0}}^{-1}\left(\overline{Q}\right)=:K$
and $h_{y}h_{i}\in K$. Note that $K \subset H$ is compact by Lemma \ref{lem:OrbitProjektionIstProper}.

Using this, we can estimate 
\begin{align*}
u\left(x\right)=v\left(h_{x}\right)  & =  v\left(h_{x}h_{i}h_{i}^{-1}h_{y}^{-1}h_{y}1_{H}\right)\\
 & \leq  v_{0}\left(h_{x}h_{i}h_{i}^{-1}h_{y}^{-1}\right)\cdot v\left(h_{y}\right)\cdot v_{0}\left(1_{H}\right)\\
 & =  v_{0}\left(1_{H}\right)\cdot v_{0}\left(h_{x}h_{i}\cdot\left(h_{y}h_{i}\right)^{-1}\right)\cdot v\left(h_{y}\right)\\
 & \leq  \left[v_{0}\left(1_{H}\right)\cdot\sup_{h\in KK^{-1}}v_{0}\left(h\right)\right]\cdot u\left(y\right)=C\cdot u\left(y\right),
\end{align*}
where the value $C=v_{0}\left(1_{H}\right)\cdot\sup_{h\in KK^{-1}}v_{0}\left(h\right)$
is independent of $x,y\in Q_{i}$ and of $i\in I$ and finite because
$v_{0}$ is locally bounded. Thus, $u$ is $\mathcal{Q}$-moderate.

In order to show the independence of the choice of $h_{\xi}$, let
$\xi\in\mathcal{O}$ and choose $h_{\xi},h_{\xi}'\in H$ such that
$h_{\xi}^{T}\xi_{0}=\xi=\left(\smash{h_{\xi}'}\right)^{T}\xi_{0}$.
Then $h_{\xi}\cdot\left(\smash{h_{\xi}'}\right)^{-1}\in H_{\xi_{0}}$
which implies 
\[
    v\left(h_{\xi}\right) = v(h_{\xi} (h_{\xi}')^{-1} \cdot h_{\xi}' \cdot 1_{H}) \leq v\left(h_{\xi}'\right)\cdot\left[\smash{v_{0}\left(1_{H}\right)\cdot\sup_{h\in H_{\xi_{0}}}v_{0}\left(h\right)}\vphantom{\sup_{h\in K K^{-1}}}\right],
\]
where the expression in brackets is an absolute constant which is
finite by local boundedness of $v_{0}$ and compactness of $H_{\xi_{0}}$.

The proof of independence of $\xi_{0}$ runs along similar lines,
and is omitted.
\end{proof}
Finally, we want to show that all induced coverings (with respect
to the same well-spread family $\left(h_{i}\right)_{i\in I}$) yield
the same decomposition spaces, at least as long as the weight $u$
is obtained by transplanting a weight on $H$ onto $\mathcal{O}$.
This will be an easy consequence of the following more general lemma.
\begin{lem}
\label{lem:EinfacheAequivalenzBeiGleichenIndexmengen}Let $\emptyset\neq U\subset\mathbb{R}^{d}$
be an open set and assume that $\mathcal{Q}=\left(Q_{i}\right)_{i\in I}$
and $\mathcal{Q}'=\left(Q_{i}'\right)_{i\in I}$ are two admissible
coverings of $U$ that are indexed by the \emph{same} set $I$. Furthermore,
assume that $\mathcal{Q}$ is a decomposition covering of $U$ that
satisfies $Q_{i}\subset Q_{i}'$ for all $i\in I$.

Then $\mathcal{Q}'$ is also a decomposition covering of $U$. More
precisely, any $\mathcal{Q}$-BAPU $\left(\varphi_{i}\right)_{i\in I}$
is also a BAPU for $\mathcal{Q}'$. Finally, if $u:U\rightarrow\left(0,\infty\right)$
is $\mathcal{Q}'$-moderate, it is also $\mathcal{Q}$-moderate and
for $p,q\in\left[1,\infty\right]$ we have 
\[
\mathcal{D}\left(\mathcal{Q},L^{p},\ell_{u}^{q}\right)=\mathcal{D}\left(\mathcal{Q}',L^{p},\ell_{u}^{q}\right)
\]
with equivalent norms.\end{lem}
\begin{proof}
The only property of a BAPU $\left(\varphi_{i}\right)_{i\in I}$ that
is specific to the covering $\mathcal{Q}$ (or $\mathcal{Q}'$) is
the requirement $\varphi_{i}\left(x\right)=0$ for all $x\in U\setminus Q_{i}$.
Because of $U\setminus Q_{i}'\subset U\setminus Q_{i}$, it is clear
that this condition is also fulfilled for the covering $\mathcal{Q}'$
instead of $\mathcal{Q}$.

If $u:U\rightarrow\left(0,\infty\right)$ is $\mathcal{Q}'$-moderate,
there is a constant $C>0$ so that $\frac{u\left(x\right)}{u\left(y\right)}\leq C$
holds for all $i\in I$ and all $x,y\in Q_{i}'$. Because of $Q_{i}\subset Q_{i}'$
the same estimate also holds for all $i\in I$ and all $x,y\in Q_{i}$.

In order to show the equality $\mathcal{D}\left(\mathcal{Q},L^{p},\ell_{u}^{q}\right)=\mathcal{D}\left(\mathcal{Q}',L^{p},\ell_{u}^{q}\right)$
note that for a discretization $u':I\rightarrow\left(0,\infty\right)$
with respect to $\mathcal{Q}$, $u'$ is also a discretization for
$\mathcal{Q}'$, as for $i\in I$ there is some $x_{i}\in Q_{i}\subset Q_{i}'$
that satisfies $u_{i}'=u\left(x_{i}\right)$. As seen above, any $\mathcal{Q}$-BAPU
$\left(\varphi_{i}\right)_{i\in I}$ is also a $\mathcal{Q}'$-BAPU.
With these choices we see 
\[
\left\Vert f\right\Vert _{\mathcal{D}\left(\mathcal{Q},L^{p},\ell_{u}^{q}\right)}=\left\Vert \left(u_{i}'\cdot\left\Vert \mathcal{F}^{-1}\left(\varphi_{i}f\right)\right\Vert _{L^{p}\left(\mathbb{R}^{d}\right)}\right)_{i\in I}\right\Vert _{\ell^{q}\left(I\right)}=\left\Vert f\right\Vert _{\mathcal{D}\left(\mathcal{Q}',L^{p},\ell_{u}^{q}\right)}
\]
for all $f\in\mathcal{D}'\left(U\right)$. This shows $\mathcal{D}\left(\mathcal{Q},L^{p},\ell_{u}^{q}\right)=\mathcal{D}\left(\mathcal{Q}',L^{p},\ell_{u}^{q}\right)$
and because any choice of discretization of the weight $u$ (with
respect to $\mathcal{Q}$ or $\mathcal{Q}'$) and of the BAPUs for
$\mathcal{Q}$ or $\mathcal{Q}'$ yield equivalent norms on $\mathcal{D}\left(\mathcal{Q},L^{p},\ell_{u}^{q}\right)$
or $\mathcal{D}\left(\mathcal{Q}',L^{p},\ell_{u}^{q}\right)$, the
claim follows (for any such choice).
\end{proof}
We can now conclude that two different induced decomposition coverings -- with respect to the same well-spread family $\left(h_{i}\right)_{i\in I}$ --
yield identical decomposition spaces. The isomorphism between $\mathcal{D}\left(\mathcal{Q},L^{p},\ell_{u}^{q}\right)$
and $\text{Co}\left(L_{v}^{p,q}\right)$ that we will prove below
(see Theorems \ref{thm:FourierStetigVonCoorbitNachDecomposition}
and \ref{thm:InverseFourierTransformationStetigkeit}) will of course
show that the same is true even if different well-spread families
$\left(h_{i}\right)_{i\in I}$ are used to obtain the two decomposition
coverings $\mathcal{Q},\mathcal{Q}'$. Even so, this does not make
the following result redundant, as it will allow us to switch from
the covering $\left(h_{i}^{-T}Q\right)_{i\in I}$ to a larger covering
$\left(h_{i}^{-T}Q'\right)_{i\in I}$ in the proof of Theorem \ref{thm:FourierStetigVonCoorbitNachDecomposition}.
\begin{cor}
\label{cor:InduzierteUeberdeckLiefernGleicheRaeumeBeiGleicherFamilieH}Let
$\mathcal{Q}=\left(h_{i}^{-T}Q\right)_{i\in I}$ and $\mathcal{Q}'=\left(h_{i}^{-T}Q'\right)_{i\in I}$
be two (possibly different) decomposition coverings of $\mathcal{O}$
induced by $H$. Then we have 
\[
\mathcal{D}\left(\mathcal{Q},L^{p},\ell_{u}^{q}\right)=\mathcal{D}\left(\mathcal{Q}',L^{p},\ell_{u}^{q}\right)
\]
with equivalent norms for every weight $u:\mathcal{O}\rightarrow\left(0,\infty\right)$
obtained by transplanting a $v_{0}$-moderate weight $v:H\rightarrow\left(0,\infty\right)$
onto $\mathcal{O}$, where $v_{0}:H\rightarrow\left(0,\infty\right)$
is submultiplicative and locally bounded.\end{cor}
\begin{proof}
The assumptions guarantee that $\left(h_{i}\right)_{i\in I}$ is well-spread
in $H$ and that $\overline{Q},\overline{Q'}\subset\mathcal{O}$ are
compact sets that satisfy $\bigcup_{i\in I}h_{i}^{-T}Q=\mathcal{O}=\bigcup_{i\in I}h_{i}^{-T}Q'$.
Thus, $Q'':=Q\cup Q'\subset\mathcal{O}$ also satisfies $\mathcal{O}=\bigcup_{i\in I}h_{i}^{-T}Q''$
and $\overline{Q''}\subset\mathcal{O}$ is compact.

Thus, $\mathcal{Q},\mathcal{Q}'$ and $\mathcal{Q}''$ are coverings
of $\mathcal{O}$ induced by $H$. Theorem \ref{thm:InduzierteUeberdeckungKonstruktion}
then shows that these are admissible coverings and Lemma \ref{lem:GewichtsTransplantation}
shows that $u$ is moderate with respect to any of these coverings.
Because $\mathcal{Q}$ is a decomposition covering, the same is true
for $\mathcal{Q}''$ by Lemma \ref{lem:EinfacheAequivalenzBeiGleichenIndexmengen}.
The inclusion 
\[
Q_{i}\cup Q_{i}'=h_{i}^{-T}Q\cup h_{i}^{-T}Q'\subset h_{i}^{-T}Q''=Q_{i}''
\]
which is valid for all $i\in I$ and Lemma \ref{lem:EinfacheAequivalenzBeiGleichenIndexmengen}
then imply 
\[
\mathcal{D}\left(\mathcal{Q},L^{p},\ell_{u}^{q}\right)=\mathcal{D}\left(\mathcal{Q}'',L^{p},\ell_{u}^{q}\right)=\mathcal{D}\left(\mathcal{Q}',L^{p},\ell_{u}^{q}\right)
\]
with equivalent norms.
\end{proof}

\section{Construction of a specific BAPU for induced coverings}

\label{sec:BAPU}In this section we construct a specific BAPU for
the covering of the dual orbit induced by $H$. This BAPU will allow
us to prove the continuity of $\mathcal{F}:\text{Co}\left(L_{v}^{p,q}\right)\rightarrow\mathcal{D}\left(\mathcal{Q},L^{p},\ell_{u}^{q}\right)$,
where $\mathcal{Q}$ is an induced decomposition covering of the dual
orbit. The idea of the construction is to take a Schwartz function
$\psi\in\mathcal{S}\left(\mathbb{R}^{d}\right)$ such that $\widehat{\psi}\in\mathcal{D}\left(\mathcal{O}\right)$
is compactly supported in $\mathcal{O}$ and then define 
\begin{equation}
\varphi_{U}\left(\xi\right):=\int_{U}\left|\widehat{\psi}\left(h^{T}\xi\right)\right|^{2}\,\text{d}h\qquad\text{for }\xi\in\mathcal{O}\label{eq:BAPUIdee}
\end{equation}
for $U\subset H$ precompact and measurable. We will then show that
$\varphi_{U}\in\mathcal{D}\left(\mathcal{O}\right)$ is smooth with
\[
\text{supp}\left(\varphi_{U}\right)\subset\overline{U}^{-T}\text{supp}\left(\smash{\widehat{\psi}}\right)
\]
and that $\left(\varphi_{U_{i}}\right)_{i\in I}$ is a (multiple of
a) partition of unity on $\mathcal{O}$ if $\left(U_{i}\right)_{i\in I}$
is a partition of $H$.

We now show that the construction indicated in equation (\ref{eq:BAPUIdee})
indeed yields a test function $\varphi_{U}\in\mathcal{D}\left(\mathcal{O}\right)$.
Note that we could also use ``differentiation under the integral
sign'' instead of Lemma \ref{lem:GlatteVerkettung} in the following
proof. But as we need Lemma \ref{lem:GlatteVerkettung} nonetheless,
we prefer the following argument. 
\begin{lem}
\label{lem:BAPUKonstruktion}Assume $\psi\in\mathcal{S}\left(\mathbb{R}^{d}\right)$
with $\widehat{\psi}\in\mathcal{D}\left(\mathcal{O}\right)$. Let
$U\subset H$ be precompact and measurable. Then 
\[
\varphi_{U}:\mathbb{R}^{d}\rightarrow\left[0,\infty\right),\xi\mapsto\int_{U}\left|\widehat{\psi}\left(h^{T}\xi\right)\right|^{2}\,\text{d}h
\]
is well-defined with $\varphi_{U}\in\mathcal{D}\left(\mathcal{O}\right)$.

More precisely, we have $\varphi_{U}\equiv0$ on $\mathbb{R}^{d}\setminus\left(U^{-T}\cdot\widehat{\psi}\left(\mathbb{C}^{\ast}\right)\right)$
and thus 
\[
\text{supp}\left(\varphi_{U}\right)\subset\overline{U}^{-T}\cdot\text{supp}\left(\smash{\widehat{\psi}}\right)\subset H^{T}\mathcal{O}=\mathcal{O}.
\]
\end{lem}
\begin{proof}
Let $\varphi:=\left|\smash{\widehat{\psi}}\right|^{2}\in\mathcal{D}\left(\mathcal{O}\right)\subset\mathcal{D}\left(\mathbb{R}^{d}\right)$.
Lemma \ref{lem:GlatteVerkettung} shows that 
\[
\Phi:\text{GL}\left(\mathbb{R}^{d}\right)\rightarrow\mathcal{D}\left(\mathbb{R}^{d}\right),h\mapsto\varphi\left(h^{T}\cdot\right)=\left|\widehat{\psi}\left(h^{T}\cdot\right)\right|^{2}
\]
is well-defined and continuous, so that $\Phi$ is in particular continuous
on the compact set $\overline{U}\subset H$.

For $k\in\mathbb{N}_{0}$ the inclusion $\iota_{k}:\mathcal{D}\left(\mathbb{R}^{d}\right)\hookrightarrow C_{b}^{k}\left(\mathbb{R}^{d}\right)$
with 
\[
C_{b}^{k}\left(\mathbb{R}^{d}\right):=\left\{ f\in C^{k}\left(\mathbb{R}^{d}\right)\with\left\Vert f\right\Vert _{C_{b}^{k}}:=\sum_{\left|\alpha\right|\leq k}\left\Vert \partial^{\alpha}f\right\Vert _{\sup}<\infty\right\} 
\]
is continuous, so that the Bochner integral of the function $\iota_{k}\circ\Phi$,
\[
\psi_{k}:=\int_{U}\left(\iota_{k}\circ\Phi\right)\left(h\right)\,\text{d}h\in C_{b}^{k}\left(\mathbb{R}^{d}\right)
\]
is well-defined, because of $\mu_{H}\left(U\right)\leq\mu_{H}\left(\overline{U}\right)<\infty$,
where $\mu_{H}$ is the Haar-measure on $H$.

As the evaluation mapping $\alpha_{\xi}:C_{b}^{k}\left(\mathbb{R}^{d}\right)\rightarrow\mathbb{C},f\mapsto f\left(\xi\right)$
is continuous for every $\xi\in\mathbb{R}^{d}$, we easily see $\psi_{k}\left(\xi\right)=\varphi_{U}\left(\xi\right)$
for each $\xi\in\mathbb{R}^{d}$, so that we conclude $\varphi_{U}=\psi_{k}\in C_{b}^{k}\left(\mathbb{R}^{d}\right)$
for all $k\in\mathbb{N}_{0}$ which shows that $\varphi_{U}$ is smooth.

Finally, if we have $0\neq\varphi_{U}\left(\xi\right)=\int_{U}\left|\widehat{\psi}\left(h^{T}\xi\right)\right|^{2}\,\text{d}h$,
there is some $h\in U$ with $h^{T}\xi\in\widehat{\psi}\left(\mathbb{C}^{\ast}\right)$,
i.e. 
\[
\xi\in h^{-T}\cdot\widehat{\psi}\left(\mathbb{C}^{\ast}\right)\subset U^{-T}\cdot\widehat{\psi}\left(\mathbb{C}^{\ast}\right).\qedhere
\]

\end{proof}
We now calculate the (inverse) Fourier transform of the function $\varphi_{U}$
as defined in the preceding lemma. The formula that we derive will
be important for the proof of the continuity of the Fourier transform
$\mathcal{F}:\text{Co}\left(L_{v}^{p,q}\right)\rightarrow\mathcal{D}\left(\mathcal{Q},L^{p},\ell_{u}^{q}\right)$.
\begin{lem}
\label{lem:BAPUFourierTransformation}Let $\psi\in\mathcal{S}\left(\mathbb{R}^{d}\right)$
with $\widehat{\psi}\in\mathcal{D}\left(\mathcal{O}\right)$. For
$\gamma:=\mathcal{F}^{-1}\left(\smash{\left|\smash{\widehat{\psi}}\right|^{2}}\right)\in\mathcal{S}\left(\mathbb{R}^{d}\right)$
and some precompact, measurable $U\subset H$ we have, with $\varphi_{U}$
defined as in Lemma \ref{lem:BAPUKonstruktion}, 
\begin{equation}
\left(\mathcal{F}^{-1}\varphi_{U}\right)\left(x\right)=\int_{U}\frac{\gamma\left(h^{-1}x\right)}{\left|\det\left(h\right)\right|}\,\text{d}h\qquad\text{for all }x\in\mathbb{R}^{d}.\label{eq:BAPUFourier}
\end{equation}
Furthermore, the estimate 
\[
\left\Vert \mathcal{F}^{-1}\varphi_{U}\right\Vert _{L^{1}\left(\mathbb{R}^{d}\right)}\leq\mu_{H}\left(U\right)\cdot\left\Vert \gamma\right\Vert _{L^{1}\left(\mathbb{R}^{d}\right)}<\infty
\]
holds, where $\mu_{H}$ denotes the chosen (left) Haar-measure on
$H$.\end{lem}
\begin{proof}
Let $\varphi:=\left|\smash{\widehat{\psi}}\right|^{2}\in\mathcal{D}\left(\mathcal{O}\right)$.
Then Lemma \ref{lem:GlatteVerkettung} shows that 
\[
\Phi:\text{GL}\left(\mathbb{R}^{d}\right)\rightarrow\mathcal{D}\left(\mathbb{R}^{d}\right)\hookrightarrow\mathcal{S}\left(\mathbb{R}^{d}\right)\hookrightarrow L^{1}\left(\mathbb{R}^{d}\right),h\mapsto \varphi\left(h^{T} \cdot \right)=\left|\widehat{\psi}\left(h^{T} \cdot \right)\right|^{2}
\]
is well-defined and continuous. Hence 
\begin{align*}
\int_{U}\int_{\mathbb{R}^{d}}\left|\left|\widehat{\psi}\left(h^{T}\xi\right)\right|^{2}\cdot e^{2\pi i\left\langle x,\xi\right\rangle }\right|\,\text{d}\xi\,\text{d}h & =  \int_{U}\left\Vert \Phi\left(h\right)\right\Vert _{L^{1}\left(\mathbb{R}^{d}\right)}\,\text{d}h\\
 & \leq  \mu_{H}\left(\overline{U}\right)\cdot\sup_{h\in\overline{U}}\left\Vert \Phi\left(h\right)\right\Vert _{L^{1}\left(\mathbb{R}^{d}\right)}<\infty,
\end{align*}
where we used that $\overline{U}$ is compact. Fubini's theorem, the
change of variables formula and Fourier inversion now yield 
\begin{eqnarray*}
\left(\mathcal{F}^{-1}\varphi_{U}\right)\left(x\right) & = & \int_{\mathbb{R}^{d}}\int_{U}\left|\widehat{\psi}\left(h^{T}\xi\right)\right|^{2}\,\text{d}h\, e^{2\pi i\left\langle x,\xi\right\rangle }\,\text{d}\xi\\
 & \overset{\text{Fubini}}{=} & \int_{U}\frac{1}{\left|\det\left(h\right)\right|}\cdot\int_{\mathbb{R}^{d}}\left|\widehat{\psi}\left(h^{T}\xi\right)\right|^{2}\cdot e^{2\pi i\left\langle h^{-1}x,h^{T}\xi\right\rangle }\cdot\left|\det\left(h^{T}\right)\right|\,\text{d}\xi\,\text{d}h\\
 & \overset{\varrho=h^{T}\xi}{=} & \int_{U}\frac{1}{\left|\det\left(h\right)\right|}\cdot\int_{\mathbb{R}^{d}}\widehat{\gamma}\left(\varrho\right)\cdot e^{2\pi i\left\langle h^{-1}x,\varrho\right\rangle }\,\text{d}\varrho\,\text{d}h\\
 & = & \int_{U}\frac{\gamma\left(h^{-1}x\right)}{\left|\det\left(h\right)\right|}\,\text{d}h.
\end{eqnarray*}
A second application of Fubini's theorem and the change of variables
formula finally yields 
\begin{eqnarray*}
\quad\qquad\qquad\qquad\qquad\left\Vert \mathcal{F}^{-1}\varphi_{U}\right\Vert _{1} & = & \int_{\mathbb{R}^{d}}\left|\int_{U}\frac{\gamma\left(h^{-1}x\right)}{\left|\det\left(h\right)\right|}\,\text{d}h\right|\,\text{d}x\\
 & \overset{\text{Fubini}}{\leq} & \int_{U}\int_{\mathbb{R}^{d}}\left|\gamma\left(h^{-1}x\right)\right|\cdot\left|\det\left(h^{-1}\right)\right|\,\text{d}x\,\text{d}h\\
 & = & \int_{U}\int_{\mathbb{R}^{d}}\left|\gamma\left(\varrho\right)\right|\,\text{d}\varrho\,\text{d}h\\
 & \leq & \mu_{H}\left(\overline{U}\right)\cdot\left\Vert \gamma\right\Vert _{L^{1}\left(\mathbb{R}^{d}\right)}<\infty.\qquad\qquad\qquad\qquad\qquad\quad\qedhere
\end{eqnarray*}

\end{proof}
We now want to show that $\left(\varphi_{U_{i}}\right)_{i\in I}$
yields a partition of unity on $\mathcal{O}$. In order to do so,
we first need the following technical lemma. Its proof is straightforward,
and therefore omitted.
\begin{lem}
\label{lem:ModulationStetigAufC0}The map 
\[
\Theta:\mathbb{R}^{d}\times C_{0}\left(\mathbb{R}^{d}\right)\rightarrow C_{0}\left(\mathbb{R}^{d}\right),\left(\omega,g\right)\mapsto M_{\omega}g
\]
is (jointly) continuous.

Here, $C_{0}\left(\mathbb{R}^{d}\right)$ is the space of (complex
valued) continuous functions vanishing at infinity endowed with the
sup-norm.
\end{lem}
We are now almost ready to show that $\left(\varphi_{U_{i}}\right)_{i\in I}$
indeed yields a (multiple of a) partition of unity on $\mathcal{O}$
for each (precompact, measurable) partition $\left(U_{i}\right)_{i\in I}$
of $H$. The only thing missing is the so-called \textbf{wavelet inversion
formula}. For the validity of this formula, we again use our assumptions
on $H$, which imply (as noted above) that $\pi$ is an irreducible,
square-integrable representation on $L^{2}\left(\mathbb{R}^{d}\right)$.
Then \cite[Theorem 3]{DufloMoore} states (in our notation) the following:
\begin{thm}
\label{thm:DufloMoore}(\cite[Theorem 3]{DufloMoore}) There is a
self-adjoint, positive operator ${K:{\rm dom}\left(K\right)\rightarrow L^{2}\left(\mathbb{R}^{d}\right)}$
(with $\text{dom}\left(K\right)\leq L^{2}\left(\mathbb{R}^{d}\right)$)
satisfying the following conditions:
\begin{enumerate}
\item For $\psi\in L^{2}\left(\mathbb{R}^{d}\right)$ the following are
equivalent:

\begin{enumerate}
\item $W_{\psi}\psi\in L^{2}\left(\mathbb{R}^{d}\right)$, 
\item $W_{\psi}f\in L^{2}\left(\mathbb{R}^{d}\right)$ for some $f\in L^{2}\left(\mathbb{R}^{d}\right)\setminus\left\{ 0\right\} $, 
\item $\psi\in\text{dom}\left(K^{-1/2}\right)$. 
\end{enumerate}
\item For $\varphi,\psi\in\text{dom}\left(K^{-1/2}\right)$ (i.e. with $W_{\psi}\psi,W_{\varphi}\varphi\in L^{2}\left(\mathbb{R}^{d}\right)$)
and arbitrary $f,g\in L^{2}\left(\mathbb{R}^{d}\right)$ we have 
\begin{equation}
\left\langle W_{\psi}f,W_{\varphi}g\right\rangle _{L^{2}\left(G\right)}=\left\langle K^{-1/2}\varphi,K^{-1/2}\psi\right\rangle _{L^{2}\left(\mathbb{R}^{d}\right)}\cdot\left\langle f,g\right\rangle _{L^{2}\left(\mathbb{R}^{d}\right)}.\label{eq:DufloMooreGleichung}
\end{equation}

\end{enumerate}
\end{thm}
For $\psi\in L^{2}\left(\mathbb{R}^{d}\right)\setminus\left\{ 0\right\} $
with $W_{\psi}\psi\in L^{2}\left(\mathbb{R}^{d}\right)$ and 
\[
C_{\psi}:=\left\langle K^{-1/2}\psi,K^{-1/2}\psi\right\rangle _{L^{2}\left(\mathbb{R}^{d}\right)}=\frac{\left\Vert W_{\psi}\psi\right\Vert _{L^{2}\left(G\right)}^{2}}{\left\Vert \psi\right\Vert _{L^{2}\left(\mathbb{R}^{d}\right)}^{2}}>0
\]
this entails the wavelet inversion formula 
\begin{equation}
f=\frac{1}{C_{\psi}}\cdot\int_{G}\left(W_{\psi}f\right)\left(x\right)\cdot\pi\left(x\right)\psi\,\text{d}x \qquad \text{ in the weak sense for all }f\in L^{2}\left(\mathbb{R}^{d}\right).\label{eq:WaveletInversionsFormel}
\end{equation}
The following lemma relates the constant $C_{\psi}$ to a continuous
partition of unity on the Fourier transform side. The discretization
of this partition of unity (essentially by cutting up the integration
domain $H$ into chunks of comparable sizes) will provide the BAPU
$(\varphi_{U_{i}})$.
\begin{lem}
\label{lem:BAPUIstZerlegungDerEins}For $\psi\in\mathcal{S}\left(\mathbb{R}^{d}\right)\setminus\left\{ 0\right\} $
with $\widehat{\psi}\in\mathcal{D}\left(\mathcal{O}\right)$ we have
$W_{\psi}\psi\in L^{2}\left(G\right)$. Furthermore, we have 
\begin{equation}
\frac{1}{C_{\psi}}\cdot\int_{H}\left|\widehat{\psi}\left(h^{T}\xi\right)\right|^{2}\,\text{d}h=1\qquad\text{ for every }\xi\in\mathcal{O}.\label{eq:SpezialKonstruktionLiefertZerlegungDerEins}
\end{equation}
\end{lem}
\begin{proof}
First note that by Theorem \ref{thm:ZulaessigkeitVonbandbeschraenktenFunktionen},
$W_{\psi}\psi\in L^{1}(G)$. In addition, $W_{\psi}\psi\in L^{\infty}(G)$
by the Cauchy-Schwarz inequality, whence finally $W_{\psi}\psi\in L^{2}(G)$.
Now \cite[Lemma 9]{FuehrWaveletFramesAndAdmissibility} yields 
\[
C_{\psi}=\int_{H}\left|\widehat{\psi}\left(h^{T}\xi\right)\right|^{2}\,\text{d}h~.\qedhere
\]

\end{proof}
With these preparations it is now easy to show that $\left(\varphi_{U_{i}}\right)_{i\in I}$
indeed yields (a multiple of) a BAPU if the sets $\left(U_{i}\right)_{i\in I}$
form a suitable partition of $H$.
\begin{thm}
\label{thm:SpezialBAPUZusammenfassung}Let $\left(h_{i}\right)_{i\in I}$
be well-spread in $H$ with $H=\bigcup_{i\in I}h_{i}U$ for some precompact,
measurable $U\subset H$. Furthermore, let $\psi\in\mathcal{S}\left(\mathbb{R}^{d}\right)\setminus\left\{ 0\right\} $
satisfy $\widehat{\psi}\in\mathcal{D}\left(\mathcal{O}\right)$.

Let $\left(i_{n}\right)_{n\in\mathbb{N}}$ be an enumeration of $I$
(note that $I$ is countably infinite by Lemma \ref{lem:WellSpreadExistenz})
and define $U_{i_{n}}:=h_{i_{n}}U\setminus\bigcup_{m=1}^{n-1}h_{i_{m}}U$
for $n\in\mathbb{N}$. Then $\left(U_{i}\right)_{i\in I}$ is a measurable
partition of $H$ satisfying $U_{i}\subset h_{i}U$ for all $i\in I$.

Define $Q:=U^{-T}\left(\widehat{\psi}^{-1}\left(\mathbb{C}^{\ast}\right)\right)\subset\mathcal{O}$.
Then $Q\subset\mathcal{O}$ is open and precompact satisfying $\overline{Q}\subset\mathcal{O}$
and $\mathcal{O}=\bigcup_{i\in I}h_{i}^{-T}Q$, so that $\mathcal{Q}=\left(h_{i}^{-T}Q\right)_{i\in I}$
is a covering of $\mathcal{O}$ induced by $H$. Finally, $\left(\varphi_{i}\right)_{i\in I}:=\left(\frac{1}{C_{\psi}}\varphi_{U_{i}}\right)_{i\in I}$
defines a BAPU that is subordinate to this covering.\end{thm}
\begin{proof}
It is easy to see that $\left(U_{i}\right)_{i\in I}$ forms a measurable
partition of $H$. Note that 
\[
Q=\bigcup_{h\in U}h^{-T}\left(\widehat{\psi}^{-1}\left(\mathbb{C}^{\ast}\right)\right)\subset\mathcal{O}
\]
is open as a union of open sets. Furthermore, we have $\overline{Q}\subset\overline{U}^{-T}\text{supp}\left(\smash{\widehat{\psi}}\right)\subset\mathcal{O}$,
so that $\overline{Q}\subset\mathcal{O}$ is compact. Because of $\psi\not\equiv0$
we also have $\widehat{\psi}\not\equiv0$, so that there exists some
$\xi_{1}\in\widehat{\psi}^{-1}\left(\mathbb{C}^{\ast}\right)\subset\mathcal{O}$.
We have $\xi_{1}=h^{T}\xi_{0}$ for some $h\in H$. This implies
\[
\bigcup_{i\in I}h_{i}^{-T}Q\supset\bigcup_{i\in I}\left(h_{i}U\right)^{-T}h^{T}\xi_{0}=H^{T}\xi_{0}=\mathcal{O}.
\]
By Lemma \ref{lem:BAPUKonstruktion} we have $\varphi_{U_{i}}\in\mathcal{D}\left(\mathcal{O}\right)$
with $\varphi_{U_{i}}\left(x\right)=0$ for 
\[
x\in\mathcal{O}\setminus\left(U_{i}^{-T}\left(\widehat{\psi}^{-1}\left(\mathbb{C}^{\ast}\right)\right)\right)\supset\mathcal{O}\setminus\left(\left(h_{i}U\right)^{-T}\left(\widehat{\psi}^{-1}\left(\mathbb{C}^{\ast}\right)\right)\right)=\mathcal{O}\setminus\left(h_{i}^{-T}Q\right).
\]
Furthermore, Lemma \ref{lem:BAPUFourierTransformation} yields 
\begin{eqnarray*}
\left\Vert \mathcal{F}^{-1}\varphi_{i}\right\Vert _{L^{1}\left(\mathbb{R}^{d}\right)} & = & \frac{1}{C_{\psi}}\left\Vert \mathcal{F}^{-1}\varphi_{U_{i}}\right\Vert _{L^{1}\left(\mathbb{R}^{d}\right)}\leq\mu_{H}\left(U_{i}\right)\cdot\frac{\left\Vert \gamma\right\Vert _{L^{1}\left(\mathbb{R}^{d}\right)}}{C_{\psi}}\\
 & \leq & \mu_{H}\left(h_{i}\overline{U}\right)\cdot\frac{\left\Vert \gamma\right\Vert _{L^{1}\left(\mathbb{R}^{d}\right)}}{C_{\psi}}=\mu_{H}\left(\overline{U}\right)\cdot\frac{\left\Vert \gamma\right\Vert _{L^{1}\left(\mathbb{R}^{d}\right)}}{C_{\psi}}=:C
\end{eqnarray*}
for $\gamma:=\mathcal{F}^{-1}\left(\smash{\left|\smash{\widehat{\psi}}\right|^{2}}\right)\in\mathcal{S}\left(\mathbb{R}^{d}\right)$.

Finally, Lemma \ref{lem:BAPUIstZerlegungDerEins} and the definition
of $\varphi_{U}$ (equation (\ref{eq:BAPUIdee})) show that for $\xi\in\mathcal{O}$
we have 
\begin{eqnarray*}
\sum_{i\in I}\varphi_{i}\left(\xi\right) & = & \frac{1}{C_{\psi}}\cdot\sum_{n=1}^{\infty}\varphi_{U_{i_{n}}}\left(\xi\right)\\
 & = & \frac{1}{C_{\psi}}\cdot\sum_{n=1}^{\infty}\int_{U_{i_{n}}}\left|\widehat{\psi}\left(h^{T}\xi\right)\right|^{2}\,\text{d}h\\
 & = & \frac{1}{C_{\psi}}\cdot\int_{\biguplus_{n\in\mathbb{N}}U_{i_{n}}}\left|\widehat{\psi}\left(h^{T}\xi\right)\right|^{2}\,\text{d}h\\
 & \overset{\text{Lemma }\ref{lem:BAPUIstZerlegungDerEins}}{=} & 1,
\end{eqnarray*}
because $\left(U_{i_{n}}\right)_{n\in\mathbb{N}}$ is a partition
of $H$. Thus, $\left(\varphi_{i}\right)_{i\in I}$ is recognized
as a $\mathcal{Q}$-BAPU.
\end{proof}

\section{Continuity of the Fourier transform from Coorbit spaces into Decomposition
spaces}

\label{sec:CoorbitAsDecomposition}In this section we will show that
the Fourier transform on $\text{Co}\left(L_{v}^{p,q}\right)$ as defined
in Corollary \ref{cor:FourierTrafoAufFeichtingerReservoir} is well-defined
and continuous as a map into the decomposition space $\mathcal{D}\left(\mathcal{Q},L^{p},\ell_{u}^{q}\right)$,
where $\mathcal{Q}$ is a covering of $\mathcal{O}$ induced by $H$
and $u$ is the transplant of a suitable weight on $H$.

We will first show this for $f\in\text{Co}\left(L_{v}^{p,q}\right)\cap\mathcal{S}\left(\mathbb{R}^{d}\right)$
and then use a density result (namely the \textbf{Atomic Decomposition}
in $\text{Co}\left(L_{v}^{p,q}\right)$, cf. \cite[Theorem 6.1]{FeichtingerCoorbit1})
to establish the result in the general case.

We start by explicitly computing the localizations 
\[
\mathcal{F}^{-1}\left(\varphi_{V}\cdot\widehat{f}\right)
\]
for an arbitrary measurable, precompact set $V\subset H$ and $f\in\mathcal{S}\left(\mathbb{R}^{d}\right)$
in terms of the wavelet transform $W_{\psi}f$. As the norm on $\text{Co}\left(L_{v}^{p,q}\right)$
is defined in terms of $W_{\psi}f$, this is the essential step in
our proof.

In the ensuing calculations, we will use the following elementary
result: 
\begin{lem}
\label{lem:FaltungDilatation}Let $f,g\in L^{1}\left(\mathbb{R}^{d}\right)$.
For $h\in\text{GL}\left(\mathbb{R}^{d}\right)$ and $x\in\mathbb{R}^{d}$
we have 
\[
\left(D_{h}\left(f\ast g\right)\right)\left(x\right)=\left|\det\left(h\right)\right|\cdot\left(\left(D_{h}f\right)\ast\left(D_{h}g\right)\right)\left(x\right),
\]
whenever either side of the equation is defined.
\end{lem}

\begin{lem}
\label{lem:LokalisierungenAusgedruecktDurchWaveletTransformation}Let
$f,\psi\in\mathcal{S}\left(\mathbb{R}^{d}\right)$ where $\widehat{\psi}$
has compact support in $\mathcal{O}$. Furthermore assume that $V\subset H$
is precompact and measurable. For $x\in\mathbb{R}^{d}$ we have 
\[
\left(\mathcal{F}^{-1}\left(\varphi_{V}\cdot\widehat{f}\right)\right)\left(x\right)=\int_{V}\left|\det\left(h\right)\right|^{-3/2}\cdot\left(\left(W_{\psi}f\right)\left(\cdot,h\right)\ast D_{h^{-T}}\psi\right)\left(x\right)\,{\rm d}h,
\]
with $\varphi_{V}$ defined in equation \eqref{eq:BAPUIdee}.\end{lem}
\begin{proof}
    Choose $\gamma:=\mathcal{F}^{-1}\left(\smash{\left|\smash{\widehat{\psi}}\right|^{2}}\right)\in\mathcal{S}\left(\mathbb{R}^{d}\right)$
as in Lemma \ref{lem:BAPUFourierTransformation}. By the convolution
theorem we have 
\[
\gamma=\mathcal{F}^{-1}\left(\widehat{\psi}\cdot\overline{\widehat{\psi}}\right)=\left(\mathcal{F}^{-1}\widehat{\psi}\right)\ast\left(\mathcal{F}^{-1}\overline{\widehat{\psi}}\right)=\psi\ast\psi^{\ast}
\]
with $\psi^{\ast}:\mathbb{R}^{d}\rightarrow\mathbb{C},x\mapsto\overline{\psi\left(-x\right)}$.

Using Lemma \ref{lem:FaltungDilatation} together with $\left(f\ast D_{h^{-T}}\psi^{\ast}\right)\left(x\right)=\left|\det\left(h\right)\right|^{1/2}\cdot\left(W_{\psi}f\right)\left(x,h\right)$
and basic properties of convolution products, we obtain
\[
f\ast D_{h^{-T}}\gamma=\left|\det\left(h\right)\right|^{-1/2}\cdot\left(\left(W_{\psi}f\right)\left(\cdot,h\right)\right)\ast\left(D_{h^{-T}}\psi\right).
\]
 Now Lemma \ref{lem:BAPUFourierTransformation} yields the representation
\[
\left(\mathcal{F}^{-1}\varphi_{V}\right)\left(x\right)=\int_{V}\frac{\gamma\left(h^{-1}x\right)}{\left|\det\left(h\right)\right|}\,\text{d}h\qquad\text{for }x\in\mathbb{R}^{d}
\]
for the inverse Fourier transform of $\varphi_{V}$. Putting everything
together and using the convolution theorem and Fubini's theorem, we
see
\begin{align*}
\left(\mathcal{F}^{-1}\left(\varphi_{V}\cdot\widehat{f}\right)\right)\left(x\right) & =\left(\left(\mathcal{F}^{-1}\varphi_{V}\right)\ast\mathcal{F}^{-1}\widehat{f}\right)\left(x\right)\\
 & =\int_{\mathbb{R}^{d}}f\left(y\right)\cdot\left(\mathcal{F}^{-1}\varphi_{V}\right)\left(x-y\right)\,{\rm d}y\\
 & =\int_{V}\int_{\mathbb{R}^{d}}f\left(y\right)\cdot\frac{\gamma\left(h^{-1}\left(x-y\right)\right)}{\left|\det\left(h\right)\right|}\,{\rm d}y\,{\rm d}h\\
 & =\int_{V}\left|\det\left(h\right)\right|^{-1}\left(f\ast D_{h^{-T}}\gamma\right)\left(x\right)\,{\rm d}h\\
 & =\int_{V}\left|\det\left(h\right)\right|^{-3/2}\left(\left(W_{\psi}f\right)\left(\cdot,h\right)\ast D_{h^{-T}}\psi\right)\left(x\right)\,{\rm d}h.
\end{align*}
Here we used Fubini's theorem, as justified by
\begin{align*}
 & \phantom{\leq \,\,\,}\int_{V}\left|\det\left(h\right)\right|^{-1}\cdot\int_{\mathbb{R}^{d}}\left|f\left(y\right)\right|\cdot\left|\gamma\left(h^{-1}\left(x-y\right)\right)\right|\,\text{d}y\,\text{d}h\\
 & \leq\int_{V}\left|\det\left(h\right)\right|^{-1}\cdot\left\Vert f\right\Vert _{L^{1}\left(\mathbb{R}^{d}\right)}\cdot\left\Vert \gamma\right\Vert _{\sup}\,\text{d}h\\
 & \leq\mu_{H}\left(\overline{V}\right)\cdot\left\Vert f\right\Vert _{L^{1}\left(\mathbb{R}^{d}\right)}\cdot\left\Vert \gamma\right\Vert _{\sup}\cdot\max_{h\in\overline{V}^{-1}}\left|\det\left(h\right)\right|<\infty~.\qedhere
\end{align*}

\end{proof}
With this representation of the localized ``pieces'' of $\widehat{f}$,
we are now ready to prove the continuity of $\mathcal{F}:\mathcal{S}\left(\mathbb{R}^{d}\right)\cap\text{Co}\left(L_{v}^{p,q}\right)\rightarrow\mathcal{D}\left(\mathcal{Q},L^{p},\ell_{u}^{q}\right)$,
where $u$ is a transplant of $h\mapsto\left|\det\left(h^{-1}\right)\right|^{\frac{1}{2}-\frac{1}{q}}\cdot v\left(h^{-1}\right)$. 
\begin{lem}
\label{lem:FouriertransformationStetigAufCoorbitGeschnittenSchwartz}Let
$p,q\in\left[1,\infty\right]$. Then 
\[
v':H\rightarrow\left(0,\infty\right),h\mapsto\left|\det\left(h^{-1}\right)\right|^{\frac{1}{2}-\frac{1}{q}}\cdot v\left(h^{-1}\right)
\]
is moderate with respect to the measurable, locally bounded, submultiplicative
weight 
\[
v_{0}':H\rightarrow\left(0,\infty\right),h\mapsto\left|\det\left(h^{-1}\right)\right|^{\frac{1}{2}-\frac{1}{q}}\cdot v_{0}\left(h^{-1}\right).
\]

Choose an arbitrary $\psi\in\mathcal{S}\left(\mathbb{R}^{d}\right)\setminus\left\{ 0\right\} $
with $\widehat{\psi}\in\mathcal{D}\left(\mathcal{O}\right)$ and let
$\left(h_{i}\right)_{i\in I}$ be well-spread in $H$ with $H=\bigcup_{i\in I}h_{i}U$
for some precompact unit-neighborhood $U\subset H$.

Let $Q:=U^{-T}\left(\widehat{\psi}^{-1}\left(\mathbb{C}^{\ast}\right)\right)$
and $\mathcal{Q}=\left(h_{i}^{-T}Q\right)_{i\in I}$ be the corresponding
induced covering of $\mathcal{O}$ (cf. Theorem \ref{thm:InduzierteUeberdeckungKonstruktion}).
Let $u:\mathcal{O}\rightarrow\left(0,\infty\right)$ be a transplant
of $v'$ onto $\mathcal{O}$. 

Then there is a constant $C>0$ satisfying 
\[
\left\Vert \widehat{f}\right\Vert _{\mathcal{D}\left(\mathcal{Q},L^{p},\ell_{u}^{q}\right)}\leq C\cdot\left\Vert f\right\Vert _{\text{Co}\left(L_{v}^{p,q}\right)}<\infty\text{ for all }f\in\mathcal{S}\left(\mathbb{R}^{d}\right)\cap\text{Co}\left(L_{v}^{p,q}\right).
\]
\end{lem}
\begin{rem*}
In the above setting, for a suitable choice of $u$, one possible
discretization of $u$ with respect to $\mathcal{Q}$ is given by
\[
u_{i}=\left|\det\left(h_{i}\right)\right|^{\frac{1}{2}-\frac{1}{q}}\cdot v\left(h_{i}\right)
\]
and any (different) choice yields a weight on $I$ that is equivalent
to $\left(u_{i}\right)_{i\in I}$.\end{rem*}
\begin{proof}
Lemma \ref{lem:SubmultiplikativAbschlussEigenschaften} shows that
$v_{0}^{\vee}:H\rightarrow\left(0,\infty\right),h\mapsto v_{0}\left(h^{-1}\right)$
and hence also $v_{0}'$ are submultiplicative. It is easy to see
that $v^{\vee}$ is moderate with respect to $v_{0}^{\vee}$. This
implies that $v'$ is moderate with respect to $v_{0}'$. It is clear
that with $v_{0}$ also $v_{0}'$ is locally bounded and measurable.

Choose $\left(\varphi_{i}\right)_{i\in I}=\left(\frac{1}{C_{\psi}}\varphi_{U_{i}}\right)_{i\in I}$
as in Theorem \ref{thm:SpezialBAPUZusammenfassung}. By that theorem,
$\left(\varphi_{i}\right)_{i\in I}$ is a $\mathcal{Q}$-BAPU.

Let $f\in\mathcal{S}\left(\mathbb{R}^{d}\right)\cap\text{Co}\left(L_{v}^{p,q}\right)$.
We use Lemma \ref{lem:LokalisierungenAusgedruecktDurchWaveletTransformation},
Minkowski's inequality for integrals (cf. \cite[Theorem 6.19]{FollandRA})
and Young's inequality to calculate, for arbitrary precompact and
measurable $V\subset H$: 
\begin{eqnarray}
\left\Vert \mathcal{F}^{-1}\left(\varphi_{V}\cdot\widehat{f}\right)\right\Vert _{L^{p}\left(\mathbb{R}^{d}\right)}\!\!\!\! & \overset{\text{Lemma }\ref{lem:LokalisierungenAusgedruecktDurchWaveletTransformation}}{\leq} & \int_{V}\left|\det\left(h\right)\right|^{-3/2}\left\Vert \left(W_{\psi}f\right)\left(\cdot,h\right)\ast D_{h^{-T}}\psi\right\Vert _{L^{p}\left(\mathbb{R}^{d}\right)}\,\text{d}h\nonumber \\
 & \overset{\text{Young}}{\leq} & \int_{V}\left|\det\left(h\right)\right|^{-3/2}\left\Vert D_{h^{-T}}\psi\right\Vert _{L^{1}\left(\mathbb{R}^{d}\right)}\cdot\left\Vert \left(W_{\psi}f\right)\left(\cdot,h\right)\right\Vert _{L^{p}\left(\mathbb{R}^{d}\right)}\,\text{d}h\nonumber \\
 & = & \left\Vert \psi\right\Vert _{L^{1}\left(\mathbb{R}^{d}\right)}\cdot\int_{V}\left|\det\left(h\right)\right|^{-1/2}\left\Vert \left(W_{\psi}f\right)\left(\cdot,h\right)\right\Vert _{L^{p}\left(\mathbb{R}^{d}\right)}\,\text{d}h\label{eq:LokalisierteFourierTrafoGeminkowskit}\\
 & = & \left\Vert \psi\right\Vert _{L^{1}\left(\mathbb{R}^{d}\right)}\cdot\mu_{H}\left(V\right)\cdot\fint_{V}\left|\det\left(h\right)\right|^{-1/2}\left\Vert \left(W_{\psi}f\right)\left(\cdot,h\right)\right\Vert _{L^{p}\left(\mathbb{R}^{d}\right)}\,\text{d}h,\nonumber 
\end{eqnarray}
where we assumed $\mu_{H}\left(V\right)>0$ in the last step.

For $i\in I$ choose $V=U_{i}\subset h_{i}U$ (cf. Theorem \ref{thm:SpezialBAPUZusammenfassung})
and assume $\mu_{H}\left(U_{i}\right)>0$. In the case $q\in\left[1,\infty\right)$,
Jensen's inequality yields (by convexity of $\mathbb{R}\rightarrow\mathbb{R}_{+},x\mapsto\left|x\right|^{q}$)
the estimate
\begin{align}
 & \phantom{\leq \,\,\,}\left\Vert \mathcal{F}^{-1}\left(\varphi_{U_{i}}\cdot\widehat{f}\right)\right\Vert _{L^{p}\left(\mathbb{R}^{d}\right)}^{q}\nonumber \\
 & \leq\left(\left\Vert \psi\right\Vert _{L^{1}\left(\mathbb{R}^{d}\right)}\cdot\mu_{H}\left(U_{i}\right)\right)^{q}\cdot\fint_{U_{i}}\left(\left|\det\left(h\right)\right|^{-1/2}\cdot\left\Vert \left(W_{\psi}f\right)\left(\cdot,h\right)\right\Vert _{L^{p}\left(\mathbb{R}^{d}\right)}\right)^{q}\,\text{d}h\nonumber \\
 & \leq\left\Vert \psi\right\Vert _{L^{1}\left(\mathbb{R}^{d}\right)}^{q}\cdot\left(\mu_{H}\left(\overline{U}\right)\right)^{q-1}\cdot\int_{U_{i}}\left(\left|\det\left(h\right)\right|^{\frac{1}{q}-\frac{1}{2}}\cdot\left\Vert \left(W_{\psi}f\right)\left(\cdot,h\right)\right\Vert _{L^{p}\left(\mathbb{R}^{d}\right)}\right)^{q}\,\frac{\text{d}h}{\left|\det\left(h\right)\right|},\qquad\label{eq:LokalisierungGejensent}
\end{align}
where we used $\mu_{H}\left(U_{i}\right)\leq\mu_{H}\left(h_{i}U\right)\leq\mu_{H}\left(\overline{U}\right)$
and $q-1\geq0$ in the last step. Note that the above estimate is
trivial in the case $\mu_{H}\left(U_{i}\right)=0$.

For $h\in U_{i}\subset h_{i}\overline{U}$ we now have $h=h_{i}u$
for some $u\in\overline{U}$. With $C_{1}:=\min_{k\in\overline{U}}\left|\det\left(k\right)\right|$
and $C_{2}:=\max_{k\in\overline{U}}\left|\det\left(k\right)\right|$,
we thus get 
\[
\frac{\left|\det\left(h\right)\right|}{\left|\det\left(h_{i}\right)\right|}=\left|\det\left(u\right)\right|\in\left[C_{1},C_{2}\right].
\]
As the map $\mathbb{R}_{+}\rightarrow\mathbb{R}_{+},x\mapsto x^{\frac{1}{q}-\frac{1}{2}}$
is monotonic (increasing for $q\leq2$ and decreasing for $q\geq2$),
we derive 
\[
\frac{\left|\det\left(h\right)\right|^{\frac{1}{q}-\frac{1}{2}}}{\left|\det\left(h_{i}\right)\right|^{\frac{1}{q}-\frac{1}{2}}}\in\left[\min\left\{ C_{1}^{\frac{1}{q}-\frac{1}{2}},C_{2}^{\frac{1}{q}-\frac{1}{2}}\right\} ,\max\left\{ C_{1}^{\frac{1}{q}-\frac{1}{2}},C_{2}^{\frac{1}{q}-\frac{1}{2}}\right\} \right]=:\left[C_{3},C_{4}\right].
\]
Now let $C_{5}:=v_{0}\left(1_{H}\right)\cdot\sup_{k\in\overline{U}^{-1}}v_{0}\left(k\right)$.
Then we have 
\[
v\left(h_{i}\right)=v\left(1_{H}hu^{-1}\right)\leq v_{0}\left(1_{H}\right)\cdot v\left(h\right)\cdot v_{0}\left(u^{-1}\right)\leq C_{5}\cdot v\left(h\right)
\]
for all $h=h_{i}u\in U_{i}$.

Putting all this together and setting $C_{6}:=\left\Vert \psi\right\Vert _{L^{1}\left(\mathbb{R}^{d}\right)}\cdot\left(\mu_{H}\left(\overline{U}\right)\right)^{1-\frac{1}{q}}/C_{\psi}$,
we arrive at 
\begin{eqnarray*}
 &  & \sum_{i\in I}\left(\left|\det\left(h_{i}\right)\right|^{\frac{1}{2}-\frac{1}{q}}\cdot v\left(h_{i}\right)\cdot\left\Vert \mathcal{F}^{-1}\left(\varphi_{i}\cdot\widehat{f}\right)\right\Vert _{L^{p}\left(\mathbb{R}^{d}\right)}\right)^{q}\\
 & \overset{\text{Eq. }\eqref{eq:LokalisierungGejensent}}{\underset{\varphi_{i}=\varphi_{U_{i}}/C_{\psi}}{\leq}} & C_{6}^{q}\cdot\sum_{i\in I}\left[\left(\left|\det\left(h_{i}\right)\right|^{\frac{1}{2}-\frac{1}{q}}v\left(h_{i}\right)\right)^{q}\cdot\int_{U_{i}}\left(\left|\det\left(h\right)\right|^{\frac{1}{q}-\frac{1}{2}}\left\Vert \left(W_{\psi}f\right)\left(\cdot,h\right)\right\Vert _{L^{p}\left(\mathbb{R}^{d}\right)}\right)^{q}\frac{\text{d}h}{\left|\det\left(h\right)\right|}\right]\\
 & \leq & C_{6}^{q}\cdot\sum_{i\in I}\left[\int_{U_{i}}\left(\frac{\left|\det\left(h\right)\right|^{\frac{1}{q}-\frac{1}{2}}}{\left|\det\left(h_{i}\right)\right|^{\frac{1}{q}-\frac{1}{2}}}\cdot v\left(h_{i}\right)\cdot\left\Vert \left(W_{\psi}f\right)\left(\cdot,h\right)\right\Vert _{L^{p}\left(\mathbb{R}^{d}\right)}\right)^{q}\,\frac{\text{d}h}{\left|\det\left(h\right)\right|}\right]\\
 & \leq & C_{4}^{q}C_{5}^{q}C_{6}^{q}\cdot\sum_{i\in I}\left[\int_{U_{i}}\left(v\left(h\right)\cdot\left\Vert \left(W_{\psi}f\right)\left(\cdot,h\right)\right\Vert _{L^{p}\left(\mathbb{R}^{d}\right)}\right)^{q}\,\frac{\text{d}h}{\left|\det\left(h\right)\right|}\right]\\
 & = & C_{4}^{q}C_{5}^{q}C_{6}^{q}\cdot\int_{H}\left(v\left(h\right)\cdot\left\Vert \left(W_{\psi}f\right)\left(\cdot,h\right)\right\Vert _{L^{p}\left(\mathbb{R}^{d}\right)}\right)^{q}\,\frac{\text{d}h}{\left|\det\left(h\right)\right|}\\
 & = & C_{4}^{q}C_{5}^{q}C_{6}^{q}\cdot\left\Vert W_{\psi}f\right\Vert _{L_{v}^{p,q}\left(G\right)}^{q}<\infty.
\end{eqnarray*}
This settles the case $q<\infty$. In the remaining case $q=\infty$,
we use equation (\ref{eq:LokalisierteFourierTrafoGeminkowskit}) to
estimate 
\begin{eqnarray*}
 &  & \sup_{i\in I}\left|\det\left(h_{i}\right)\right|^{\frac{1}{2}-\frac{1}{q}}\cdot v\left(h_{i}\right)\cdot\left\Vert \mathcal{F}^{-1}\left(\varphi_{i}\cdot\widehat{f}\right)\right\Vert _{L^{p}\left(\mathbb{R}^{d}\right)}\\
 & \leq & C_{\psi}^{-1}\left\Vert \psi\right\Vert _{L^{1}\left(\mathbb{R}^{d}\right)}\cdot\sup_{i\in I}\int_{U_{i}}\left(\frac{\left|\det\left(h\right)\right|}{\left|\det\left(h_{i}\right)\right|}\right)^{-1/2}\cdot v\left(h_{i}\right)\cdot\left\Vert \left(W_{\psi}f\right)\left(\cdot,h\right)\right\Vert _{L^{p}\left(\mathbb{R}^{d}\right)}\,\text{d}h\\
 & \leq & C_{\psi}^{-1}C_{1}^{-1/2}C_{5}\cdot\left\Vert \psi\right\Vert _{L^{1}\left(\mathbb{R}^{d}\right)}\cdot\sup_{i\in I}\int_{U_{i}}v\left(h\right)\cdot\left\Vert \left(W_{\psi}f\right)\left(\cdot,h\right)\right\Vert _{L^{p}\left(\mathbb{R}^{d}\right)}\,\text{d}h\\
 & \overset{q=\infty}{\leq} & C_{\psi}^{-1}C_{1}^{-1/2}C_{5}\cdot\left\Vert \psi\right\Vert _{L^{1}\left(\mathbb{R}^{d}\right)}\cdot\mu_{H}\left(U_{i}\right)\cdot\left\Vert W_{\psi}f\right\Vert _{L_{v}^{p,q}}\\
 & \leq & C_{\psi}^{-1}C_{1}^{-1/2}C_{5}\cdot\left\Vert \psi\right\Vert _{L^{1}\left(\mathbb{R}^{d}\right)}\cdot\mu_{H}\left(\overline{U}\right)\cdot\left\Vert W_{\psi}f\right\Vert _{L_{v}^{p,q}}<\infty,
\end{eqnarray*}
where we again used $\mu_{H}\left(U_{i}\right)\leq\mu_{H}\left(h_{i}U\right)\leq\mu_{H}\left(\overline{U}\right)$.

Now note that $u_{i}=u\left(h_{i}^{-T}\xi_{0}\right)=v'\left(h_{i}^{-1}\right)=\left|\det\left(h_{i}\right)\right|^{\frac{1}{2}-\frac{1}{q}}\cdot v\left(h_{i}\right)$
is a valid discretization of a suitable transplant of $v'$ onto $H$
(cf. remark \ref{rem:SpezielleDiskretisierung}). Thus, the above
estimates show 
\[
\left\Vert \widehat{f}\right\Vert _{\mathcal{D}\left(\mathcal{Q},L^{p},\ell_{u}^{q}\right)}\leq\begin{cases}
C_{4}C_{5}C_{6}\cdot\left\Vert W_{\psi}f\right\Vert _{L_{v}^{p,q}\left(G\right)}=C_{4}C_{5}C_{6}\cdot\left\Vert f\right\Vert _{\text{Co}\left(L_{v}^{p,q}\right)}, & q<\infty,\\
C_{\psi}^{-1}C_{1}^{-1/2}C_{5}\cdot\left\Vert \psi\right\Vert _{L^{1}\left(\mathbb{R}^{d}\right)}\cdot\mu_{H}\left(\overline{U}\right)\cdot\left\Vert f\right\Vert _{\text{Co}\left(L_{v}^{p,q}\right)}, & q=\infty,
\end{cases}
\]
where all constants $C_{1},\dots,C_{6}$ are independent of $f$.
By Lemma \ref{lem:GewichtsTransplantation} and Lemma \ref{lem:GewichtsDiskretisierung},
any two discretizations of transplants of $v'$ yield equivalent norms
on $\mathcal{D}\left(\mathcal{Q},L^{p},\ell_{u}^{q}\right)$. Thus,
the proof is complete.
\end{proof}
In order to establish the general result, we first show that the decomposition
space $\mathcal{D}\left(\mathcal{Q},L^{p},\ell_{u}^{q}\right)$ satisfies
a form of the Fatou property.
\begin{lem}
\label{lem:FatouFuerDecompositionRaeume}Let $\emptyset\neq U\subset\mathbb{R}^{d}$
be open and assume that $\mathcal{Q}=\left(Q_{i}\right)_{i\in I}$
is a decomposition covering of $U$. Let $u:U\rightarrow\left(0,\infty\right)$
be $\mathcal{Q}$-moderate and let $p,q\in\left[1,\infty\right]$
be arbitrary. Assume that $\left(f_{n}\right)_{n\in\mathbb{N}}$ is
a sequence in $\mathcal{D}\left(\mathcal{Q},L^{p},\ell_{u}^{q}\right)$
that satisfies $f_{n}\xrightarrow[n\rightarrow\infty]{}f\in\mathcal{D}'\left(U\right)$
where convergence is to be understood in the weak-$\ast$-sense, i.e.
pointwise on $\mathcal{D}\left(U\right)$.

Finally assume that $\liminf_{n\rightarrow\infty}\left\Vert f_{n}\right\Vert _{\mathcal{D}\left(\mathcal{Q},L^{p},\ell_{u}^{q}\right)}$
is finite. Then $f\in\mathcal{D}\left(\mathcal{Q},L^{p},\ell_{u}^{q}\right)$
holds with 
\[
\left\Vert f\right\Vert _{\mathcal{D}\left(\mathcal{Q},L^{p},\ell_{u}^{q}\right)}\leq\liminf_{n\rightarrow\infty}\left\Vert f_{n}\right\Vert _{\mathcal{D}\left(\mathcal{Q},L^{p},\ell_{u}^{q}\right)}.
\]
\end{lem}
\begin{proof}
Let $\left(\varphi_{i}\right)_{i\in I}$ be a $\mathcal{Q}$-BAPU.
For $i\in I$ and $x\in\mathbb{R}^{d}$, \cite[Theorem 7.23]{RudinFA}
applied to $\mathcal{F}^{-1}$ instead of $\mathcal{F}$ yields, with
$e_{x}:\mathbb{R}^{d}\rightarrow\mathbb{C},\xi\mapsto e^{2\pi i\left\langle x,\xi\right\rangle }$,
\begin{eqnarray*}
\left(\mathcal{F}^{-1}\left(\varphi_{i}f\right)\right)\left(x\right) & = & \left(\varphi_{i}f\right)\left(e_{x}\right)=f\left(\varphi_{i}e_{x}\right)\\
 & \overset{\varphi_{i}e_{x}\in\mathcal{D}\left(U\right)}{=} & \lim_{n\rightarrow\infty}f_{n}\left(\varphi_{i}e_{x}\right)=\lim_{n\rightarrow\infty}\left(\varphi_{i}f_{n}\right)\left(e_{x}\right)\\
 & = & \lim_{n\rightarrow\infty}\left(\mathcal{F}^{-1}\left(\varphi_{i}f_{n}\right)\right)\left(x\right),
\end{eqnarray*}
where we used the fact that $\varphi_{i}f$ and $\varphi_{i}f_{n}$
are distributions with compact support so that \cite[Theorem 7.23]{RudinFA}
is applicable.

Using the Fatou property of $L^{p}\left(\mathbb{R}^{d}\right)$, we
see $\mathcal{F}^{-1}\left(\varphi_{i}f\right)\in L^{p}\left(\mathbb{R}^{d}\right)$
for all $i\in I$ with 
\[
\left\Vert \mathcal{F}^{-1}\left(\varphi_{i}f\right)\right\Vert _{L^{p}\left(\mathbb{R}^{d}\right)}\leq\liminf_{n\rightarrow\infty}\left\Vert \mathcal{F}^{-1}\left(\varphi_{i}f_{n}\right)\right\Vert _{L^{p}\left(\mathbb{R}^{d}\right)}\leq\liminf_{n\rightarrow\infty}\frac{1}{u_{i}}\left\Vert f_{n}\right\Vert _{\mathcal{D}\left(\mathcal{Q},L^{p},\ell_{u}^{q}\right)}<\infty,
\]
where $\left(u_{i}\right)_{i\in I}$ is the chosen discretization
of $u$.

As $\ell^{q}\left(I\right)$ also enjoys the Fatou property and is
solid, we finally derive 
\begin{eqnarray*}
\left\Vert f\right\Vert _{\mathcal{D}\left(\mathcal{Q},L^{p},\ell_{u}^{q}\right)} & = & \left\Vert \left(u_{i}\cdot\left\Vert \mathcal{F}^{-1}\left(\varphi_{i}f\right)\right\Vert _{L^{p}\left(\mathbb{R}^{d}\right)}\right)\right\Vert _{\ell^{q}\left(I\right)}\\
 & \leq & \left\Vert \left(u_{i}\cdot\liminf_{n\rightarrow\infty}\left\Vert \mathcal{F}^{-1}\left(\varphi_{i}f_{n}\right)\right\Vert _{L^{p}\left(\mathbb{R}^{d}\right)}\right)\right\Vert _{\ell^{q}\left(I\right)}\\
 & \leq & \liminf_{n\rightarrow\infty}\left\Vert \left(u_{i}\cdot\left\Vert \mathcal{F}^{-1}\left(\varphi_{i}f_{n}\right)\right\Vert _{L^{p}\left(\mathbb{R}^{d}\right)}\right)\right\Vert _{\ell^{q}\left(I\right)}\\
 & = & \liminf_{n\rightarrow\infty}\left\Vert f_{n}\right\Vert _{\mathcal{D}\left(\mathcal{Q},L^{p},\ell_{u}^{q}\right)}<\infty.
\end{eqnarray*}
In particular, we have $f\in\mathcal{D}\left(\mathcal{Q},L^{p},\ell_{u}^{q}\right)$.
\end{proof}
We now prove the first half of our claimed isomorphism between $\text{Co}\left(L_{v}^{p,q}\right)$
and $\mathcal{D}\left(\mathcal{Q},L^{p},\ell_{u}^{q}\right)$, namely
the continuity of the Fourier transform from $\text{Co}\left(L_{v}^{p,q}\right)$
to $\mathcal{D}\left(\mathcal{Q},L^{p},\ell_{u}^{q}\right)$. The
proof uses the density (in a suitable topology) of $\mathcal{S}\left(\mathbb{R}^{d}\right)\cap\text{Co}\left(L_{v}^{p,q}\right)$
in $\text{Co}\left(L_{v}^{p,q}\right)$ together with Lemma \ref{lem:FouriertransformationStetigAufCoorbitGeschnittenSchwartz},
where we use Lemma \ref{lem:FatouFuerDecompositionRaeume} to pass
to the limit.
\begin{thm}
\label{thm:FourierStetigVonCoorbitNachDecomposition}Let $v:H\rightarrow\left(0,\infty\right)$
be measurable and moderate with respect to the measurable, locally
bounded, submultiplicative weight $v_{0}:H\rightarrow\left(0,\infty\right)$.
Let $p,q\in\left[1,\infty\right]$, choose $v'$ as in Lemma \ref{lem:FouriertransformationStetigAufCoorbitGeschnittenSchwartz}
and let $u:\mathcal{O}\rightarrow\left(0,\infty\right)$ be a transplant
of $v'$ onto $\mathcal{O}$. Finally, assume that $\mathcal{Q}$
is an arbitrary decomposition covering of $\mathcal{O}$ induced by
$H$.

Then the Fourier transform 
\[
\mathcal{F}:\text{Co}\left(L_{v}^{p,q}\right)\rightarrow\mathcal{D}\left(\mathcal{Q},L^{p},\ell_{u}^{q}\right),f\mapsto\mathcal{F}f
\]
with 
\[
\mathcal{F}f:\mathcal{D}\left(\mathcal{O}\right)\rightarrow\mathbb{C},g\mapsto f\left(\mathcal{F}^{-1}\overline{g}\right)\text{ for }f\in\text{Co}\left(L_{v}^{p,q}\right)
\]
defined as in Corollary \ref{cor:FourierTrafoAufFeichtingerReservoir}
is a well-defined, continuous linear map.\end{thm}
\begin{proof}
By definition of an induced covering, there is a well-spread family
$\left(h_{i}\right)_{i\in I}$ in $H$ and a precompact subset $Q\subset\mathbb{R}^{d}$
that satisfies $\mathcal{O}=\bigcup_{i\in I}h_{i}^{-T}Q$ and $\overline{Q}\subset\mathcal{O}$
as well as $\mathcal{Q}=\left(h_{i}^{-T}Q\right)_{i\in I}$. As $\left(h_{i}\right)_{i\in I}$
is well-spread, there exists a precompact set $U\subset H$ that satisfies
$H=\bigcup_{i\in I}h_{i}U$. Choose an arbitrary $\psi\in\mathcal{S}\left(\mathbb{R}^{d}\right)\setminus\left\{ 0\right\} $
with $\widehat{\psi}\in\mathcal{D}\left(\mathcal{O}\right)$. Define
$Q':=U^{-T}\left(\widehat{\psi}^{-1}\left(\mathbb{C}^{\ast}\right)\right)$
and let $\mathcal{Q}'=\left(h_{i}^{-T}Q'\right)_{i\in I}$ be the
corresponding induced covering of $\mathcal{O}$ (cf. Theorem \ref{thm:SpezialBAPUZusammenfassung}).
By Corollary \ref{cor:InduzierteUeberdeckLiefernGleicheRaeumeBeiGleicherFamilieH}
we have $\mathcal{D}\left(\mathcal{Q},L^{p},\ell_{u}^{q}\right)=\mathcal{D}\left(\mathcal{Q}',L^{p},\ell_{u}^{q}\right)$
with equivalent norms, so that it suffices to consider $\mathcal{Q}'$
instead of $\mathcal{Q}$.

Let $w:H\rightarrow\left(0,\infty\right)$ be defined as in Lemma
\ref{lem:CoorbitVoraussetzungen}. Theorem \ref{thm:ZulaessigkeitVonbandbeschraenktenFunktionen}
shows that $\psi\in\mathcal{B}_{w}$ is a ``better vector'', so
that by the Atomic Decomposition Theorem \cite[Theorem 6.1]{FeichtingerCoorbit1},
there is some unit neighborhood $V\subset G$, such that for every
$V$-dense and relatively separated family $X=\left(x_{j}\right)_{j\in J}$
in $G$ the following are true: 
\begin{enumerate}
\item There is a bounded linear \textbf{analysis operator} $A:\text{Co}\left(L_{v}^{p,q}\right)\rightarrow\left(L_{v}^{p,q}\right)_{d}\left(X\right)$
such that for every $f\in\text{Co}\left(L_{v}^{p,q}\right)$ we have
\begin{equation}
f=\sum_{j\in J}\left[\left(Af\right)_{j}\cdot\pi\left(x_{j}\right)\psi\right]\label{eq:AnalyseOperator}
\end{equation}
with convergence (at least) in the weak-$\ast$-topology on $\left(\mathcal{H}_{w}^{1}\right)^{\neg}$. 
\item The \textbf{synthesis operator} 
\[
S:\left(L_{v}^{p,q}\right)_{d}\left(X\right)\rightarrow\text{Co}\left(L_{v}^{p,q}\right),\left(\lambda_{j}\right)_{j\in J}\mapsto\sum_{j\in J}\left[\lambda_{j}\cdot\pi\left(x_{j}\right)\psi\right]
\]
is well-defined and bounded. 
\end{enumerate}

Here, the space $\left(L_{v}^{p,q}\right)_{d}$ is the so-called \textbf{associated
discrete BK-space} to $L_{v}^{p,q}$ (cf. \cite[Definition 3.4]{FeichtingerCoorbit1}).
The only property of this space that we need is that it is a solid
sequence space, i.e. if $\left(\lambda_{j}\right)_{j\in J}$ and $\left(\gamma_{j}\right)_{j\in J}$
are sequences so that $\left|\lambda_{j}\right|\leq\left|\gamma_{j}\right|$
holds for all $j\in J$ and with $\left(\gamma_{j}\right)_{j\in J}\in\left(L_{v}^{p,q}\right)_{d}$,
then we have $\left(\lambda_{j}\right)_{j\in J}\in\left(L_{v}^{p,q}\right)_{d}$
with a corresponding norm estimate $\left\Vert \smash{\left(\lambda_{j}\right)_{j\in J}}\right\Vert _{L_{v}^{p,q}}\leq\left\Vert \smash{\left(\gamma_{j}\right)_{j\in J}}\right\Vert _{L_{v}^{p,q}}$.

Let $W\subset V$ be a compact unit neighborhood that satisfies $WW^{-1}\subset V$.
By Lemma \ref{lem:WellSpreadExistenz} it follows that there is a
countably infinite family $\left(g_{j}\right)_{j\in J}$ in $G$ that
is $V$-dense and $W$-separated. This family is a fortiori relatively
separated, so that the above results apply to $\left(x_{j}\right)_{j\in J}=\left(g_{j}\right)_{j\in J}$.

Let $\left(j_{n}\right)_{n\in\mathbb{N}}$ be an enumeration of $J$
and let $f\in\text{Co}\left(L_{v}^{p,q}\right)$. For $n\in\mathbb{N}$
define 
\[
f_{n}:=\sum_{\ell=1}^{n}\left[\left(Af\right)_{j_{\ell}}\cdot\pi\left(g_{j_{\ell}}\right)\psi\right]=S\left(Af\cdot\chi_{\left\{ j_{1},\dots,j_{n}\right\} }\right)\in\text{Co}\left(L_{v}^{p,q}\right).
\]
Now note that for $g_{j}=\left(x_{j},k_{j}\right)$ we have $\pi\left(g_{j}\right)\psi=\left|\det\left(k_{j}\right)\right|^{-1/2}\cdot L_{x_{j}}D_{k_{j}^{-T}}\psi\in\mathcal{S}\left(\mathbb{R}^{d}\right)$
and thus $f_{n}\in\mathcal{S}\left(\mathbb{R}^{d}\right)\cap\text{Co}\left(L_{v}^{p,q}\right)$.

The convergence in equation \eqref{eq:AnalyseOperator} in the weak-$\ast$-topology
on $\left(\mathcal{H}_{w}^{1}\right)^{\neg}$ and Corollary \ref{cor:FourierTrafoAufFeichtingerReservoir}
show $\mathcal{F}f_{n}\xrightarrow[n\rightarrow\infty]{}\mathcal{F}f$
with convergence in the weak-$\ast$-topology on $\mathcal{D}'\left(\mathcal{O}\right)$.
Finally note that $\mathcal{F}f_{n}$ coincides with the ``ordinary''
Fourier transform $\widehat{f_{n}}$ of $f_{n}\in\mathcal{S}\left(\mathbb{R}^{d}\right)\subset L^{2}\left(\mathbb{R}^{d}\right)$
by Remark \ref{rem:SpezialFourierTrafoIsFortsetzungVonNormaler}.
Hence, we get 
\begin{eqnarray*}
\left\Vert \mathcal{F}f_{n}\right\Vert _{\mathcal{D}\left(\mathcal{Q},L^{p},\ell_{u}^{q}\right)} & = & \left\Vert \widehat{f_{n}}\right\Vert _{\mathcal{D}\left(\mathcal{Q},L^{p},\ell_{u}^{q}\right)}\\
& \overset{\text{Lemma } \ref{lem:FouriertransformationStetigAufCoorbitGeschnittenSchwartz}}{\leq} & C\cdot\left\Vert f_{n}\right\Vert _{\text{Co}\left(L_{v}^{p,q}\right)}\\
 & = & C\cdot\left\Vert S\left(Af\cdot\chi_{\left\{ j_{1},\dots,j_{n}\right\} }\right)\right\Vert _{\text{Co}\left(L_{v}^{p,q}\right)}\\
 & \leq & C\cdot\vertiii S\cdot\left\Vert Af\cdot\chi_{\left\{ j_{1},\dots,j_{n}\right\} }\right\Vert _{\left(L_{v}^{p,q}\right)_{d}}\\
 & \overset{\text{solidity}}{\leq} & C\cdot\vertiii S\cdot\left\Vert Af\right\Vert _{\left(L_{v}^{p,q}\right)_{d}}\\
 & \leq & C\cdot\vertiii S\cdot\vertiii A\cdot\left\Vert f\right\Vert _{\text{Co}\left(L_{v}^{p,q}\right)}<\infty,
\end{eqnarray*}
where the constant $C$ is taken from Lemma \ref{lem:FouriertransformationStetigAufCoorbitGeschnittenSchwartz}
and is thus independent of $f$. Application of Lemma \ref{lem:FatouFuerDecompositionRaeume}
finishes the proof.\qedhere

\end{proof}

\section{Continuity of the inverse Fourier transform from Decomposition spaces
into Coorbit spaces}

\label{sec:InverseFourierTrafoStetig}In this section we show that
the inverse Fourier transform $\mathcal{F}^{-1}:\mathcal{D}\left(\mathcal{Q},L^{p},\ell_{u}^{q}\right)\rightarrow\text{Co}\left(L_{v}^{p,q}\right)$
is well-defined and continuous. This poses the problem that an element
$f$ of $\mathcal{D}\left(\mathcal{Q},L^{p},\ell_{u}^{q}\right)$
is a distribution $f\in\mathcal{D}'\left(\mathcal{O}\right)$ on $\mathcal{O}$
and not (necessarily) a tempered distribution (cf. the remark following
Definition \ref{def:DecompositionSpace}). Thus, it is not immediately
clear how to define the inverse Fourier transform $\mathcal{F}^{-1}f$
of $f$.

In order to solve this problem, we use the map
\[
\Theta:\left(\mathcal{H}_{w}^{1}\right)^{\neg}\rightarrow\left(\mathcal{F}\left(\mathcal{D}\left(\mathcal{O}\right)\right)\right)',f\mapsto\left(\varphi\mapsto f\left(\overline{\varphi}\right)\right)
\]
introduced in Corollary \ref{cor:FourierTrafoAufRaumAbgewaelzt} to
identify the coorbit space $\text{Co}\left(L_{v}^{p,q}\right)$ with
the alternative coorbit space
\[
\widetilde{\text{Co}}_{\psi}\left(L_{v}^{p,q}\right):=\left\{ f\in\left(\mathcal{F}\left(\mathcal{D}\left(\mathcal{O}\right)\right)\right)'\with W_{\psi}f\in L_{v}^{p,q}\left(G\right)\right\} .
\]
Here, we define the wavelet transform $W_{\psi}f$ for $f\in\left(\mathcal{F}\left(\mathcal{D}\left(\mathcal{O}\right)\right)\right)'$
and $\psi\in\mathcal{S}\left(\mathbb{R}^{d}\right)$ with $\widehat{\psi}\in\mathcal{D}\left(\mathcal{O}\right)$
by 
\begin{equation}
W_{\psi}f:G\rightarrow\mathbb{C},\left(x,h\right)\mapsto W_{\psi}f\left(x,h\right):=f\left(\overline{\pi\left(x,h\right)\psi}\right).\label{eq:WaveletTrafoAufSpezialReservoir}
\end{equation}
On the space $\widetilde{\text{Co}}_{\psi}\left(L_{v}^{p,q}\right)$, the
definition of the inverse fourier transform is then straightforward.

The following theorem makes the claimed identification of $\text{Co}\left(L_{v}^{p,q}\right)$
with $\widetilde{\text{Co}}_{\psi}\left(L_{v}^{p,q}\right)$ explicit:
\begin{thm}
\label{thm:IsomorphismusZuSpezialCoorbitRaum}Let $\psi\in\mathcal{S}\left(\mathbb{R}^{d}\right)\setminus\left\{ 0\right\} $
with $\widehat{\psi}\in\mathcal{D}\left(\mathcal{O}\right)$ be arbitrary.
Then the restriction of the map $\Theta$ to $\text{Co}\left(L_{v}^{p,q}\right)$
induces an isometric isomorphism
\[
\Theta:\text{Co}\left(L_{v}^{p,q}\right)\rightarrow\widetilde{\text{Co}}_{\psi}\left(L_{v}^{p,q}\right)
\]
as long as $\psi$ is used as an analyzing vector for $\text{Co}\left(L_{v}^{p,q}\right)$,
i.e. $\left\Vert f\right\Vert _{\text{Co}\left(L_{v}^{p,q}\right)}=\left\Vert W_{\psi}f\right\Vert _{L_{v}^{p,q}}$.
Here, $\widetilde{\text{Co}}_{\psi}\left(L_{v}^{p,q}\right)$ is endowed
with the norm $\left\Vert f\right\Vert _{\widetilde{\text{Co}}_{\psi}\left(L_{v}^{p,q}\right)}:=\left\Vert W_{\psi}f\right\Vert _{L_{v}^{p,q}}$. 

In particular, $\left\Vert \cdot\right\Vert _{\widetilde{\text{Co}}_{\psi}\left(L_{v}^{p,q}\right)}$
defines a norm on $\widetilde{\text{Co}}_{\psi}\left(L_{v}^{p,q}\right)$,
so that $W_{\psi}:\widetilde{\text{Co}}_{\psi}\left(L_{v}^{p,q}\right)\rightarrow L_{v}^{p,q}\left(G\right)$
is injective.

Furthermore, the above implies that the space $\widetilde{\text{Co}}_{\psi}\left(L_{v}^{p,q}\right)$
is independent of $\psi$ with equivalent norms for different choices.
We will thus write $\widetilde{\text{Co}}\left(L_{v}^{p,q}\right)$
in the future.
\end{thm}
The hard part is the claimed surjectivity of the restricted map $\Theta$.
In order to prove it, an important step is to show that the extended
wavelet transform $W_{\psi}f$ defined in equation \eqref{eq:WaveletTrafoAufSpezialReservoir} is
a continuous function that satisfies the reproduction
formula
\[
W_{\psi}f =W_{\psi}f \ast\frac{W_{\psi}\psi}{C_{\psi}}.
\]
After establishing these properties, we will give the proof of Theorem
\ref{thm:IsomorphismusZuSpezialCoorbitRaum}. First of all, we will
need the following technical Lemma regarding the continuity of modulation
on certain spaces. The proof is a straightforward application of the
Leibniz rule together with the fact that all derivatives of the complex
exponentials $x\mapsto\exp(i\langle x,\xi\rangle)$ are bounded on
compact sets and is therefore omitted.
\begin{lem}
\label{lem:ModulationStetigAufCkMitKompaktemTraeger}Let $\emptyset\neq U\subset\mathbb{R}^{d}$
be an open set and let $K\subset U$ be compact and $k\in\mathbb{N}_{0}$.
Then the map 
\[
\Phi:C_{K}^{k}\left(U\right)\times\mathbb{R}^{d}\rightarrow C_{K}^{k}\left(U\right),\left(f,\xi\right)\mapsto\left(x\mapsto e^{i\left\langle x,\xi\right\rangle }\cdot f\left(x\right)\right)
\]
is well-defined and continuous. Here, $C_{K}^{k}\left(U\right)$ is
the space 
\[
C_{K}^{k}\left(U\right):=\left\{ f\in C^{k}\left(U\right)\with\text{supp}\left(f\right)\subset K\right\} 
\]
with norm 
\[
\left\Vert f\right\Vert _{C_{K}^{k}\left(U\right)}:=\max_{\substack{\alpha\in\mathbb{N}_{0}^{d}\\
\left|\alpha\right|\leq k
}
}\sup_{x\in U}\left|\left(\partial^{\alpha}f\right)\left(x\right)\right|.
\]

\end{lem}
Using this result on the continuity of modulation, we can now show
that the extended wavelet transform as defined in equation \eqref{eq:WaveletTrafoAufSpezialReservoir}
actually defines a continuous function that satisfies the expected
reproduction formula. 
\begin{lem}
\label{lem:SpezialWaveletTrafoErfuelltReproduktionsFormel}Let $\psi\in\mathcal{S}\left(\mathbb{R}^{d}\right)$
with $\widehat{\psi}\in\mathcal{D}\left(\mathcal{O}\right)$. Then we have 
\[
\text{supp}\left(\mathcal{F}^{-1}\overline{\pi\left(x,h\right)\psi}\right)\subset h^{-T}\text{supp}\left(\smash{\widehat{\psi}}\right)\subset\mathcal{O},
\]
which implies $\mathcal{F}^{-1}\overline{\pi\left(x,h\right)\psi}\in\mathcal{D}\left(\mathcal{O}\right)$,
i.e. $\overline{\pi\left(x,h\right)\psi}\in\mathcal{F}\left(\mathcal{D}\left(\mathcal{O}\right)\right)$
for all $\left(x,h\right)\in G$. This shows that $W_{\psi}f$ as
defined in equation \eqref{eq:WaveletTrafoAufSpezialReservoir} is
well-defined.

Furthermore, for $\psi\neq0$ and $\varphi\in\mathcal{S}\left(\mathbb{R}^{d}\right)$
with $\widehat{\varphi}\in\mathcal{D}\left(\mathcal{O}\right)$ we
have 
\begin{equation}
\overline{\pi\left(\alpha\right)\varphi}=\frac{1}{C_{\psi}}\cdot\int_{G}\left\langle \pi\left(\beta\right)\psi,\pi\left(\alpha\right)\varphi\right\rangle _{L^{2}\left(\mathbb{R}^{d}\right)}\cdot\overline{\pi\left(\beta\right)\psi}\,\text{d}\beta,\label{eq:PsiDarstellungAlsTestfunktion}
\end{equation}
for all $\alpha\in G$, where the integral is to be understood in
the weak sense in $\mathcal{F}\left(\mathcal{D}\left(\mathcal{O}\right)\right)$.

Finally, for $f\in\left(\mathcal{F}\left(\mathcal{D}\left(\mathcal{O}\right)\right)\right)'$
and $\psi\neq0$, the wavelet transform $W_{\psi}f$ defined
in equation \eqref{eq:WaveletTrafoAufSpezialReservoir} is a continuous
function that satisfies the reproduction formula 
\begin{equation}
W_{\psi}f=W_{\psi}f\ast\frac{W_{\psi}\psi}{C_{\psi}}.\label{eq:ReproduktionsFormelFuerSpezialWaveletTrafo}
\end{equation}
\end{lem}
\begin{proof}
We first note the general identity $\mathcal{F}^{-1}\overline{f}=\overline{\widehat{f}}$
which is valid for arbitrary $f\in L^{1}\left(\mathbb{R}^{d}\right)$.
This yields 
\begin{eqnarray}
\mathcal{F}^{-1}\overline{\pi\left(x,h\right)\psi} & = & \overline{\mathcal{F}\left(\pi\left(x,h\right)\psi\right)}\nonumber \\
 & \overset{\text{Eq. }\eqref{eq:QausiRegulaereDarstellungAufFourierSeite}}{=} & \left|\det\left(h\right)\right|^{1/2}\cdot\overline{M_{-x}D_{h}\widehat{\psi}}\nonumber \\
 & = & \left|\det\left(h\right)\right|^{1/2}\cdot\left(\overline{\widehat{\psi}}\circ h^{T}\right)\cdot e_{x}\label{eq:KonjugierteInverseFourierTrafo}
\end{eqnarray}
with $e_{x}:\mathbb{R}^{d}\rightarrow\mathbb{C},\xi\mapsto e^{2\pi i\left\langle x,\xi\right\rangle }$.
We conclude 
\[
\text{supp}\left(\mathcal{F}^{-1}\overline{\pi\left(x,h\right)\psi}\right)=\text{supp}\left(\overline{\widehat{\psi}}\circ h^{T}\right)=h^{-T}\left(\text{supp}\left(\smash{\widehat{\psi}}\right)\right)\subset h^{-T}\mathcal{O}=\mathcal{O},
\]
as claimed. As $\left(\overline{\widehat{\psi}}\circ h^{T}\right)\cdot e_{x}$
is smooth, we get $\mathcal{F}^{-1}\overline{\pi\left(x,h\right)\psi}\in\mathcal{D}\left(\mathcal{O}\right)$.

Fix $\alpha=\left(x,h\right)\in G$. For $\beta=\left(y,g\right)\in G$
with $\left\langle \pi\left(\beta\right)\psi,\pi\left(\alpha\right)\varphi\right\rangle _{L^{2}}\neq0$,
the Plancherel theorem yields 
\begin{eqnarray*}
0\neq\left\langle \pi\left(\beta\right)\psi,\pi\left(\alpha\right)\varphi\right\rangle _{L^{2}} & = & \left\langle \widehat{\pi\left(\beta\right)\psi},\widehat{\pi\left(\alpha\right)\varphi}\right\rangle _{L^{2}}\\
 & \overset{\text{Eq. }\eqref{eq:QausiRegulaereDarstellungAufFourierSeite}}{=} & \left|\det\left(gh\right)\right|^{1/2}\cdot\left\langle M_{-y}D_{g}\widehat{\psi},M_{-x}D_{h}\widehat{\varphi}\right\rangle _{L^{2}},
\end{eqnarray*}
which implies 
\[
\emptyset\neq\text{supp}\left(D_{g}\widehat{\psi}\right)\cap\text{supp}\left(D_{h}\widehat{\varphi}\right)=h^{-T}\text{supp}\left(\widehat{\varphi}\right)\cap g^{-T}\text{supp}\left(\smash{\widehat{\psi}}\right).
\]
For $K_{1}:=\text{supp}\left(\widehat{\varphi}\right)$, $K_{2}:=\text{supp}\left(\smash{\widehat{\psi}}\right)\subset\mathcal{O}$
and $L=L\left(K_{1},K_{2}\right)\subset H$ compact as in Lemma \ref{lem:AdmissibilityVorbereitung},
the same lemma yields $g\in hL$ and thus 
\[
\text{supp}\left(\mathcal{F}^{-1}\overline{\pi\left(\beta\right)\psi}\right)\subset g^{-T}\text{supp}\left(\smash{\widehat{\psi}}\right)\subset h^{-T}L^{-T}K_{2}=:K_{3}\subset\mathcal{O}.
\]
Note that $K_{3}=K_{3}\left(h\right)$ depends upon $h$ but that
$\alpha=\left(x,h\right)\in G$ is fixed.

This shows that for $\ell\in\mathbb{N}_{0}$ the map 
\begin{eqnarray*}
\Phi_{\ell}: & G\rightarrow C_{K_{3}}^{\ell}\left(\mathcal{O}\right), & \beta=\left(y,g\right)\mapsto \,\,\,\,\,\,\, \left\langle \pi\left(\beta\right)\psi,\pi\left(\alpha\right)\varphi\right\rangle _{L^{2}}\cdot\mathcal{F}^{-1}\overline{\pi\left(\beta\right)\psi}\\
 &  & \phantom{\beta=\left(y\right)}\!\overset{\text{Eq. }\eqref{eq:KonjugierteInverseFourierTrafo}}{=}\left\langle \pi\left(\beta\right)\psi,\pi\left(\alpha\right)\varphi\right\rangle _{L^{2}}\cdot\left|\det\left(g\right)\right|^{1/2}\cdot e_{y}\cdot\left(\overline{\widehat{\psi}}\circ g^{T}\right)
\end{eqnarray*}
is well-defined with $\Phi\left(y,g\right)=0$ for $g\notin hL$.
The strong continuity of $\pi$ and the Lemmata \ref{lem:ModulationStetigAufCkMitKompaktemTraeger}
and \ref{lem:GlatteVerkettung} (with the ensuing remark) show that
$\Phi_{\ell}$ is actually continuous. This implies that $\Phi_{\ell}$
is measurable and that $\Phi_{\ell}\left(G\right)\subset C_{K_{3}}^{\ell}\left(\mathcal{O}\right)$
is $\sigma$-compact and hence separable.

The continuity of $\Gamma:H\rightarrow\mathcal{D}\left(\mathbb{R}^{d}\right),g\mapsto D_{g}\overline{\widehat{\psi}}$,
which was shown in Lemma \ref{lem:GlatteVerkettung} and the ensuing
remark imply finiteness of the constant 
\[
C_{\rho}:=\max_{\gamma\leq\rho}\max_{g\in hL}\left\Vert \partial^{\gamma}\left(\overline{\widehat{\psi}}\circ g^{T}\right)\right\Vert _{\sup}=\max_{\gamma\leq\rho}\max_{g\in hL}\left\Vert \partial^{\gamma}\left(\Gamma\left(g\right)\right)\right\Vert _{\sup}
\]
for all $\rho\in\mathbb{N}_{0}^{d}$. Using the Leibniz rule, we
arrive at 
\begin{align*}
\left\Vert \partial^{\rho}\left(e_{y}\cdot\left(\overline{\widehat{\psi}}\circ g^{T}\right)\right)\right\Vert _{\sup} & =  \left\Vert \sum_{\gamma\leq\rho}{\rho \choose \gamma}\cdot\partial^{\gamma}e_{y}\cdot\partial^{\rho-\gamma}\left(\overline{\widehat{\psi}}\circ g^{T}\right)\right\Vert _{\sup}\\
 & \leq  \sum_{\gamma\leq\rho}{\rho \choose \gamma}\left\Vert \left(2\pi iy\right)^{\gamma}e_{y}\right\Vert _{\sup}\cdot C_{\rho}\\
 & \leq  C_{\rho}\cdot\sum_{\gamma\leq\rho}{\rho \choose \gamma}\left|2\pi y\right|^{\left|\gamma\right|}\leq C_{\rho}'\cdot\left(1+\left|y\right|\right)^{\left|\rho\right|}
\end{align*}
for all $g\in hL$ and some constant $C_{\rho}'>0$. Define $C:=\max_{g\in hL}\left|\det\left(g\right)\right|^{1/2}$
and $\zeta:=\pi\left(\alpha\right)\varphi$ and note 
\[
\left|\left\langle \pi\left(\beta\right)\psi,\pi\left(\alpha\right)\varphi\right\rangle _{L^{2}}\right|=\left|\left\langle \zeta,\pi\left(\beta\right)\psi\right\rangle _{L^{2}}\right|=\left|\left(W_{\psi}\zeta\right)\left(\beta\right)\right|.
\]
Together with $\Phi_{\ell}\left(y,g\right)=0$ for $g\notin hL$,
this proves the estimate 
\begin{eqnarray*}
\left\Vert \partial^{\rho}\Phi_{\ell}\left(\beta\right)\right\Vert _{\text{sup}} & = & \left|\left(W_{\psi}\zeta\right)\left(\beta\right)\right|\cdot\left|\det\left(g\right)\right|^{1/2}\cdot\left\Vert \partial^{\rho}\left(e_{y}\cdot\left(\overline{\widehat{\psi}}\circ g^{T}\right)\right)\right\Vert _{\sup}\\
 & \leq & C\cdot C_{\rho}'\cdot\left(1+\left|y\right|\right)^{\left|\rho\right|}\cdot\left|\left(W_{\psi}\zeta\right)\left(\beta\right)\right|
\end{eqnarray*}
for all $\rho\in\mathbb{N}_{0}^{d}$ and $\beta=\left(y,g\right)\in G$.
In summary, we showed 
\[
    \left\Vert \Phi_{\ell}\left(y,g\right)\right\Vert _{C_{K_{3}}^{\ell}}\leq C \max_{\left|\rho\right| \leq \ell}C_{\rho}'\cdot\left(1+\left|y\right|\right)^{\ell}\cdot\left|\left(W_{\psi}\zeta\right)\left(\beta\right)\right|.
\]
But equation \eqref{eq:QausiRegulaereDarstellungAufFourierSeite}
yields $\widehat{\zeta}=\left|\det\left(h\right)\right|^{1/2}\cdot M_{-x}D_{h}\widehat{\varphi}$,
showing that $\text{supp}\left(\smash{\widehat{\zeta}}\right)\subset h^{-T}\text{supp}\left(\widehat{\varphi}\right)\subset\mathcal{O}$
is compact, so that Theorem \ref{thm:ZulaessigkeitVonbandbeschraenktenFunktionen}
(with $w_{0}\equiv1$ and $N=\ell$) yields $W_{\psi}\zeta\in L_{\left(y,g\right)\mapsto\left(1+\left|y\right|\right)^{\ell}}^{1}\left(G\right)$.
This shows that $\Phi_{\ell}$ is Bochner integrable, so that the
integral 
\[
\varphi_{\ell}:=\frac{1}{C_{\psi}}\cdot\int_{G}\Phi_{\ell}\left(y,g\right)\,\text{d}\left(y,g\right)\in C_{K_{3}}^{\ell}\left(\mathcal{O}\right)\hookrightarrow\bigcap_{p\in\left[1,\infty\right]}L^{p}\left(\mathbb{R}^{d}\right)
\]
is well-defined.

For $f\in L^{2}\left(\mathbb{R}^{d}\right)$, the map $C_{K_{3}}^{\ell}\left(\mathcal{O}\right)\rightarrow\mathbb{C},g\mapsto\left\langle g,f\right\rangle _{L^{2}}$
is a bounded linear functional. Using the left invariance of the Haar
measure, the weak inversion formula (\ref{eq:WaveletInversionsFormel})
and Plancherel's theorem, we calculate 
\begin{eqnarray*}
\left\langle \varphi_{\ell},f\right\rangle _{L^{2}} & = & \frac{1}{C_{\psi}}\cdot\int_{G}\left\langle \left\langle \pi\left(\beta\right)\psi,\pi\left(\alpha\right)\varphi\right\rangle _{L^{2}}\cdot\mathcal{F}^{-1}\overline{\pi\left(\beta\right)\psi},f\right\rangle _{L^{2}}\,\text{d}\beta\\
 & = & \frac{1}{C_{\psi}}\cdot\int_{G}\left\langle \pi\left(\alpha^{-1}\beta\right)\psi,\varphi\right\rangle _{L^{2}}\cdot\left\langle \overline{\pi\left(\beta\right)\psi},\widehat{f}\right\rangle _{L^{2}}\,\text{d}\beta\\
 & = & \frac{1}{C_{\psi}}\cdot\overline{\int_{G}\left\langle \varphi,\pi\left(\alpha^{-1}\beta\right)\psi\right\rangle _{L^{2}}\cdot\left\langle \pi\left(\alpha^{-1}\beta\right)\psi,\pi\left(\alpha^{-1}\right)\overline{\widehat{f}}\right\rangle _{L^{2}}\,\text{d}\beta}\\
 & \overset{\gamma=\alpha^{-1}\beta}{=} & \frac{1}{C_{\psi}}\cdot\overline{\int_{G}\left(W_{\psi}\varphi\right)\left(\gamma\right)\cdot\left\langle \pi\left(\gamma\right)\psi,\pi\left(\alpha^{-1}\right)\overline{\widehat{f}}\right\rangle _{L^{2}}\,\text{d}\gamma}\\
 & \overset{\text{Eq. }\eqref{eq:WaveletInversionsFormel}}{=} & \overline{\left\langle \varphi,\pi\left(\alpha^{-1}\right)\overline{\widehat{f}}\right\rangle _{L^{2}}}=\left\langle \overline{\widehat{f}},\pi\left(\alpha\right)\varphi\right\rangle _{L^{2}}\\
 & = & \left\langle \overline{\pi\left(\alpha\right)\varphi},\widehat{f}\right\rangle _{L^{2}}=\left\langle \mathcal{F}^{-1}\overline{\pi\left(\alpha\right)\varphi},f\right\rangle _{L^{2}}.
\end{eqnarray*}
As this holds for every $f\in L^{2}\left(\mathbb{R}^{d}\right)$,
we get $\mathcal{F}^{-1}\overline{\pi\left(\alpha\right)\varphi}=\varphi_{\ell}$
almost everywhere and then everywhere, as both sides are continuous
functions. In particular, $\varphi_{\ell}\in\mathcal{D}\left(\mathcal{O}\right)$.

Now choose $f\in\left(\mathcal{F}\left(\mathcal{D}\left(\mathcal{O}\right)\right)\right)'$.
By \cite[Theorem 6.8]{RudinFA} and Hahn-Banach, $f\circ\mathcal{F}|_{C_{K_{3}}^{\ell}\left(\mathcal{O}\right)\cap\mathcal{D}\left(\mathcal{O}\right)}$
admits a continuous extension $\widetilde{f}$ to $C_{K_{3}}^{\ell}\left(\mathcal{O}\right)$
for a suitable $\ell\in\mathbb{N}_{0}$. We thus get
\begin{align*}
f\left(\overline{\pi\left(\alpha\right)\varphi}\right) & =\left(f\circ\mathcal{F}\right)\left(\mathcal{F}^{-1}\overline{\pi\left(\alpha\right)\varphi}\right)\\
 & =\left(f\circ\mathcal{F}\right)\left(\varphi_{\ell}\right)=\widetilde{f}\left(\varphi_{\ell}\right)\\
 & =\frac{1}{C_{\psi}}\int_{G}\widetilde{f}\left(\Phi_{\ell}\left(\beta\right)\right)\,\text{d}\beta\\
 & =\frac{1}{C_{\psi}}\int_{G}\left\langle \pi\left(\beta\right)\psi,\pi\left(\alpha\right)\varphi\right\rangle _{L^{2}}\cdot\left(f\circ\mathcal{F}\right)\left(\mathcal{F}^{-1}\overline{\pi\left(\beta\right)\psi}\right)\,\text{d}\beta\\
 & =\frac{1}{C_{\psi}}\int_{G}\left\langle \pi\left(\beta\right)\psi,\pi\left(\alpha\right)\varphi\right\rangle _{L^{2}}\cdot f\left(\overline{\pi\left(\beta\right)\psi}\right)\,\text{d}\beta,
\end{align*}
which proves equation (\ref{eq:PsiDarstellungAlsTestfunktion}), as
$f\in\left(\mathcal{F}\left(\mathcal{D}\left(\mathcal{O}\right)\right)\right)'$
was arbitrary. Additionally, the choice $\varphi=\psi$ shows
\begin{align*}
\left(W_{\psi}f\right)\left(\alpha\right)\overset{\text{Eq. }\eqref{eq:WaveletTrafoAufSpezialReservoir}}{=}f\left(\overline{\pi\left(\alpha\right)\psi}\right) & =\frac{1}{C_{\psi}}\int_{G}\left\langle \pi\left(\beta\right)\psi,\pi\left(\alpha\right)\psi\right\rangle _{L^{2}}\cdot f\left(\overline{\pi\left(\beta\right)\psi}\right)\,\text{d}\beta\\
 & =\frac{1}{C_{\psi}}\int_{G}\left(W_{\psi}f\right)\left(\beta\right)\cdot\left(W_{\psi}\psi\right)\left(\beta^{-1}\alpha\right)\,\text{d}\beta\\
 & =\left(\left(W_{\psi}f\right)\ast\frac{W_{\psi}\psi}{C_{\psi}}\right)\left(\alpha\right)
\end{align*}
which is nothing else than equation (\ref{eq:ReproduktionsFormelFuerSpezialWaveletTrafo}).

The only thing missing is continuity of $W_{\psi}f$. For this, let
$h_{0}\in H$ be arbitrary and choose a compact neighborhood $K_{4}\subset H$
of $h_{0}$. For $\alpha=\left(x,h\right)\in\mathbb{R}^{d}\times K_{4}$
we then have $\text{supp}\left(\mathcal{F}^{-1}\overline{\pi\left(x,h\right)\psi}\right)\subset h^{-T}K_{2}\subset K_{4}^{-T}K_{2}=:K_{5}$.
For $f\in\left(\mathcal{F}\left(\mathcal{D}\left(\mathcal{O}\right)\right)\right)'$
we can choose as above a continuous extension $\widetilde{f}$ of
$f\circ\mathcal{F}|_{\mathcal{D}\left(\mathcal{O}\right)\cap C_{K_{5}}^{\ell}\left(\mathcal{O}\right)}$
to $C_{K_{5}}^{\ell}\left(\mathcal{O}\right)$ for some $\ell\in\mathbb{N}_{0}$.
We then have 
\begin{eqnarray*}
\left(W_{\psi}f\right)\left(x,h\right) & = & \left(f\circ\mathcal{F}\right)\left(\mathcal{F}^{-1}\overline{\pi\left(x,h\right)\psi}\right)\\
 & \overset{\text{Eq. }\eqref{eq:KonjugierteInverseFourierTrafo}}{=} & \left|\det\left(h\right)\right|^{1/2}\cdot\left(f\circ\mathcal{F}\right)\left(\left(\overline{\widehat{\psi}}\circ h^{T}\right)\cdot e_{x}\right)\\
 & = & \left|\det\left(h\right)\right|^{1/2}\cdot\widetilde{f}\left(\left(\overline{\widehat{\psi}}\circ h^{T}\right)\cdot e_{x}\right)
\end{eqnarray*}
for all $\left(x,h\right)\in\mathbb{R}^{d}\times K_{4}$. Noting that
the Lemmata \ref{lem:ModulationStetigAufCkMitKompaktemTraeger} and
\ref{lem:GlatteVerkettung} (with the ensuing remark) show that the
right-hand side defines a continuous function in $\left(x,h\right)$
completes the proof.
\end{proof}
Using this rather technical result, we can now give the proof of Theorem
\ref{thm:IsomorphismusZuSpezialCoorbitRaum}, i.e. of the identification of $\text{Co}\left( L_{v}^{p,q} \right)$ with $\widetilde{\text{Co}}_{\psi}\left( L_{v}^{p,q} \right)$.
\begin{proof}[Proof of Theorem \ref{thm:IsomorphismusZuSpezialCoorbitRaum}]
Let $w:H\rightarrow\left(0,\infty\right)$ be the control weight
for $L_{v}^{p,q}\left(G\right)$ as defined in Lemma \ref{lem:CoorbitVoraussetzungen};
we will interpret this to be a weight on $G$ by $w\left(x,h\right)=w\left(h\right)$
for $\left(x,h\right)\in G$.

For $f\in\text{Co}\left(L_{v}^{p,q}\right)\subset\left(\mathcal{H}_{w}^{1}\right)^{\neg}$
we have
\begin{equation}
    \left(W_{\psi}\left(\Theta f\right)\right)\left(x,h\right) \overset{\text{Eq. } (\ref{eq:WaveletTrafoAufSpezialReservoir})}{=}\left(\Theta f\right)\left(\overline{\pi\left(x,h\right)\psi}\right)=f\left(\pi\left(x,h\right)\psi\right)=\left(W_{\psi}f\right)\left(x,h\right).\label{eq:WaveletTrafoVertauschtMitTheta}
\end{equation}
This shows that $\Theta$ is well-defined and isometric. But note
that we do not yet know that $\left\Vert \cdot\right\Vert _{\widetilde{\text{Co}}_{\psi}\left(L_{v}^{p,q}\right)}$
is a norm. Nevertheless, $\Theta$ is injective, because it is the
restriction of an injective map (cf. Corollary \ref{cor:FourierTrafoAufRaumAbgewaelzt}).

In the proof of the surjectivity of $\Theta$ below, we will need
the fact that $W_{\psi}$ is injective on $\widetilde{\text{Co}}_{\psi}\left(L_{v}^{p,q}\right)$,
which we now prove. Together with the continuity of $W_{\psi}f$ for
$f\in\left(\mathcal{F}\left(\mathcal{D}\left(\mathcal{O}\right)\right)\right)'$
(cf. Lemma \ref{lem:SpezialWaveletTrafoErfuelltReproduktionsFormel}),
this will also show that $\left\Vert \cdot\right\Vert _{\widetilde{\text{Co}}_{\psi}\left(L_{v}^{p,q}\right)}$
is a norm. Choose $f\in\widetilde{\text{Co}}_{\psi}\left(L_{v}^{p,q}\right)$
with $W_{\psi}f\equiv0$ and let $\varphi\in\mathcal{F}\left(\mathcal{D}\left(\mathcal{O}\right)\right)$
be arbitrary. Note that we have $\widehat{\overline{\varphi}}=\overline{\mathcal{F}^{-1}\varphi}\in\mathcal{D}\left(\mathcal{O}\right)$.
Thus, equation \eqref{eq:PsiDarstellungAlsTestfunktion} (with $\alpha=1_{G}$
and with $\overline{\varphi}$ instead of $\varphi$) shows
\[
f\left(\varphi\right)=f\left(\overline{\pi\left(\alpha\right)\overline{\varphi}}\right)=\frac{1}{C_{\psi}}\cdot\int_{G}\left\langle \pi\left(\beta\right)\psi,\pi\left(\alpha\right)\overline{\varphi}\right\rangle _{L^{2}\left(\mathbb{R}^{d}\right)}\cdot\underbrace{f\left(\overline{\pi\left(\beta\right)\psi}\right)}_{=\left(W_{\psi}f\right)\left(\beta\right)=0}\,\text{d}\beta=0.
\]
As $\varphi\in\mathcal{F}\left(\mathcal{D}\left(\mathcal{O}\right)\right)$
was arbitrary, we conclude $f\equiv0$.

It remains to show that $\Theta$ is surjective.
This will in particular imply that $\widetilde{\text{Co}}_{\psi}(L_{v}^{p,q})$ is independent of the choice of $\psi$, as the same is true of $\text{Co}\left( L_{v}^{p,q} \right)$ and of $\Theta$.
To this end, let $f\in\widetilde{\text{Co}}_{\psi}\left(L_{v}^{p,q}\right)$
be arbitrary.

Let $U\subset G$ be an arbitrary open, precompact neighborhood of
$1_{G}$ and set $E:=\frac{W_{\psi}\psi}{C_{\psi}}$. Note that Theorem
\ref{thm:ZulaessigkeitVonbandbeschraenktenFunktionen} implies $E\in W^{R}\left(L^{\infty},L_{w}^{1}\right)$,
i.e. $K_{U^{-1}}E\in L_{w}^{1}\left(G\right)$, where $K_{U^{-1}}E$
denotes the (right sided) control function of $E$ with respect to
$U^{-1}$ (cf. equation (\ref{eq:RechtsseitigeWienerKontrollFunktion})).

The definition of $\widetilde{\text{Co}}_{\psi}\left(L_{v}^{p,q}\right)$
yields $F:=W_{\psi}f\in L_{v}^{p,q}\left(G\right)$, whereas Lemma
\ref{lem:SpezialWaveletTrafoErfuelltReproduktionsFormel} implies
that $F$ obeys the reproduction formula $F=F\ast E$. Observe that
we have $E\left(z^{-1}\right)=\overline{E\left(z\right)}$ for all
$z\in G$. Using this, we get, for $\alpha\in G$ and $\beta\in U$
the estimate 
\begin{eqnarray*}
\left|F\left(\alpha\beta\right)\right| & = & \left|\left(F\ast E\right)\left(\alpha\beta\right)\right|\\
 & \overset{\left|E\left(z^{-1}\right)\right|=\left|E\left(z\right)\right|}{\leq} & \int_{G}\left|F\left(y\right)\right|\cdot\left|E\left(\beta^{-1}\alpha^{-1}y\right)\right|\,\text{d}y\\
 & \overset{\left(\ast\right)}{\leq} & \int_{G}\left|F\left(y\right)\right|\cdot\left(K_{U^{-1}}E\right)\left(\alpha^{-1}y\right)\,\text{d}y\\
 & = & \left(\left|F\right|\ast\left(K_{U^{-1}}E\right)^{\vee}\right)\left(\alpha\right).
\end{eqnarray*}
In the step marked with $\left(\ast\right)$ we used the fact that
$U^{-1}\alpha^{-1}y\subset G$ is open and that $E$ is continuous,
so that we have 
\[
\left(K_{U^{-1}}E\right)\left(\alpha^{-1}y\right)=\left\Vert \chi_{U^{-1}\alpha^{-1}y}\cdot E\right\Vert _{L^{\infty}\left(G\right)}=\sup_{\gamma\in U^{-1}}\left|E\left(\gamma\alpha^{-1}y\right)\right|.
\]
The above estimate means nothing but
\[\left(K_{U}'F\right)\left(\alpha\right)\leq\left(\left|F\right|\ast\left(K_{U^{-1}}E\right)^{\vee}\right)\left(\alpha\right)\]
where $K_{U}'F$ is the left-sided control function of $F$ with respect
to $U$ (cf. equation \eqref{eq:LinksseitigeWienerKontrollFunktion}).
Note that the continuity of $F$ (cf. Lemma \ref{lem:SpezialWaveletTrafoErfuelltReproduktionsFormel})
implies that $K_{U}'F$ is lower semicontinuous and hence measurable.

As noted above, we have $\left|F\right|\in L_{v}^{p,q}\left(G\right)$
and $K_{U^{-1}}E\in L_{w}^{1}\left(G\right)$. Using Lemma \ref{lem:CoorbitVoraussetzungen},
this implies $\left|F\right|\ast\left(K_{U^{-1}}E\right)^{\vee}\in L_{v}^{p,q}\left(G\right)$.
Thus, the solidity of $L_{v}^{p,q}$ yields $K_{U}'F\in L_{v}^{p,q}\left(G\right)$,
i.e. $F\in W\left(L^{\infty},L_{v}^{p,q}\right)$.

Now \cite[Lemma 3.3]{RauhutWienerAmalgam} shows that we have the
estimate
\[
r\left(x\right):=\vertiii{L_{x^{-1}}}_{W\left(L^{\infty},L_{v}^{p,q}\right)}\leq\vertiii{L_{x^{-1}}}_{L_{v}^{p,q}}\overset{\text{Lemma }\ref{lem:GemischterLebesgueRaumLinksRechtsTranslation}}{\leq}w\left(x^{-1}\right)=w\left(x\right).
\]
Finally, \cite[Lemma 3.2]{RauhutWienerAmalgam} yields $W\left(L^{\infty},L_{v}^{p,q}\right)\hookrightarrow L_{1/r}^{\infty}\left(G\right)\hookrightarrow L_{1/w}^{\infty}\left(G\right)$
and hence $F\in L_{1/w}^{\infty}\left(G\right)$.

By Lemma \ref{lem:SpezialWaveletTrafoErfuelltReproduktionsFormel},
$F$ obeys the reproduction formula $F=F\ast W_{\psi}\psi / C_{\psi}$ so that
\cite[Theorem 4.1(iv)]{FeichtingerCoorbit1} guarantees the existence
of $g\in\left(\mathcal{H}_{w}^{1}\right)^{\neg}$ satisfying $W_{\psi}g=W_{\psi}f\in L_{v}^{p,q}\left(G\right)$.
This immediately entails $g\in\text{Co}\left(L_{v}^{p,q}\left(G\right)\right)$
and thus $\Theta g\in\widetilde{\text{Co}}_{\psi}\left(L_{v}^{p,q}\right)$.

But equation \eqref{eq:WaveletTrafoVertauschtMitTheta} shows $W_{\psi}\left(\smash{\Theta g}\right)=W_{\psi}g=W_{\psi}f$
which implies $f=\Theta g$, because we have seen above that $W_{\psi}$
is injective on $\widetilde{\text{Co}}_{\psi}\left(L_{v}^{p,q}\right)$.
\end{proof}
We now show that the map $\mathcal{F}^{-1}:\mathcal{D}'\left(\mathcal{O}\right)\rightarrow\left(\mathcal{F}\left(\mathcal{D}\left(\mathcal{O}\right)\right)\right)',f\mapsto f\circ\mathcal{F}^{-1}$
restricts to a continuous map $\mathcal{F}^{-1}:\mathcal{D}\left(\mathcal{Q},L^{p},\ell_{u}^{q}\right)\rightarrow \widetilde{\text{Co}}\left(L_{v}^{p,q}\right)$.
\begin{lem}
\label{lem:SpezialWaveletTrafoLiegtInPassendemRaum}Assume that $\mathcal{Q}=\left(h_{i}^{-T}Q\right)_{i\in I}$
is a decomposition covering of $\mathcal{O}$ induced by $H$ and
choose 
\[
u_{i}=\left|\det\left(h_{i}\right)\right|^{\frac{1}{2}-\frac{1}{q}}\cdot v\left(h_{i}\right)\text{ for }i\in I.
\]
Choose $\psi\in\mathcal{S}\left(\mathbb{R}^{d}\right)\setminus\left\{ 0\right\} $
with $\widehat{\psi}\in\mathcal{D}\left(\mathcal{O}\right)$ and let
$p,q\in\left[1,\infty\right]$. Then there is a constant $C>0$ such
that 
\[
\left\Vert W_{\psi}\left(f\circ\mathcal{F}^{-1}\right)\right\Vert _{L_{v}^{p,q}}\leq C\cdot\left\Vert f\right\Vert _{\mathcal{D}\left(\mathcal{Q},L^{p},\ell_{u}^{q}\right)}<\infty
\]
holds for all $f\in\mathcal{D}\left(\mathcal{Q},L^{p},\ell_{u}^{q}\right)$.
Thus, the map
\[
\mathcal{F}^{-1}:\mathcal{D}\left(\mathcal{Q},L^{p},\ell_{u}^{q}\right)\rightarrow\widetilde{\text{Co}}\left(L_{v}^{p,q}\right),f\mapsto f\circ\mathcal{F}^{-1}
\]
is well-defined and bounded.\end{lem}
\begin{rem}
\label{rem:InverseFourierTransformationIstNatuerlich}It is worth
noting that the inverse Fourier transform $\mathcal{F}^{-1}:\mathcal{D}\left(\mathcal{Q},L^{p},\ell_{u}^{q}\right)\rightarrow\widetilde{\text{Co}}\left(L_{v}^{p,q}\right)$
defined as above coincides on $L^{2}\left(\mathbb{R}^{d}\right)\cap\mathcal{D}\left(\mathcal{Q},L^{p},\ell_{u}^{q}\right)$
with the ordinary Fourier transform, where $f\in L^{2}\left(\mathbb{R}^{d}\right)$
is considered as an element of $\mathcal{D}'\left(\mathcal{O}\right)$
by $f\left(\varphi\right):=\int f\cdot\varphi\,{\rm d}x=\left\langle f,\overline{\varphi}\right\rangle _{L^{2}}$
for $\varphi\in\mathcal{D}\left(\mathcal{O}\right)$.

For the proof, %
{} simply note that we have
\[
\left(f\circ\mathcal{F}^{-1}\right)\left(\varphi\right)=f\left(\mathcal{F}^{-1}\varphi\right)=\left\langle f,\overline{\mathcal{F}^{-1}\varphi}\right\rangle _{L^{2}}=\left\langle f,\widehat{\overline{\varphi}}\right\rangle _{L^{2}}=\left\langle \mathcal{F}^{-1}f,\overline{\varphi}\right\rangle _{L^{2}}=\left(\mathcal{F}^{-1}f\right)\left(\varphi\right)
\]
for $f\in L^{2}\left(\mathbb{R}^{d}\right)\cap\mathcal{D}\left(\mathcal{Q},L^{p},\ell_{u}^{q}\right)$
and $\varphi\in\mathcal{F}\left(\mathcal{D}\left(\mathcal{O}\right)\right)\subset L^{2}\left(\mathbb{R}^{d}\right)$.\end{rem}
\begin{proof}
Let $\left(\varphi_{i}\right)_{i\in I}$ be a $\mathcal{Q}$-BAPU.
Define $K:=\text{supp}\left(\smash{\widehat{\psi}}\right)\subset\mathcal{O}$.
For $h\in H$ we define 
\[
I_{h}:=\left\{ i\in I\with h^{-T}K\cap h_{i}^{-T}\overline{Q}\neq\emptyset\right\} .
\]
Note that $\overline{Q}\subset\mathcal{O}$ is compact by definition
of an induced covering. Therefore, Lemma \ref{lem:AdmissibilityVorbereitung}
yields a constant $C_{1}=C_{1}\left(\left(h_{i}\right)_{i\in I},\overline{Q},K\right)>0$
with $\left|I_{h}\right|\leq C$ for all $h\in H$. Note that we have
\begin{equation}
\sum_{i\in I_{h}}\varphi_{i}\left(x\right)=1\qquad\text{for all }x\in h^{-T}K,\label{eq:LokalisiertePartitonOfUnity}
\end{equation}
because for $x\in h^{-T}K$ and $i\in I$ with $\varphi_{i}\left(x\right)\neq0$
we have $x\in\varphi_{i}^{-1}\left(\mathbb{C}^{\ast}\right)\subset h_{i}^{-T}Q$
and thus $x\in h^{-T}K\cap h_{i}^{-T}Q$, i.e. $i\in I_{h}$. This
shows $1=\sum_{i\in I}\varphi_{i}\left(x\right)=\sum_{i\in I_{h}}\varphi_{i}\left(x\right)$.

Let $\left(x,h\right)\in G$ be arbitrary. Because of Lemma \ref{lem:SpezialWaveletTrafoErfuelltReproduktionsFormel}
we know $F_{x,h}:=\mathcal{F}^{-1}\overline{\pi\left(x,h\right)\psi}\in\mathcal{D}\left(\mathcal{O}\right)$
with 
\[
\text{supp}\left(F_{x,h}\right)\subset h^{-T}\text{supp}\left(\smash{\widehat{\psi}}\right)=h^{-T}K.
\]
Together with equation (\ref{eq:LokalisiertePartitonOfUnity}) this
yields 
\[
F_{x,h}=\sum_{i\in I_{h}}\varphi_{i}F_{x,h}.
\]
Let $f\in\mathcal{D}\left(\mathcal{Q},L^{p},\ell_{u}^{q}\right)$.
Then the above identity yields the fundamental \textbf{localization
identity}
\begin{eqnarray}
\left(W_{\psi}\left(f\circ\mathcal{F}^{-1}\right)\right)\left(x,h\right) & \overset{\text{Eq. }\eqref{eq:WaveletTrafoAufSpezialReservoir}}{=} & \left(f\circ\mathcal{F}^{-1}\right)\left(\overline{\pi\left(x,h\right)\psi}\right)\nonumber \\
 & = & f\left(F_{x,h}\right)\nonumber \\
 & = & \sum_{i\in I_{h}}f\left(\varphi_{i}F_{x,h}\right)=\sum_{i\in I_{h}}\left(\varphi_{i}f\right)\left(F_{x,h}\right).\label{eq:LokalisierungsIdentitaet}
\end{eqnarray}
Note that $\varphi_{i}f$ is a distribution with compact support
(and hence a tempered distribution) and that we have $\mathcal{F}^{-1}\left(\varphi_{i}f\right)\in L^{p}\left(\mathbb{R}^{d}\right)$
because of $f\in\mathcal{D}\left(\mathcal{Q},L^{p},\ell_{u}^{q}\right)$.
Using this and the definition of the quasi-regular representation,
we calculate 
\begin{eqnarray*}
\left(\varphi_{i}f\right)\left(F_{x,h}\right) & = & \left(\varphi_{i}f\right)\left(\mathcal{F}^{-1}\overline{\pi\left(x,h\right)\psi}\right)=\left(\mathcal{F}^{-1}\left(\varphi_{i}f\right)\right)\left(\overline{\pi\left(x,h\right)\psi}\right)\\
 & \overset{\text{Eq. }\eqref{eq:QuasiRegulaereDarstellung}}{=} & \left|\det\left(h\right)\right|^{-1/2}\cdot\int_{\mathbb{R}^{d}}\left(\mathcal{F}^{-1}\left(\varphi_{i}f\right)\right)\left(y\right)\cdot\overline{\left(D_{h^{-T}}\psi\right)\left(y-x\right)}\,\text{d}y\\
 & = & \left|\det\left(h\right)\right|^{-1/2}\cdot\int_{\mathbb{R}^{d}}\left(\mathcal{F}^{-1}\left(\varphi_{i}f\right)\right)\left(y\right)\cdot\left(D_{h^{-T}}\psi^{\ast}\right)\left(x-y\right)\,\text{d}y\\
 & = & \left|\det\left(h\right)\right|^{-1/2}\cdot\left(\left(\mathcal{F}^{-1}\left(\varphi_{i}f\right)\right)\ast\left(D_{h^{-T}}\psi^{\ast}\right)\right)\left(x\right)
\end{eqnarray*}
with $\psi^{\ast}\left(y\right)=\overline{\psi\left(-y\right)}$ for
$y\in\mathbb{R}^{d}$.

Using Young's inequality, we derive 
\begin{eqnarray*}
\left\Vert x\mapsto\left(\varphi_{i}f\right)\left(F_{x,h}\right)\right\Vert _{L^{p}\left(\mathbb{R}^{d}\right)} & = & \left|\det\left(h\right)\right|^{-1/2}\cdot\left\Vert \left(\mathcal{F}^{-1}\left(\varphi_{i}f\right)\right)\ast\left(D_{h^{-T}}\psi^{\ast}\right)\right\Vert _{L^{p}\left(\mathbb{R}^{d}\right)}\\
 & \leq & \left|\det\left(h\right)\right|^{-1/2}\cdot\left\Vert D_{h^{-T}}\psi^{\ast}\right\Vert _{L^{1}\left(\mathbb{R}^{d}\right)}\cdot\left\Vert \mathcal{F}^{-1}\left(\varphi_{i}f\right)\right\Vert _{L^{p}\left(\mathbb{R}^{d}\right)}\\
 & = & \left|\det\left(h\right)\right|^{1/2}\cdot\left\Vert \psi^{\ast}\right\Vert _{L^{1}\left(\mathbb{R}^{d}\right)}\cdot\left\Vert \mathcal{F}^{-1}\left(\varphi_{i}f\right)\right\Vert _{L^{p}\left(\mathbb{R}^{d}\right)}.
\end{eqnarray*}
Together with the localization identity (\ref{eq:LokalisierungsIdentitaet})
this shows 
\begin{eqnarray}
\left\Vert \left(W_{\psi}\left(f\circ\mathcal{F}^{-1}\right)\right)\left(\cdot,h\right)\right\Vert _{L^{p}\left(\mathbb{R}^{d}\right)} & \leq & \sum_{i\in I_{h}}\left\Vert x\mapsto\left(\varphi_{i}f\right)\left(F_{x,h}\right)\right\Vert _{L^{p}\left(\mathbb{R}^{d}\right)}\nonumber \\
 & \leq & \left|\det\left(h\right)\right|^{1/2}\cdot\left\Vert \psi^{\ast}\right\Vert _{L^{1}\left(\mathbb{R}^{d}\right)}\cdot\sum_{i\in I_{h}}\left\Vert \mathcal{F}^{-1}\left(\varphi_{i}f\right)\right\Vert _{L^{p}\left(\mathbb{R}^{d}\right)}\label{eq:KontrolleVonWaveletTrafoPunktweiseDurchDekompositionsNorm}
\end{eqnarray}
for all $h\in H$.

By definition of an induced covering, $\left(h_{i}\right)_{i\in I}$
is well-spread in $H$, so that there is a precompact, measurable
set $U\subset H$ with $H=\bigcup_{i\in I}h_{i}U$. Choose $K_{2}:=\overline{U}^{-T}K\cup\overline{Q}$
and note that the family $\mathcal{Q}':=\left(Q_{i}'\right)_{i\in I}:=\left(h_{i}^{-T}K_{2}\right)_{i\in I}$
is an induced covering of $\mathcal{O}$. Note that for $i\in I$,
$h\in h_{i}U$ and $j\in I_{h}$ we have 
\[
\emptyset\neq h^{-T}K\cap h_{j}^{-T}\overline{Q}\subset h_{i}^{-T}U^{-T}K\cap h_{j}^{-T}\overline{Q}\subset Q_{i}'\cap Q_{j}'
\]
and thus $j\in i^{\ast_{\mathcal{Q}'}}$, i.e. 
\begin{equation}
I_{h}\subset i^{\ast_{\mathcal{Q}'}}\text{ for all }i\in I\text{ and }h\in h_{i}U,\label{eq:InverseFourierStetigkeitSchnittEnthaltenInCluster}
\end{equation}
where the cluster $i^{\ast_{\mathcal{Q}'}}$ is taken with respect
to $\mathcal{Q}'$. Theorem \ref{thm:InduzierteUeberdeckungKonstruktion}
shows that $\mathcal{Q}'$ is an admissible covering of $\mathcal{O}$,
so that 
\[
C_{1}:=\sup_{i\in I}\left|i^{\ast_{\mathcal{Q}'}}\right|\in\mathbb{N}
\]
is a finite constant. Furthermore, Lemmata \ref{lem:GewichtsDiskretisierung}
and \ref{lem:GewichtsTransplantation} show that $\left(u_{i}\right)_{i\in I}$
is $\mathcal{Q}'$-moderate, so that we have $u_{i}\leq C_{2}\cdot u_{j}$
for all $i\in I$ and $j\in i^{\ast_{\mathcal{Q}'}}$ for some constant
$C_{2}>0$. Finally, let $C_{3}:=\max_{u\in\overline{U}}\left|\det\left(u\right)\right|^{\frac{1}{2}-\frac{1}{q}}$
and $C_{4}:=\sup_{u\in\overline{U}}v_{0}\left(1_{H}\right)v_{0}\left(u\right)$.
Then we have, for $i\in I$ and $h=h_{i}u\in h_{i}U$: 
\[
\left|\det\left(h\right)\right|^{\frac{1}{2}-\frac{1}{q}}\leq C_{3}\cdot\left|\det\left(h_{i}\right)\right|^{\frac{1}{2}-\frac{1}{q}}\quad\text{ and }\quad v\left(h\right)=v\left(1_{H}h_{i}u\right)\leq v_{0}\left(1_{H}\right)v\left(h_{i}\right)v_{0}\left(u\right)\leq C_{4}\cdot v\left(h_{i}\right).
\]

We first show the claim of the lemma in the case $q=\infty$. To this
end, let $h\in H$ be arbitrary. Then we have $h\in h_{i}U$ for some
$i\in I$. We can then estimate 
\begin{eqnarray*}
    &  & \hspace{-0.4cm} v\left(h\right)\cdot\left\Vert \left(W_{\psi}\left(f\circ\mathcal{F}^{-1}\right)\right)\left(\cdot,h\right)\right\Vert _{L^{p}\left(\mathbb{R}^{d}\right)}\\
    & \overset{\text{Eq. }\eqref{eq:KontrolleVonWaveletTrafoPunktweiseDurchDekompositionsNorm}}{\leq} &\hspace{-0.4cm}  v\left(h\right) \cdot \left|\det\left(h\right)\right|^{1/2}\cdot\left\Vert \psi^{\ast}\right\Vert _{L^{1}\left(\mathbb{R}^{d}\right)}\cdot\sum_{j\in I_{h}}\left\Vert \mathcal{F}^{-1}\left(\varphi_{j}f\right)\right\Vert _{L^{p}\left(\mathbb{R}^{d}\right)}\\
    & \overset{\frac{1}{2}-\frac{1}{q}=\frac{1}{2}\text{ and Eq.}\eqref{eq:InverseFourierStetigkeitSchnittEnthaltenInCluster}}{\leq} & \hspace{-0.4cm} C_{3}C_{4} \cdot \left\Vert \psi^{\ast}\right\Vert _{L^{1}\left(\mathbb{R}^{d}\right)}\cdot\sum_{j\in i^{\ast_{\mathcal{Q}'}}}\left[u_{i}\cdot\left\Vert \mathcal{F}^{-1}\left(\varphi_{j}f\right)\right\Vert _{L^{p}\left(\mathbb{R}^{d}\right)}\right]\\
    & \leq & \hspace{-0.4cm} C_{2}C_{3}C_{4} \cdot \left\Vert \psi^{\ast}\right\Vert _{L^{1}\left(\mathbb{R}^{d}\right)}\cdot\sum_{j\in i^{\ast_{\mathcal{Q}'}}}\left[u_{j}\cdot\left\Vert \mathcal{F}^{-1}\left(\varphi_{j}f\right)\right\Vert _{L^{p}\left(\mathbb{R}^{d}\right)}\right]\\
    & \overset{q=\infty}{\leq} & \hspace{-0.4cm} C_{2}C_{3}C_{4} \cdot \left\Vert \psi^{\ast}\right\Vert _{L^{1}\left(\mathbb{R}^{d}\right)}\cdot\left|i^{\ast_{\mathcal{Q}'}}\right|\cdot\left\Vert f\right\Vert _{\mathcal{D}\left(\mathcal{Q},L^{p},\ell_{u}^{q}\right)}\\
    & \leq & \hspace{-0.4cm} C_{1}C_{2}C_{3}C_{4} \cdot \left\Vert \psi^{\ast}\right\Vert _{L^{1}\left(\mathbb{R}^{d}\right)}\cdot\left\Vert f\right\Vert _{\mathcal{D}\left(\mathcal{Q},L^{p},\ell_{u}^{q}\right)}<\infty,
\end{eqnarray*}
where the constants $C_{1},\dots,C_{4}$ and $\left\Vert \psi^{\ast}\right\Vert _{L^{1}\left(\mathbb{R}^{d}\right)}$
are independent of $f$.

In the case $1\leq q<\infty$ we first note that equation (\ref{eq:KontrolleVonWaveletTrafoPunktweiseDurchDekompositionsNorm})
implies, for $i\in I$ and $h\in h_{i}U$, the estimate 
\begin{eqnarray*}
 &  & \left|\det\left(h\right)\right|^{-1}\cdot\left(v\left(h\right)\cdot\left\Vert \left(W_{\psi}\left(f\circ\mathcal{F}^{-1}\right)\right)\left(\cdot,h\right)\right\Vert _{L^{p}\left(\mathbb{R}^{d}\right)}\right)^{q}\\
 & \overset{\text{Eq. }\eqref{eq:KontrolleVonWaveletTrafoPunktweiseDurchDekompositionsNorm}}{\leq} & \left|\det\left(h\right)\right|^{-1}\cdot\left(v\left(h\right)\cdot\left|\det\left(h\right)\right|^{1/2}\cdot\left\Vert \psi^{\ast}\right\Vert _{L^{1}\left(\mathbb{R}^{d}\right)}\cdot\sum_{j\in I_{h}}\left\Vert \mathcal{F}^{-1}\left(\varphi_{j}f\right)\right\Vert _{L^{p}\left(\mathbb{R}^{d}\right)}\right)^{q}\\
 & \overset{\text{Eq. }\eqref{eq:InverseFourierStetigkeitSchnittEnthaltenInCluster}}{\leq} & \left(\left\Vert \psi^{\ast}\right\Vert _{L^{1}\left(\mathbb{R}^{d}\right)}\cdot\left|\det\left(h\right)\right|^{\frac{1}{2}-\frac{1}{q}}\cdot v\left(h\right)\right)^{q}\cdot\left|i^{\ast_{\mathcal{Q}'}}\right|^{q}\cdot\max\left\{ \left\Vert \mathcal{F}^{-1}\left(\varphi_{j}f\right)\right\Vert _{L^{p}\left(\mathbb{R}^{d}\right)}^{q}\with j\in i^{\ast_{\mathcal{Q}'}}\right\} \\
 & \leq & \left(C_{1}C_{3}C_{4}\left\Vert \psi^{\ast}\right\Vert _{L^{1}\left(\mathbb{R}^{d}\right)}\right)^{q}\cdot\sum_{j\in i^{\ast_{\mathcal{Q}'}}}\left[u_{i}^{q}\cdot\left\Vert \mathcal{F}^{-1}\left(\varphi_{j}f\right)\right\Vert _{L^{p}\left(\mathbb{R}^{d}\right)}^{q}\right]\\
 & \leq & \left(C_{1}C_{2}C_{3}C_{4}\left\Vert \psi^{\ast}\right\Vert _{L^{1}\left(\mathbb{R}^{d}\right)}\right)^{q}\cdot\sum_{j\in i^{\ast_{\mathcal{Q}'}}}\left[u_{j}\cdot\left\Vert \mathcal{F}^{-1}\left(\varphi_{j}f\right)\right\Vert _{L^{p}\left(\mathbb{R}^{d}\right)}\right]^{q}\\
 & =: & C_{5}^{q}\cdot\sum_{j\in i^{\ast_{\mathcal{Q}'}}}\left[u_{j}\cdot\left\Vert \mathcal{F}^{-1}\left(\varphi_{j}f\right)\right\Vert _{L^{p}\left(\mathbb{R}^{d}\right)}\right]^{q}.
\end{eqnarray*}
Because of $H=\bigcup_{i\in I}h_{i}U$, this yields 
\begin{eqnarray*}
\left\Vert W_{\psi}\left(f\circ\mathcal{F}^{-1}\right)\right\Vert _{L_{v}^{p,q}}^{q} & = & \int_{H}\left(v\left(h\right)\cdot\left\Vert \left(W_{\psi}\left(f\circ\mathcal{F}^{-1}\right)\right)\left(\cdot,h\right)\right\Vert _{L^{p}\left(\mathbb{R}^{d}\right)}\right)^{q}\frac{\text{d}h}{\left|\det\left(h\right)\right|}\\
 & \leq & \sum_{i\in I}\int_{h_{i}U}\left(v\left(h\right)\cdot\left\Vert \left(W_{\psi}\left(f\circ\mathcal{F}^{-1}\right)\right)\left(\cdot,h\right)\right\Vert _{L^{p}\left(\mathbb{R}^{d}\right)}\right)^{q}\frac{\text{d}h}{\left|\det\left(h\right)\right|}\\
 & \leq & \sum_{i\in I}\left(\mu_{H}\left(h_{i}U\right)C_{5}^{q}\cdot\sum_{j\in i^{\ast_{\mathcal{Q}'}}}\left[u_{j}\cdot\left\Vert \mathcal{F}^{-1}\left(\varphi_{j}f\right)\right\Vert _{L^{p}\left(\mathbb{R}^{d}\right)}\right]^{q}\right)\\
 & \overset{\left(\ast\right)}{\leq} & \mu_{H}\left(\overline{U}\right)\cdot C_{5}^{q}\cdot\sum_{j\in I}\sum_{i\in j^{\ast_{\mathcal{Q}'}}}\left[u_{j}\cdot\left\Vert \mathcal{F}^{-1}\left(\varphi_{j}f\right)\right\Vert _{L^{p}\left(\mathbb{R}^{d}\right)}\right]^{q}\\
 & \overset{\left|j^{\ast_{\mathcal{Q}'}}\right|\leq C_{1}}{\leq} & \mu_{H}\left(\overline{U}\right)\cdot C_{5}^{q}C_{1}\cdot\sum_{j\in I}\left[u_{j}\cdot\left\Vert \mathcal{F}^{-1}\left(\varphi_{j}f\right)\right\Vert _{L^{p}\left(\mathbb{R}^{d}\right)}\right]^{q}\\
 & = & \mu_{H}\left(\overline{U}\right)\cdot C_{5}^{q}C_{1}\cdot\left\Vert f\right\Vert _{\mathcal{D}\left(\mathcal{Q},L^{p},\ell_{u}^{q}\right)}^{q}<\infty
\end{eqnarray*}
which proves the claim in the case $1\leq q<\infty$. In the step
marked with $\left(\ast\right)$, we used the equivalence 
\[
j\in i^{\ast_{\mathcal{Q}'}}\quad\Leftrightarrow\quad Q_{i}'\cap Q_{j}'\neq\emptyset\quad\Leftrightarrow\quad i\in j^{\ast_{\mathcal{Q}'}}
\]
which is valid for all $i,j\in I$.
\end{proof}
It is now easy to show that the map $\Theta^{-1}\circ\mathcal{F}^{-1}:\mathcal{D}\left(\mathcal{Q},L^{p},\ell_{u}^{q}\right)\rightarrow\text{Co}\left(L_{v}^{p,q}\right)$
(with $\Theta$ as in Theorem \ref{thm:IsomorphismusZuSpezialCoorbitRaum}
and $\mathcal{F}^{-1}$ as in Lemma \ref{lem:SpezialWaveletTrafoLiegtInPassendemRaum}
above) is a bounded inverse to the Fourier transform $\mathcal{F}:\text{Co}\left(L_{v}^{p,q}\right)\rightarrow\mathcal{D}\left(\mathcal{Q},L^{p},\ell_{u}^{q}\right)$
as defined in Theorem \ref{thm:FourierStetigVonCoorbitNachDecomposition}.
\begin{thm}
\label{thm:InverseFourierTransformationStetigkeit}Let $p,q\in\left[1,\infty\right]$
and assume that $\mathcal{Q}$ is a decomposition covering of $\mathcal{O}$
induced by $H$. Finally, let $u:\mathcal{O}\rightarrow\left(0,\infty\right)$
be a transplant of $v'$ onto $\mathcal{O}$, where 
\[
v':H\rightarrow\left(0,\infty\right),h\mapsto\left|\det\left(h^{-1}\right)\right|^{\frac{1}{2}-\frac{1}{q}}\cdot v\left(h^{-1}\right)
\]
is defined as in Lemma \ref{lem:FouriertransformationStetigAufCoorbitGeschnittenSchwartz}.

Then $\mathcal{F}:\text{Co}\left(L_{v}^{p,q}\right)\rightarrow\mathcal{D}\left(\mathcal{Q},L^{p},\ell_{u}^{q}\right)$
as defined in Theorem \ref{thm:FourierStetigVonCoorbitNachDecomposition}
is an isomorphism of Banach spaces with bounded inverse
\[
\Theta^{-1}\circ\mathcal{F}^{-1}:\mathcal{D}\left(\mathcal{Q},L^{p},\ell_{u}^{q}\right)\rightarrow\text{Co}\left(L_{v}^{p,q}\right).
\]

\end{thm}

\begin{proof}
Let $f\in\text{Co}\left(L_{v}^{p,q}\right)$ and define $g:=\left(\Theta^{-1}\circ\mathcal{F}^{-1}\right)\left(\mathcal{F}f\right)$.
We will show $\Theta f=\Theta g$. The injectivity of $\Theta$ (cf.
Theorem \ref{thm:IsomorphismusZuSpezialCoorbitRaum}) then implies
$f=g$, i.e. $\left(\Theta^{-1}\circ\mathcal{F}^{-1}\right)\circ\mathcal{F}=\text{id}_{\text{Co}\left(L_{v}^{p,q}\right)}$,
which in particular entails the surjectivity of $\Theta^{-1}\circ\mathcal{F}^{-1}$.

In order to show $\Theta f=\Theta g$, choose an arbitrary $\psi\in\mathcal{S}\left(\mathbb{R}^{d}\right)\setminus\left\{ 0\right\} $
with $\widehat{\psi}\in\mathcal{D}\left(\mathcal{O}\right)$. We have
$\Theta g=\mathcal{F}^{-1}\left(\mathcal{F}f\right)$. Note that we
cannot simpliy ``cancel'' $\mathcal{F}^{-1}$ and $\mathcal{F}$,
as $\mathcal{F}f$ is defined by equation (\ref{eq:FouriertrafoAufFeichtingerReservoir})
and $\mathcal{F}^{-1}\left(\mathcal{F}f\right)$ is defined as in
Lemma \ref{lem:SpezialWaveletTrafoLiegtInPassendemRaum}.

Using these definitions, we derive
\begin{eqnarray*}
\left(W_{\psi}\left(\Theta g\right)\right)\left(x,h\right) & \overset{\text{Eq. }\eqref{eq:WaveletTrafoAufSpezialReservoir}}{=} & \left(\Theta g\right)\left(\overline{\pi\left(x,h\right)\psi}\right)\\
 & = & \left(\mathcal{F}^{-1}\left(\mathcal{F}f\right)\right)\left(\overline{\pi\left(x,h\right)\psi}\right)\\
 & \overset{\text{Def. of }\mathcal{F}^{-1}\text{ in Lemma }\ref{lem:SpezialWaveletTrafoLiegtInPassendemRaum}}{=} & \left(\mathcal{F}f\right)\left(\mathcal{F}^{-1}\overline{\pi\left(x,h\right)\psi}\right)\\
 & \overset{\text{Eq. }\eqref{eq:FouriertrafoAufFeichtingerReservoir}}{=} & f\left(\mathcal{F}^{-1}\overline{\mathcal{F}^{-1}\overline{\pi\left(x,h\right)\psi}}\right)\\
 & = & f\left(\mathcal{F}^{-1}\mathcal{F}\overline{\overline{\pi\left(x,h\right)\psi}}\right)\\
 & = & f\left(\overline{\overline{\pi\left(x,h\right)\psi}}\right)\\
 & = & \left(\Theta f\right)\left(\overline{\pi\left(x,h\right)\psi}\right)\\
 & \overset{\text{Eq. }\eqref{eq:WaveletTrafoAufSpezialReservoir}}{=} & \left(W_{\psi}\left(\Theta f\right)\right)\left(x,h\right),
\end{eqnarray*}
where we used the easily verified identity $\mathcal{F}^{-1}\overline{\varphi}=\overline{\widehat{\varphi}}$
for $\varphi\in\mathcal{S}\left(\mathbb{R}^{d}\right)$. As $W_{\psi}$
is injective on $\widetilde{\text{Co}}\left(L_{v}^{p,q}\right)$ (cf.
Theorem \ref{thm:IsomorphismusZuSpezialCoorbitRaum}), the above identity
shows $\Theta g=\Theta f$.

In the opposite direction, we note that surjectivity of $\mathcal{F}^{-1}:\mathcal{F}\left(\mathcal{D}\left(\mathcal{O}\right)\right)\rightarrow\mathcal{D}\left(\mathcal{O}\right)$
implies that 
\[
\mathcal{F}^{-1}:\mathcal{D}\left(\mathcal{Q},L^{p},\ell_{u}^{q}\right)\leq\mathcal{D}'\left(\mathcal{O}\right)\rightarrow\widetilde{\text{Co}}\left(L_{v}^{p,q}\right)\leq\left(\mathcal{F}\left(\mathcal{D}\left(\mathcal{O}\right)\right)\right)',f\mapsto f\circ\mathcal{F}^{-1}
\]
as defined in Lemma \ref{lem:SpezialWaveletTrafoLiegtInPassendemRaum} is injective. As $\Theta$ and hence also $\Theta^{-1}:\widetilde{\text{Co}}\left(L_{v}^{p,q}\right)\rightarrow\text{Co}\left(L_{v}^{p,q}\right)$
are bijective by Theorem \ref{thm:IsomorphismusZuSpezialCoorbitRaum},
this shows that $\Theta^{-1}\circ\mathcal{F}^{-1}:\mathcal{D}\left(\mathcal{Q},L^{p},\ell_{u}^{q}\right)\rightarrow\text{Co}\left(L_{v}^{p,q}\right)$
is injective. Above, we have already seen that this map is surjective
with
\[
\left(\Theta^{-1}\circ\mathcal{F}^{-1}\right)\circ\mathcal{F}=\text{id}_{\text{Co}\left(L_{v}^{p,q}\right)}.
\]
This shows that $\Theta^{-1}\circ\mathcal{F}^{-1}$ is bijective with
bijective(!) inverse $\mathcal{F}:\text{Co}\left(L_{v}^{p,q}\right)\rightarrow\mathcal{D}\left(\mathcal{Q},L^{p},\ell_{u}^{q}\right)$.
\end{proof}

\section{A sample application: Dilation invariance of certain coorbit spaces}

\label{sec:KonjugationsInvarianz}In this section we discuss the issue
of invariance of coorbit spaces under dilations by matrices that are
not necessarily contained in the group. Here, we will restrict ourselves
to two examples showing that this question can be fairly subtle, with
the answer depending on the dilation group. As was pointed out in
the introduction, the question of comparing coorbit spaces associated
to different dilation groups arises rather naturally in this context,
and the decomposition space view will allow (at least in one case)
a rather speedy answer to it.

We start by spelling out how the wavelet transform of a function $f$ dilated
by some matrix $g\in{\rm GL}(\mathbb{R}^{d})$ over the semidirect
product $\mathbb{R}^{d}\rtimes H$, with $g$ not necessarily contained
in $H$, is related to a wavelet transform of $f$ over the group
$\mathbb{R}^{d}\rtimes g^{-1}Hg$. For this purpose, it is beneficial
to introduce the quasi-regular representation $\sigma$ of the full
affine group $\mathbb{R}^{d}\rtimes{\rm GL}(\mathbb{R}^{d})$ acting
unitarily on ${\rm L}^{2}(\mathbb{R}^{d})$ by
\[
\sigma\left(x,g\right)f=\left|{\rm det}\left(g\right)\right|^{-1/2}\cdot f\left(g^{-1}\left(y-x\right)\right),\qquad(x,g)\in\mathbb{R}^{d}\times{\rm GL}\left(\mathbb{R}^{d}\right),
\]
thus extending the quasi-regular representations of both $\mathbb{R}^{d}\rtimes H$
and $\mathbb{R}^{d}\rtimes g^{-1}Hg$. The proof of the lemma consists
in straightforward computations, and is therefore omitted.
\begin{lem}
\label{lem:dilate_vs_conj}Let $H_{1}$ denote a closed matrix group
fulfilling our standing admissibility assumptions, let $g\in{\rm GL}(\mathbb{R}^{d})$
be arbitrary, and define $H_{2}=g^{-1}H_{1}g$. 
\begin{enumerate}
    \item[(a)] Let $\mathcal{O}_{1}=H_{1}^{T}\xi_{0}$ denote the open dual
orbit associated to $H_{1}$, then the open dual orbit associated
to $H_{2}$ is given by 
\[
\mathcal{O}_{2}=g^{T}\mathcal{O}_{1}=H_{2}^{T}g^{T}\xi_{0}~.
\]

\item[(b)] Assume that $\psi_{1}\in\mathcal{S}\left(\mathbb{R}^{d}\right)$ satisfies
$\widehat{\psi}_{1}\in\mathcal{D}(\mathcal{O}_{1})$. Then $\psi_{2}=\sigma(0,g^{-1})\psi_{1}\in\mathcal{S}\left(\mathbb{R}^{d}\right)$
fulfills $\widehat{\psi}_{2}\in\mathcal{D}(\mathcal{O}_{2})$. 
\item[(c)] Let $H_{i}$ and $\psi_{i}$ be as in the previous parts, and let
$f\in L^{2}(\mathbb{R}^{d})$. Denote by $W_{\psi_{i}}^{i}$ the associated
wavelet transforms. Then we have the relation 
\begin{equation}
\left(W_{\psi_{1}}^{1}\left(\sigma(0,g)f\right)\right)(x,h)=\left(W_{\psi_{2}}^{2}f\right)(g^{-1}x,g^{-1}hg)~.\label{eqn:dilate_vs_conj}
\end{equation}

\item[(d)] Let $v_{1}$ denote a moderate weight function on $H_{1}$, and let
\[
v_{2}:H_{2}\to(0,\infty)~,~h\mapsto v_{1}(ghg^{-1})~.
\]
Then we have, for $f\in{\rm L}^{2}(\mathbb{R}^{d})$, that 
\[
\sigma(0,g)f\in{\rm Co}\left(L_{v_{1}}^{p,q}(\mathbb{R}^{d}\rtimes H_{1})\right)\Longleftrightarrow f\in{\rm Co}\left(L_{v_{2}}^{p,q}(\mathbb{R}^{d}\rtimes H_{2})\right)~.
\]
In particular, ${\rm Co}\left(L_{v_{1}}^{p,q}(\mathbb{R}^{d}\rtimes H_{1})\right)\cap{\rm L}^{2}(\mathbb{R}^{d})$
is invariant under $\sigma(0,g)$ iff 
\[
{\rm Co}\left(L_{v_{1}}^{p,q}(\mathbb{R}^{d}\rtimes H_{1})\right)\cap{\rm L}^{2}(\mathbb{R}^{d})\subset{\rm Co}\left(L_{v_{2}}^{p,q}(\mathbb{R}^{d}\rtimes H_{2})\right)
\]
holds. 
\end{enumerate}
\end{lem}
Note that in general, the question of embeddings between coorbit spaces with respect to different dilation groups
is not even well-posed; this is one reason why the statement in part
(d) is restricted to $L^{2}$-functions. By definition, 
\[
\text{Co}\left(Y\right)=\left\{ f\in\left(\mathcal{H}_{w}^{1}\left(G\right)\right)^{\neg}\with W_{\psi}f\in Y\right\} ,
\]
i.e. the elements $f\in\text{Co}\left(Y\right)$ ``live'' in the
space $\left(\mathcal{H}_{w}^{1}\right)^{\neg}$ of antilinear functionals
on 
\[
\mathcal{H}_{w}^{1}\left(G\right)=\left\{ g\in L^{2}\left(G\right)\with W_{\psi}g\in L_{w}^{1}\left(G\right)\right\} ,
\]
where $w:G\rightarrow\left(0,\infty\right)$ is a suitable \textbf{control-weight}
for the solid BF-space $Y$ and where $\psi\in\mathcal{A}_{w}\setminus\left\{ 0\right\} $
is fixed. Clearly, this definition depends on $G$, and thus on $H$.

Thus, it is not obvious how an element $f\in\text{Co}\left(Y\right)\subset\left(\mathcal{H}_{w}^{1}\left(G\right)\right)^{\neg}$
for some group $G$ can be interpreted as an element $f\in\text{Co}\left(Y'\right)\subset\left(\mathcal{H}_{w'}^{1}\left(G'\right)\right)^{\neg}$
for a different group $G'$ (and different $Y',w'$). But in the setting
of this paper, both groups are of the form $G=\mathbb{R}^{d}\rtimes H$
and $G'=\mathbb{R}^{d}\rtimes H'$ and operate on $L^{2}\left(\mathbb{R}^{d}\right)$
by virtue of the quasi-regular representation. Thus, we adopt the
following conventions: 
\begin{defn}
\label{def:CoorbitEinbettung}Let $p_{1},p_{2},q_{1},q_{2}\in\left[1,\infty\right]$. 
\begin{enumerate}
\item Let $H_{1},H_{2}\leq\text{GL}\left(\mathbb{R}^{d}\right)$ be closed
subgroups that fulfill our standing assumptions and assume also that
$v_{i}:H_{i}\rightarrow\left(0,\infty\right)$ obeys our standing
assumptions for $i=1,2$. Let $G_{i}=\mathbb{R}^{d}\rtimes H_{i}$
for $i=1,2$. We then say that a bounded linear map 
\[
T:\text{Co}\left(L_{v_{1}}^{p_{1},q_{1}}\left(G_{1}\right)\right)\rightarrow\text{Co}\left(L_{v_{2}}^{p_{2},q_{2}}\left(G_{2}\right)\right)
\]
is an \textbf{embedding of coorbit spaces} if $Tf=f$ holds for all
$f\in L^{2}\left(\mathbb{R}^{d}\right)\cap\text{Co}\left(L_{v_{1}}^{p_{1},q_{1}}\left(G_{1}\right)\right)$. 
\item Let $\emptyset\neq U_{1},U_{2}\subset\mathbb{R}^{d}$ be open and
let $\mathcal{Q}^{\left(i\right)}$ be a decomposition covering of
$U_{i}$ for $i=1,2$. Finally, assume that $u_{i}:U_{i}\rightarrow\left(0,\infty\right)$
is $\mathcal{Q}^{\left(i\right)}$-moderate for $i=1,2$. We then
say that a bounded linear map 
\[
S:\mathcal{D}\left(\mathcal{Q},L^{p_{1}},\ell_{u_{1}}^{q_{1}}\right)\rightarrow\mathcal{D}\left(\mathcal{Q},L^{p_{2}},\ell_{u_{2}}^{q_{2}}\right)
\]
is an \textbf{embedding of decomposition spaces} if the identity $Sf=f$
holds for all $f\in L^{2}\left(\mathbb{R}^{d}\right)\cap\mathcal{D}\left(\mathcal{Q},L^{p_{1}},\ell_{u_{1}}^{q_{1}}\right)$.
\end{enumerate}
\end{defn}
\begin{rem*}
It is worth noting that an embedding in the above sense is not required
to be an injective map.
\end{rem*}
The existence of an embedding of coorbit spaces can be characterized
by the existence of embeddings between the associated decomposition
spaces as follows: 
\begin{lem}
\label{lem:CoorbitEinbettungCharakterisierung}Let $H_{1},H_{2}$,
$G_{1},G_{2}$ and $v_{1},v_{2}$ be as in Definition \ref{def:CoorbitEinbettung}
and let $\mathcal{Q}^{\left(j\right)}=\left(\mathcal{Q}_{i}^{\left(j\right)}\right)_{i\in I^{\left(j\right)}}$
be an induced decomposition covering of the dual orbit $\mathcal{O}_{j}\subset\mathbb{R}^{d}$
(with respect to $H_{j}$) for $j=1,2$.

Finally, choose $p_{1},p_{2},q_{1},q_{2}\in\left[1,\infty\right]$,
define 
\[
v_{j}':H_{j}\rightarrow\left(0,\infty\right),h\mapsto\left|\det\left(h^{-1}\right)\right|^{\frac{1}{2}-\frac{1}{q_{j}}}\cdot v_{j}\left(h^{-1}\right)
\]
and let $u^{\left(j\right)}=\left(u_{i}^{\left(j\right)}\right)_{i\in I^{\left(j\right)}}$be
a transplant of $v_{j}'$ to $\mathcal{O}_{j}$ for $j=1,2$.

Then $T:\text{Co}\left(L_{v_{1}}^{p_{1},q_{1}}\left(G_{1}\right)\right)\rightarrow\text{Co}\left(L_{v_{2}}^{p_{2},q_{2}}\left(G_{2}\right)\right)$
is an embedding of coorbit spaces if and only if 
\[
\mathcal{F}\circ T\circ\mathcal{F}^{-1}:\mathcal{D}\left(\mathcal{Q}^{\left(1\right)},L^{p_{1}},\ell_{u^{\left(1\right)}}^{q_{1}}\right)\rightarrow\mathcal{D}\left(\mathcal{Q}^{\left(2\right)},L^{p_{2}},\ell_{u^{\left(2\right)}}^{q_{2}}\right)
\]
is an embedding of decomposition spaces. Here, $\mathcal{F}$ and
$\mathcal{F}^{-1}$ are defined as in Theorems \ref{thm:FourierStetigVonCoorbitNachDecomposition}
and \ref{thm:InverseFourierTransformationStetigkeit}, respectively.\end{lem}
\begin{proof}
Theorem \ref{thm:InverseFourierTransformationStetigkeit} implies
that $T$ is bounded if and only if $\mathcal{F}\circ T\circ\mathcal{F}^{-1}$
is bounded.

Now let $T$ be an embedding of coorbit spaces and let $f\in\mathcal{D}\left(\mathcal{Q}^{\left(1\right)},L^{p_{1}},\ell_{u^{\left(1\right)}}^{q_{1}}\right)\cap L^{2}\left(\mathbb{R}^{d}\right)$.
The Remarks \ref{rem:SpezialFourierTrafoIsFortsetzungVonNormaler}
and \ref{rem:InverseFourierTransformationIstNatuerlich} show that
$\mathcal{F}^{-1}f$ is the ``ordinary'' inverse Fourier transform
of $f\in L^{2}\left(\mathbb{R}^{d}\right)$ and that $\mathcal{F}\mathcal{F}^{-1}f=f$
holds, because the Fourier transform $\mathcal{F}$ as defined in
Theorem \ref{thm:FourierStetigVonCoorbitNachDecomposition} also coincides
with the standard Fourier transform of $\mathcal{F}^{-1}f\in L^{2}\left(\mathbb{R}^{d}\right)$.
Hence, we get 
\begin{eqnarray*}
\left(\mathcal{F}\circ T\circ\mathcal{F}^{-1}\right)\left(f\right) & \overset{\mathcal{F}^{-1}f\in L^{2}\left(\mathbb{R}^{d}\right)\cap\text{Co}\left(L_{v_{1}}^{p_{1},q_{1}}\left(G_{1}\right)\right)}{=} & \mathcal{F}\mathcal{F}^{-1}f\\
 & = & f,
\end{eqnarray*}
i.e. $\mathcal{F}\circ T\circ\mathcal{F}^{-1}$ is an embedding of
decomposition spaces.

The proof of the converse direction is completely analogous.
\end{proof}
We will now analyze the existence of embeddings between the coorbit
space $\text{Co}\left(L_{v}^{p,q}\left(\mathbb{R}^{2}\rtimes H\right)\right)$
and the coorbit space $\text{Co}\left(L_{g^{-1}vg}^{p,q}\left(\mathbb{R}^{2}\times g^{-1}Hg\right)\right)$
with respect to the conjugated group $g^{-1}Hg$, where $g^{-1}vg$
is defined by 
\[
g^{-1}vg:g^{-1}Hg\rightarrow\left(0,\infty\right),h\mapsto v\left(ghg^{-1}\right).
\]
We will see that both spaces coincide (up to harmless identifications)
for the similitude group, whereas the same is in general not true
for the shearlet group.

Our first example is the \textbf{similitude group} 
\[
H_{1}:=\left\{ \begin{pmatrix}a & b\\
-b & a
\end{pmatrix}\with a^{2}+b^{2}\neq0\right\} =\left\{ r\cdot\begin{pmatrix}\cos\varphi & \sin\varphi\\
-\sin\varphi & \cos\varphi
\end{pmatrix}\with r>0\text{ and }\varphi\in\left[0,2\pi\right]\right\} .
\]
Here, the dual orbit is given by $\mathcal{O}_{1}=\mathbb{R}^{2}\setminus\left\{ 0\right\} $.
In the following, we will be using the well-spread family 
\[
\left(h_{k}\right)_{k\in\mathbb{Z}}:=\left(2^{-k}\cdot\text{id}_{\mathbb{R}^{2}}\right)_{k\in\mathbb{Z}}
\]
and the precompact, open sets 
\begin{align*}
P & :=\left\{ x\in\mathbb{R}^{2}\with\frac{2}{3}<\left|x\right|<\frac{3}{2}\right\} ,\\
Q & :=\left\{ x\in\mathbb{R}^{2}\with\frac{1}{2}<\left|x\right|<2\right\} 
\end{align*}
that satisfy $\overline{P}\subset Q\subset\overline{Q}\subset\mathcal{O}_{1}$
as well as $\mathcal{O}_{1}=\bigcup_{k\in\mathbb{Z}}h_{k}P=\bigcup_{k\in\mathbb{Z}}h_{k}Q$.
Thus, Theorem \ref{thm:InduzierteUeberdeckungKonstruktion} shows
that 
\[
\mathcal{Q}:=\left(Q_{k}\right)_{k\in\mathbb{Z}}:=\left(h_{k}^{-T}Q\right)_{k\in\mathbb{Z}}=\left(2^{k}\cdot Q\right)_{k\in\mathbb{Z}}
\]
is a decomposition covering of $\mathcal{O}_{1}$ induced by $H_{1}$.

Now let $g\in\text{GL}\left(\mathbb{R}^{d}\right)$ be arbitrary.
The conjugate group $H_{1}\left(g\right):=g^{-1}H_{1}g\leq\text{GL}\left(\mathbb{R}^{d}\right)$
then has the same open dual orbit $\mathcal{O}_{1}\left(g\right)=\mathbb{R}^{2}\setminus\left\{ 0\right\} =\mathcal{O}_{1}$
and the family $\left(g^{-1}h_{k}g\right)_{k\in\mathbb{Z}}=\left(h_{k}\right)_{k\in\mathbb{Z}}$
is well-spread in $H_{1}\left(g\right)$. Thus, $\mathcal{Q}$ is
also a decomposition covering of $\mathcal{O}_{1}\left(g\right)=\mathcal{O}_{1}$
induced by $g^{-1}H_{1}g$.

For a weight $v:H_{1}\rightarrow\left(0,\infty\right)$ that is $v_{0}$-moderate
for some locally bounded, submultiplicative, measurable weight $v_{0}:H_{1}\rightarrow\left(0,\infty\right)$
we can then define $g^{-1}vg:=v\circ\Phi_{g}$ and $g^{-1}v_{0}g:=v_{0}\circ\Phi_{g}$
for the isomorphism $\Phi_{g}:H_{1}\left(g\right)\rightarrow H_{1},h\mapsto ghg^{-1}$.
Then $g^{-1}vg$ is $g^{-1}v_{0}g$-moderate and we get 
\[
\mathcal{D}\left(\mathcal{Q},L^{p},\ell_{u}^{q}\right)=\mathcal{D}\left(\mathcal{Q},L^{p},\ell_{g^{-1}ug}^{q}\right)
\]
for all $p,q\in\left[1,\infty\right]$, where we have chosen the discretizations
\[
u_{k}=\left|\det\left(h_{k}\right)\right|^{\frac{1}{2}-\frac{1}{q}}\cdot v\left(h_{k}\right)\text{ for }k\in\mathbb{Z}
\]
and 
\[
\left(g^{-1}ug\right)_{k}=\left|\det\left(h_{k}\right)\right|^{\frac{1}{2}-\frac{1}{q}}\cdot v\left(gh_{k}g^{-1}\right)=u_{k}\text{ for }k\in\mathbb{Z}
\]
of the transplant of 
\[
v':H_{1}\rightarrow\left(0,\infty\right),h\mapsto\left|\det\left(h^{-1}\right)\right|^{\frac{1}{2}-\frac{1}{q}}\cdot v\left(h^{-1}\right)
\]
or of 
\[
\left(g^{-1}vg\right)':g^{-1}H_{1}g\rightarrow\left(0,\infty\right),h\mapsto\left|\det\left(h^{-1}\right)\right|^{\frac{1}{2}-\frac{1}{q}}\cdot\left(g^{-1}vg\right)\left(h^{-1}\right)
\]
onto $\mathcal{O}_{1}$ or onto $\mathcal{O}_{1}\left(g\right)$,
respectively (see remark \ref{rem:SpezielleDiskretisierung} or Lemma
\ref{lem:FouriertransformationStetigAufCoorbitGeschnittenSchwartz}
for the validity of this choice).

The identity map $\text{id}:\mathcal{D}\left(\mathcal{Q},L^{p},\ell_{u}^{q}\right)\rightarrow\mathcal{D}\left(\mathcal{Q},L^{p},\ell_{g^{-1}ug}^{q}\right)$
is thus an embedding of decomposition spaces, so that Lemma \ref{lem:CoorbitEinbettungCharakterisierung}
shows that 
\[
\mathcal{F}^{-1}\circ\mathcal{F}:\text{Co}\left(L_{v}^{p,q}\left(\mathbb{R}^{2}\rtimes H_{1}\right)\right)\rightarrow\text{Co}\left(L_{g^{-1}vg}^{p,q}\left(\mathbb{R}^{2}\rtimes g^{-1}H_{1}g\right)\right)
\]
is an embedding of coorbit spaces. The same holds of course for the
inverse map.

It should be noted that the above embedding reduces to the identity
as long as $\left(\mathcal{H}_{w}^{1}\right)^{\neg}$ is identified
with a subspace of $\mathcal{D}'\left(\mathbb{R}^{2}\setminus\left\{ 0\right\} \right)$
(cf. Corollary \ref{cor:FourierTrafoAufFeichtingerReservoir}). It
is furthermore worth noting that the same argument could be applied
to any admissible group $H$ and any $g\in\text{GL}\left(\mathbb{R}^{d}\right)$
as long as $H$ admits a well-spread family $\left(h_{i}\right)_{i\in I}$
that commutes with $g$ (i.e. $g^{-1}h_{i}g=h_{i}$ holds for all
$i\in I$) and as long as the dual orbits of $g^{-1}Hg$ and $H$
coincide.

As a corollary to these observations, we obtain that the homogeneous
Besov spaces $\dot{B}_{s}^{p,q}(\mathbb{R}^{2})$ are invariant
under arbitrary dilations. We expect that this result is well-known,
although we were not able to locate a convenient source for it. Note
that a proof of this fact using the standard tensor wavelet ONB's
promises to be fairly cumbersome, due to the rather poor compatibility
of those bases with arbitrary dilations.

Our second example is the \textbf{shearlet group} 
\[
H_{2}:=\left\{ \varepsilon\begin{pmatrix}a & b\\
0 & a^{1/2}
\end{pmatrix}\with a\in\left(0,\infty\right),b\in\mathbb{R},\varepsilon\in\left\{ \pm1\right\} \right\} .
\]
Here we will show (using the standard definition of coorbit spaces
instead of the characterization via decomposition spaces) that there
is some $\psi\in\mathcal{S}\left(\mathbb{R}^{d}\right)$ that belongs
to the ``conjugated'' coorbit space $\text{Co}\left(L_{g^{-1}vg}^{1}\left(G_{2}\left(g\right)\right)\right)$
with $G_{2}(g):=\mathbb{R}^{2}\rtimes g^{-1}H_{2}g$ and $v:H_{2}\rightarrow\left(0,\infty\right),h\mapsto\left|\det\left(h\right)\right|^{7/6}$,
where $g:=\left(\begin{smallmatrix}0 & -1\\
1 & 0
\end{smallmatrix}\right)$ is the rotation by $90^{\circ}$ in the counter-clockwise direction,
but not to the shearlet coorbit space $\text{Co}\left(L_{v}^{1}\left(G_{2}\right)\right)$
for $G_{2}=\mathbb{R}^{2}\rtimes H_{2}$. We will then see that the
decomposition space point of view provides useful intuition why this
is true.

We first note that the dual orbit of $H_{2}$ is given by $\mathcal{O}_{2}:=H_{2}^{T}\left(\begin{smallmatrix}1\\
0
\end{smallmatrix}\right)=\mathbb{R}^{\ast}\times\mathbb{R}$. In contrast, the dual orbit of the conjugated group $H_{2}\left(g\right):=g^{-1}H_{2}g$
is 
\[
g^{T}H_{2}^{T}g^{-T}\left(\begin{smallmatrix}0\\
1
\end{smallmatrix}\right)=g^{T}H_{2}^{T}\left(\begin{smallmatrix}-1\\
0
\end{smallmatrix}\right)=g^{T}\left(\mathbb{R}^{\ast}\times\mathbb{R}\right)=g^{T}\mathcal{O}_{2}=\mathbb{R}\times\mathbb{R}^{\ast}.
\]
For the construction of $\psi$, choose an arbitrary $\varphi\in\mathcal{D}\left(B_{1}\left(\begin{smallmatrix}0\\
3
\end{smallmatrix}\right)\right)$ with $\varphi\geq0$ and $\varphi\equiv1$ on $B_{1/2}\left(\begin{smallmatrix}0\\
3
\end{smallmatrix}\right)$ and define $\psi:=\mathcal{F}^{-1}\varphi\in\mathcal{S}\left(\mathbb{R}^{d}\right)$
and $\psi_{0}:=\mathcal{F}^{-1}\left(L_{\left(\begin{smallmatrix}3\\
0
\end{smallmatrix}\right)}\varphi\right)\in\mathcal{S}\left(\mathbb{R}^{d}\right)$. This choice ensures that 
\[
\text{supp}\left(\widehat{\psi_{0}}\right)=\text{supp}\left(L_{\left(\begin{smallmatrix}3\\
0
\end{smallmatrix}\right)}\varphi\right)=\text{supp}\left(\varphi\right)+\left(3,0\right)^{T}\subset B_{1}\left(\smash{\left(3,3\right)^{T}}\right)
\]
is a compact subset of $\mathbb{R}^{\ast}\times\mathbb{R}=\mathcal{O}_{2}$.
Thus, Theorem \ref{thm:ZulaessigkeitVonbandbeschraenktenFunktionen}
shows $\psi_{0}\in\mathcal{B}_{w}$, where $w$ is a control weight
(as in Lemma \ref{lem:CoorbitVoraussetzungen}) for $L_{v}^{1}\left(G_{2}\right)$.

For $x\in\mathbb{R}^{d}$ and $h\in\text{GL}\left(\mathbb{R}^{d}\right)$
we can now calculate 
\begin{eqnarray*}
\left(W_{\psi_{0}}\psi\right)\left(x,h\right) & = & \left\langle \psi,\pi\left(x,h\right)\psi_{0}\right\rangle _{L^{2}}\\
 & \overset{\text{Plancherel}}{=} & \left\langle \widehat{\psi},\mathcal{F}\left(\pi\left(x,h\right)\psi_{0}\right)\right\rangle _{L^{2}}\\
 & \overset{\text{Eq. }\eqref{eq:QausiRegulaereDarstellungAufFourierSeite}}{=} & \left|\det\left(h\right)\right|^{1/2}\cdot\left\langle \widehat{\psi},M_{-x}D_{h}\widehat{\psi_{0}}\right\rangle _{L^{2}}\\
 & \overset{\widehat{\psi}=\varphi}{=} & \left|\det\left(h\right)\right|^{1/2}\cdot\int_{\mathbb{R}^{d}}\varphi\left(y\right)\cdot\overline{e^{2\pi i\left\langle -x,y\right\rangle }\left(L_{\left(\begin{smallmatrix}3\\
0
\end{smallmatrix}\right)}\varphi\right)\left(h^{T}y\right)}\,\text{d}y\\
 & \overset{\varphi\geq0}{=} & \left|\det\left(h\right)\right|^{1/2}\cdot\int_{\mathbb{R}^{d}}\varphi\left(y\right)\cdot\varphi\left(h^{T}y-\left(\begin{smallmatrix}3\\
0
\end{smallmatrix}\right)\right)\cdot e^{2\pi i\left\langle x,y\right\rangle }\,\text{d}y.
\end{eqnarray*}
For $\xi\in\mathbb{R}$ with $\left|\xi\right|\leq1/2$, the Lipschitz
continuity (with Lipschitz constant $L=1$) of the cosine implies
\[
\cos\left(\xi\right)\geq\cos\left(0\right)-\left|\cos\left(0\right)-\cos\left(\xi\right)\right|\geq1-\left|\xi\right|\geq\frac{1}{2}.
\]
For $y\in\text{supp}\left(\varphi\right)\subset B_{1}\left(\begin{smallmatrix}0\\
3
\end{smallmatrix}\right)$ we have $\left|y\right|\leq1+\left|\left(\begin{smallmatrix}0\\
3
\end{smallmatrix}\right)\right|=4$. For $x\in\mathbb{R}^{2}$ with $\left|x\right|\leq\frac{1}{16\pi}$
this implies $\left|2\pi\left\langle x,y\right\rangle \right|\leq\frac{1}{2}$.
Using this and the estimate for the cosine above, we arrive at 
\begin{eqnarray}
 &  & \left|\left(W_{\psi_{0}}\psi\right)\left(x,h\right)\right|\nonumber \\
 & \geq & \left|\det\left(h\right)\right|^{1/2}\cdot\text{Re}\left(\int_{\mathbb{R}^{d}}\varphi\left(y\right)\cdot\varphi\left(h^{T}y-\left(\begin{smallmatrix}3\\
0
\end{smallmatrix}\right)\right)\cdot e^{2\pi i\left\langle x,y\right\rangle }\,\text{d}y\right)\nonumber \\
 & \overset{\varphi\geq0}{=} & \left|\det\left(h\right)\right|^{1/2}\cdot\int_{\mathbb{R}^{d}}\varphi\left(y\right)\cdot\varphi\left(h^{T}y-\left(\begin{smallmatrix}3\\
0
\end{smallmatrix}\right)\right)\cdot\cos\left(2\pi\left\langle x,y\right\rangle \right)\,\text{d}y\nonumber \\
 & \geq & \left|\det\left(h\right)\right|^{1/2}\cdot\frac{1}{2}\int_{\mathbb{R}^{d}}\varphi\left(y\right)\cdot\varphi\left(h^{T}y-\left(\begin{smallmatrix}3\\
0
\end{smallmatrix}\right)\right)\,\text{d}y\nonumber \\
 & \overset{\varphi\equiv1\text{ on }B_{1/2}\left(\begin{smallmatrix}0\\
3
\end{smallmatrix}\right)}{\geq} & \left|\det\left(h\right)\right|^{1/2}\cdot\frac{1}{2}\int_{B_{1/2}\left(\begin{smallmatrix}0\\
3
\end{smallmatrix}\right)}\varphi\left(h^{T}y-\left(\begin{smallmatrix}3\\
0
\end{smallmatrix}\right)\right)\,\text{d}y\nonumber \\
 & \overset{\left(\ast\right)}{\geq} & \frac{\left|\det\left(h\right)\right|^{1/2}}{2}\cdot\lambda\left(B_{1/2}\left(\begin{smallmatrix}0\\
3
\end{smallmatrix}\right)\cap h^{-T}\left(B_{1/2}\left(\begin{smallmatrix}3\\
3
\end{smallmatrix}\right)\right)\right)\label{eq:ShearletNichtKonjugationsInvariantWaveletTrafoAbschaetzung}
\end{eqnarray}
for $\left|x\right|\leq\frac{1}{16\pi}$. Here $\lambda$ denotes
the $2$-dimensional Lebesgue measure.

In the step marked with $\left(\ast\right)$, we used $\varphi\equiv1$
on $B_{1/2}\left(\begin{smallmatrix}0\\
3
\end{smallmatrix}\right)$ and the fact that $y\in h^{-T}\left(B_{1/2}\left(\begin{smallmatrix}3\\
3
\end{smallmatrix}\right)\right)$ implies $h^{T}y-\left(\begin{smallmatrix}3\\
0
\end{smallmatrix}\right)\in B_{1/2}\left(\begin{smallmatrix}3\\
3
\end{smallmatrix}\right)-\left(\begin{smallmatrix}3\\
0
\end{smallmatrix}\right)=B_{1/2}\left(\begin{smallmatrix}0\\
3
\end{smallmatrix}\right)$.

Note that we have $M:=\left(0,\frac{1}{4}\right)\times\left(3-\frac{1}{4},3+\frac{1}{4}\right)\subset B_{1/2}\left(\begin{smallmatrix}0\\
3
\end{smallmatrix}\right)$. Let $\left(\begin{smallmatrix}x\\
y
\end{smallmatrix}\right)\in M$ be arbitrary. For $h_{\alpha,\beta}:=\left(\begin{smallmatrix}\alpha & \beta\\
0 & \alpha^{1/2}
\end{smallmatrix}\right)\in H_{2}$ with $\alpha\in\left(0,\infty\right)$ and $\beta\in\mathbb{R}$
we then have 
\begin{eqnarray}
 &  & \begin{pmatrix}x\\
y
\end{pmatrix}\in h_{\alpha,\beta}^{-T}\left(B_{1/2}\begin{pmatrix}3\\
3
\end{pmatrix}\right)\nonumber \\
 & \Leftrightarrow & \begin{pmatrix}\alpha x\\
\beta x+\alpha^{1/2}y
\end{pmatrix}=h_{\alpha,\beta}^{T}\begin{pmatrix}x\\
y
\end{pmatrix}\in B_{1/2}\begin{pmatrix}3\\
3
\end{pmatrix}\nonumber \\
 & \Leftarrow & \left|x\right|\left|\alpha-\frac{3}{x}\right|<\frac{1}{4}\qquad\text{ and }\qquad\left|x\right|\left|\beta-\left(\frac{3}{x}-\frac{\alpha^{1/2}y}{x}\right)\right|<\frac{1}{4}\nonumber \\
 & \overset{x>0}{\Leftrightarrow} & \alpha\in\left(\frac{11}{4x},\frac{13}{4x}\right)\qquad\text{ and }\qquad\beta\in B_{\frac{1}{4x}}\left(\beta_{\alpha,x,y}\right),\label{eq:ShearletNichtKonjugationsInvariantTraegerSchnittAbschaetzung}
\end{eqnarray}
where we used the abbreviation $\beta_{\alpha,x,y}=\frac{3}{x}-\frac{\alpha^{1/2}y}{x}$.

Recall that a (left) Haar integral on the shearlet group $H_{2}$
is given by 
\[
\int_{H_{2}}f\left(h\right)\,\text{d}h=\int_{\mathbb{R}^{\ast}}\int_{\mathbb{R}}f\left(\text{sgn}\left(\alpha\right)h_{\left|\alpha\right|,\beta}\right)\,\text{d}\beta\,\frac{\text{d}\alpha}{\alpha^{2}}~.
\]
Thus, we finally arrive at 
\begin{eqnarray*}
 &  & \left\Vert \psi\right\Vert _{\text{Co}\left(L_{v}^{1}\right)}=\left\Vert W_{\psi_{0}}\psi\right\Vert _{L_{v}^{1}\left(G_{2}\right)}\\
 & = & \int_{G_{2}}\left|\left(W_{\psi_{0}}\psi\right)\left(x,h\right)\right|\cdot v\left(h\right)\,\text{d}\begin{pmatrix}x\\
h
\end{pmatrix}\\
& \overset{\text{Eq. }\eqref{eq:ShearletNichtKonjugationsInvariantWaveletTrafoAbschaetzung}}{\geq} & \int_{0}^{1/16\pi}\int_{H_{2}}\frac{\left|\det\left(h\right)\right|^{1/2}}{2}\cdot\lambda\left(B_{1/2}\left(\begin{smallmatrix}0\\
3
\end{smallmatrix}\right)\cap h^{-T}\left(B_{1/2}\left(\begin{smallmatrix}3\\
3
\end{smallmatrix}\right)\right)\right)\cdot\left|\det\left(h\right)\right|^{7/6}\,\frac{\text{d}h}{\left|\det\left(h\right)\right|}\,\text{d}x\\
 & \geq & C\cdot\int_{0}^{\infty}\int_{\mathbb{R}}\left|\det\left(h_{\alpha,\beta}\right)\right|^{2/3}\cdot\lambda\left(B_{1/2}\left(\begin{smallmatrix}0\\
3
\end{smallmatrix}\right)\cap h_{\alpha,\beta}^{-T}\left(B_{1/2}\left(\begin{smallmatrix}3\\
3
\end{smallmatrix}\right)\right)\right)\,\text{d}\beta\,\frac{\text{d}\alpha}{\alpha^{2}}\\
 & = & C\cdot\int_{0}^{\infty}\int_{\mathbb{R}}\alpha^{-1}\cdot\lambda\left(B_{1/2}\left(\begin{smallmatrix}0\\
3
\end{smallmatrix}\right)\cap h_{\alpha,\beta}^{-T}\left(B_{1/2}\left(\begin{smallmatrix}3\\
3
\end{smallmatrix}\right)\right)\right)\,\text{d}\beta\,\text{d}\alpha\\
 & \overset{M\subset B_{1/2}\left(\begin{smallmatrix}0\\
3
\end{smallmatrix}\right)}{\underset{\text{Fubini}}{\geq}} & C\cdot\int_{0}^{1/4}\int_{3-\frac{1}{4}}^{3+\frac{1}{4}}\int_{0}^{\infty}\int_{\mathbb{R}}\alpha^{-1}\cdot\chi_{h_{\alpha,\beta}^{-T}\left(B_{1/2}\left(\begin{smallmatrix}3\\
3
\end{smallmatrix}\right)\right)}\begin{pmatrix}x\\
y
\end{pmatrix}\,\text{d}\beta\,\text{d}\alpha\,\text{d}y\,\text{d}x\\
 & \overset{\text{Eq. }\eqref{eq:ShearletNichtKonjugationsInvariantTraegerSchnittAbschaetzung}}{\geq} & C\cdot\int_{0}^{1/4}\int_{3-\frac{1}{4}}^{3+\frac{1}{4}}\int_{\left(\frac{11}{4x},\frac{13}{4x}\right)}\alpha^{-1}\int_{B_{\frac{1}{4x}}\left(\beta_{\alpha,x,y}\right)}\,\text{d}\beta\,\text{d}\alpha\,\text{d}y\,\text{d}x\\
 & = & C\cdot\int_{0}^{1/4}\int_{3-\frac{1}{4}}^{3+\frac{1}{4}}\frac{1}{2x}\cdot\int_{\left(\frac{11}{4x},\frac{13}{4x}\right)}\alpha^{-1}\,\text{d}\alpha\,\text{d}y\,\text{d}x\\
 & = & \frac{C}{4}\cdot\int_{0}^{1/4}\frac{1}{x}\cdot\left[\ln\left(\frac{13}{4x}\right)-\ln\left(\frac{11}{4x}\right)\right]\,\text{d}x\\
 & = & \frac{C}{4}\cdot\ln\left(\frac{13}{11}\right)\cdot\int_{0}^{1/4}\frac{1}{x}\,\text{d}x=\infty,
\end{eqnarray*}
i.e. $\psi\notin\text{Co}\left(L_{v}^{1}\left(G_{2}\right)\right)$.

Regarding the question of membership of $\psi$ in the coorbit space
of the conjugate group $G_{2}\left(g\right)=\mathbb{R}^{2}\rtimes g^{-1}H_{2}g$,
we note that $\text{supp}\left(\smash{\widehat{\psi}}\right)\subset B_{1}\left(\begin{smallmatrix}0\\
3
\end{smallmatrix}\right)$ is a compact subset of the dual orbit $\mathcal{O}_{2}\left(g\right)=\mathbb{R}\times\mathbb{R}^{\ast}$
of $g^{-1}H_{2}g$. By Theorem \ref{thm:ZulaessigkeitVonbandbeschraenktenFunktionen}
this shows $\psi\in\mathcal{B}_{w'}$, where $w'$ is a control weight
(as in Lemma \ref{lem:CoorbitVoraussetzungen}) for $L_{g^{-1}vg}^{1}\left(\mathbb{R}^{d} \rtimes g^{-1}H_{2}g\right)$.
Note that the ,,atomic decomposition`` theorem for coorbit spaces
(cf. \cite[Theorem 6.1]{FeichtingerCoorbit1}) implies that the inclusion
\[
\mathcal{B}_{w'}\subset\text{Co}\left(L_{g^{-1}vg}^{1}\left(G_{2}\left(g\right)\right)\right)
\]
is valid. All in all, this proves that $\psi\in\mathcal{S}\left(\mathbb{R}^{d}\right)\subset L^{2}\left(\mathbb{R}^{d}\right)$
satisfies 
\[
\psi\in\text{Co}\left(L_{g^{-1}vg}^{1}\left(G_{2}\left(g\right)\right)\right)\setminus\text{Co}\left(L_{v}^{1}\left(G_{2}\right)\right).
\]
This shows that the coorbit spaces of the shearlet group are (in general)
not invariant under dilation. It should be noted that the above reasoning
could be adapted to other transformations as well; we claim that the
only orthogonal transformations under which all shearlet coorbit spaces
are invariant are the reflections $(x_{1},x_{2})^{T}\mapsto(-x_{1},x_{2})^{T}$,
$\left(x_{1},x_{2}\right)^{T}\mapsto\left(x_{1},-x_{2}\right)^{T}$
and the rotation $\left(x_{1},x_{2}\right)^{T}\mapsto\left(-x_{1},-x_{2}\right)^{T}$.

Intuitive reasons for this phenomenon (and for the choice of $\psi$)
that are suggested by the decomposition space point of view are the
following: 
\begin{enumerate}
\item $\widehat{\psi}$ does \emph{not} vanish on the ,,blind spot`` $\mathbb{R}^{2}\setminus\mathcal{O}_{2}=\left\{ 0\right\} \times\mathbb{R}$
of the shearlet group. 
\item Choose 
\[
v':H_{2}\rightarrow\left(0,\infty\right),h\mapsto\left|\det\left(h^{-1}\right)\right|^{\frac{1}{2}-\frac{1}{q}}\cdot v\left(h^{-1}\right)=\left|\det\left(h\right)\right|^{1-\frac{1}{2}-\frac{7}{6}}=\left|\det\left(h\right)\right|^{-\frac{2}{3}}
\]
as in Theorem \ref{thm:FourierStetigVonCoorbitNachDecomposition}
and let $\xi_{0}:=\left(\begin{smallmatrix}1\\
0
\end{smallmatrix}\right)\in\mathcal{O}_{2}$.

For $\left(\begin{smallmatrix}x\\
y
\end{smallmatrix}\right)\in\mathcal{O}_{2}$ we then have $h_{\left(x,y\right)}^{T}\xi_{0}=\left(\begin{smallmatrix}x\\
y
\end{smallmatrix}\right)$ for 
\[
h_{\left(x,y\right)}:=\text{sgn}\left(x\right)\cdot\begin{pmatrix}\left|x\right| & \text{sgn}\left(x\right)\cdot y\\
0 & \left|x\right|^{1/2}
\end{pmatrix}\in H_{2}.
\]
This shows that 
\[
u\left(\begin{smallmatrix}x\\
y
\end{smallmatrix}\right):=v'\left(h_{\left(x,y\right)}\right)=\left|\det\left(h_{\left(x,y\right)}\right)\right|^{-\frac{2}{3}}=\left|x\right|^{-1}
\]
is a valid transplant of $v'$ onto $\mathcal{O}_{2}$.

But this (and thus every) transplant of $v'$ onto $\mathcal{O}_{2}$
blows up near the ,,blind spot`` $\mathbb{R}^{2}\setminus\mathcal{O}_{2}=\left\{ 0\right\} \times\mathbb{R}$.

\end{enumerate}
Using these observations one can show $\widehat{\psi}\notin\mathcal{D}\left(\mathcal{Q}_{2},L^{1},\ell_{u}^{1}\right)$
for a suitable decomposition covering $\mathcal{Q}_{2}$ of $\mathcal{O}_{2}$
induced by $H_{2}$, which implies $\psi\notin\text{Co}\left(L_{v}^{1}\left(G_{2}\right)\right)$
by Theorem \ref{thm:FourierStetigVonCoorbitNachDecomposition}.

\section{Outlook}

While the discussion in the previous section is in parts somewhat ad hoc 
and restricted, we believe that it nicely illustrates the use that can be
made of the decomposition space formalism. The systematic study of
embeddings between decomposition spaces, and their application to
the study of wavelet coorbit spaces, is the subject of ongoing research,
and will be treated in more detail in upcoming publications. Another
direction of research that is currently pursued concerns the extension
of the results to include quasi-Banach spaces, in particular with
the aim of treating spaces of the type ${\rm Co}(L_{v}^{p,q})$ with
$p$ and/or $q$ in $(0,1)$.

\section*{Acknowledgements}

We thank Karlheinz Gröchenig and Hans Feichtinger for interesting
discussions and comments. This research was funded by the Excellence
Initiative of the German federal and state governments, and by the German Research Foundation (DFG),
under the contract FU 402/5-1.

\bibliographystyle{plain}
\bibliography{BibliographyV2}

\end{document}